\documentclass[11pt]{article}
\usepackage{latexsym}
\usepackage{amssymb}
\usepackage{amsthm}
\usepackage{amscd}
\usepackage{amsmath}
\usepackage{mathrsfs}
\usepackage{graphicx}
\usepackage{shuffle}
\usepackage{hyperref}
\usepackage{mathtools}
\usepackage[bbgreekl]{mathbbol}
\usepackage{mathdots}
\usepackage{ytableau}
\usepackage[enableskew,vcentermath]{youngtab}
\usepackage{tikz-cd}

\usepackage[all]{xy}
\input xy \xyoption{frame}
\xyoption{dvips}

\usepackage[colorinlistoftodos]{todonotes}
\usetikzlibrary{chains,scopes}
\usetikzlibrary{shapes.geometric,positioning}

\DeclareSymbolFontAlphabet{\mathbb}{AMSb}
\DeclareSymbolFontAlphabet{\mathbbl}{bbold}

\numberwithin{equation}{section}
\allowdisplaybreaks[1]
\UseCrayolaColors

\theoremstyle{definition}
\newtheorem* {theorem*}{Theorem}
\newtheorem* {corollary*}{Corollary}
\newtheorem* {conjecture*}{Conjecture}
\newtheorem{theorem}{Theorem}[section]
\newtheorem{thmdef}[theorem]{Theorem-Definition}

\newtheorem{problem}[theorem]{Problem}
\theoremstyle{definition}

\newtheorem* {example*}{Example}

\newtheorem{lemma}[theorem]{Lemma}
\theoremstyle{definition}
\newtheorem{definition}[theorem]{Definition}
\theoremstyle{definition}

\newtheorem{conjecture}[theorem]{Conjecture}
\newtheorem{proposition}[theorem]{Proposition}
\newtheorem{corollary}[theorem]{Corollary}

\newtheorem{remark}[theorem]{Remark}
\theoremstyle{definition}
\newtheorem {example}[theorem]{Example}
\theoremstyle{definition}

\theoremstyle{definition}

\theoremstyle{definition}

\xyoption{dvips}

\def\modu{\ (\mathrm{mod}\ }

\def\({\left(}
\def\){\right)}

\newcommand{\CC}{\mathbb{C}}

\newcommand{\cR}{\mathcal{R}}

\newcommand{\cC}{\mathcal{C}}

\newcommand{\cZ}{\mathcal{Z}}
\def\cX{\mathcal{X}}
\def\cY{\mathcal{Y}}

\def\CC{\mathbb{C}}

\def\ZZ{\mathbb{Z}}

\def\GL{\textsf{GL}}

\def\ch{\mathrm{ch}}

\def\fk{\mathfrak}

\def\barr{\begin{array}}
\def\earr{\end{array}}
\def\ba{\begin{aligned}}
\def\ea{\end{aligned}}
\def\be{\begin{equation}}
\def\ee{\end{equation}}

\def\quand{\quad\text{and}\quad}
\def\qquord{\qquad\text{or}\qquad}

\newcommand{\gl}{\mathfrak{gl}}

\def\hs{\hspace{0.5mm}}

\def\ds{\displaystyle}

\def\ben{\begin{enumerate}}
\def\een{\end{enumerate}}

\def\hs{\hspace{0.5mm}}

\def\fpf{{\textsf {FPF}}}

\def\D{\hat D}

\def\Des{\mathrm{Des}}

\def\Ffpf{\hat F^\fpf}

\def\Ifpf{I^{\fpf}}

\def\e{\textbf{e}}

\newcommand{\Fl}{\textsf{Fl}}

\renewcommand{\O}{\operatorname{O}}
\newcommand{\Sp}{\textsf{Sp}}

\newcommand{\cA}{\mathcal{A}^\O}
\newcommand{\cB}{\mathcal{B}}

\def\iR{\cR^\O}
\def\iRfpf{\cR^{\Sp}}

\def\cAfpf{\mathcal{A}^\Sp}

\def\iF{\hat F}
\def\Ffpf{\hat F_\fpf}

\def\arcstart{\ \xy<0cm,-.15cm>\xymatrix@R=.1cm@C=.3cm }
\newcommand{\arcstartc}[1]{\ \xy<0cm,-.15cm>\xymatrix@R=.1cm@C=#1cm}

\def\ellhat{\ell^{\textsf{O}}}
\def\ellfpf{\ell^{\textsf{Sp}}}

\def\Ffpf{\hat F^\fpf}

\newcommand{\del}{\operatorname{del}}

\newcommand{\ipush}{\textsf{ipush}}

\newcommand{\fpush}{\textsf{fpush}}

\def\sA{\mathscr{A}}

\def\q{\mathfrak{q}}

\def\Tab{\textsf{Tab}_{EG}}

\def\Q{\mathrm{Q}}

\def\BB{\mathbb{B}}

\newcommand{\row}{\operatorname{row}}

\def\fsim{\mathbin{=_\textsf{Sp}}}

\def\isim{\mathbin{=_\textsf{O}}}

\def\simCK{\overset{\textsf{K}}\sim}
\def\simICK{\overset{\textsf{O}}\sim}

\def\simFCK{\overset{\textsf{Sp}}\sim}

\newcommand{\weight}{\operatorname{wt}}

\def\cW{\mathcal{W}}

\def\row{\textsf{row}}
\def\col{\textsf{col}}

\def\whSym
{\mathfrak{m}\textsf{Sym}}

\def\whQSym
{\mathfrak{m}\textsf{QSym}}

\def\D{\textsf{D}}
\def\SD{\textsf{SD}}

\def\PP{\ZZ_{>0}}
\def\NN{\ZZ_{\geq 0}}

\def\EG{\textsf{EG}}

\def\HM{\textsf{HM}}

\def\Tab{\textsf{Tab}}
\def\STab{\textsf{ShTab}}
\def\O{\textsf{O}}

\def\PO{P_{\textsf{EG}}^\O}
\def\QO{Q_{\textsf{EG}}^\O}

\def\PSp{P_{\textsf{EG}}^\Sp}
\def\QSp{Q_{\textsf{EG}}^\Sp}

\def\fkb{\fk b}
\def\ck{\textsf{ck}}
\def\ifkb{\fk{ib}}
\def\ffkb{\fk{fb}}

\newcommand{\push}{\textsf{push}}

\newcommand{\ytab}[1]{
\ytableausetup{boxsize = .38cm,aligntableaux=center}
{\footnotesize\begin{ytableau}  #1  \end{ytableau}}
}

\def\fks{{\fk d}}
\def\shword{\textsf{shword}}

\def\minfpf{1_\fpf}

\def\pair{\textsf{pair}}

\def\fO{f^\O}
\def\eO{e^\O}

\def\fSp{f^\Sp}
\def\eSp{e^\Sp}

\def\pipe{ \hs / \hs }

\def\ttimes{\textcircled{$\times$}}

\def\cS{\textsf{Perm}}
\def\cSe{\textsf{Even}}

\def\dbl{\textsf{dbl}}
\def\invert{\textsf{inv}}

\def\K{\textsf{K}}

\def\ROWINSERT{\mathsf{row\_insert}}
\def\REVERSE{\mathsf{reverse}}
\def\COLUMNINSERT{\mathsf{column\_insert}}
\def\REORIENT{\mathsf{reorient}}
\def\INSERTUPTO{\mathsf{insert\_up\_to\_column}}
\def\TOTAL{\mathsf{total\_insertion}}

\def\spck{\mathsf{ck}_0^\Sp}
\def\ock{\mathsf{ck}_0^\O}

\usepackage{fullpage}
\usepackage[nomarkers,figuresonly,nofiglist]{endfloat}

\begin{document}

\title{Bumping operators and insertion algorithms for queer supercrystals}

\author{
Eric Marberg
\\ Department of Mathematics \\  Hong Kong University of Science and Technology \\ {\tt eric.marberg@gmail.com}
}

\date{}

\maketitle

\begin{abstract}
Results of Morse and Schilling show that the set of increasing factorizations of reduced words for a permutation is naturally a crystal for the general linear Lie algebra. Hiroshima has recently constructed two superalgebra analogues of such crystals. Specifically, Hiroshima has shown that the sets of increasing factorizations of involution words and fpf-involution words for a self-inverse permutation are each crystals for the queer Lie superalgebra. In this paper, we prove that these crystals are normal and identify their connected components. To accomplish this, we study two insertion algorithms that may be viewed as shifted analogues of the Edelman-Greene correspondence. We prove that the connected components of Hiroshima's crystals are the subsets of factorizations with the same insertion tableau for these algorithms, and that passing to the recording tableau defines a crystal morphism. This confirms a conjecture of Hiroshima. Our methods involve a detailed investigation of certain analogues of the Little map, through which we extend several results of Hamaker and Young.
\end{abstract}


\setcounter{tocdepth}{2}
\tableofcontents

\section{Introduction}

This article is about combinatorial models for crystals for quantum  queer Lie superalgebras.
The crystals of primary interest will arise as sets of factorizations of reduced words for permutations.
Our main results show how certain insertion algorithms that map words to pairs of shifted tableaux
classify the connected components of these crystals and may be interpreted as crystal isomorphisms.

\emph{Crystal bases} or \emph{(Kashiwara) crystals} are combinatorial objects that arise 
in the representation theory of Lie algebras (or, more precisely, 
of quantum deformations of the corresponding universal enveloping algebras). The theory of crystals 
first appeared in independent work of Kashiwara \cite{Kashiwara1990,Kashiwara1991} and Lusztig \cite{Lusztig1990a,Lusztig1990b} in the 1990s; for a history of the relevant literature, see  
\cite[\S1]{BumpSchilling}.

A theory of abstract crystals exists for any finite-dimensional Lie (super)algebra.
We confine our attention in this paper to abstract crystals for 
the \emph{general linear Lie algebra} $\gl_n$ and the \emph{queer Lie superalgebra} $\q_n$.
The latter object is the second super-analogue of $\gl_n$, and has a number of interesting features.
For example, all of its Cartan subalgebras are noncommutative, 
and this gives the highest weight space of any highest weight $\q_n$-module the structure of a Clifford algebra.
For background on $\q_n$ and its somewhat complicated representation theory, see \cite{GJKKK15,GJKKK,GJKK10,GJKKK10}.

The data of an abstract $\gl_n$-crystal or $\q_n$-crystal
 is equivalent to a certain directed weighted \emph{crystal graph}. Under this identification, crystal isomorphisms correspond to graph isomorphisms.
The weakly connected components of a crystal graph are called
its \emph{full subcrystals}.
The categories of $\gl_n$- and $\q_n$-crystals are both equipped with a tensor product and a \emph{standard crystal} $\BB_n$, 
which is derived from the vector representation of the associated Lie superalgebra.

It is an interesting problem to determine whether or not an abstract crystal
is \emph{normal} in the sense of being isomorphic to a disjoint union of full subcrystals of tensor powers of $\BB_n$,
since such abstract crystals will correspond to actual representations of the associated Lie superalgebra.
Stembridge \cite{Stembridge2003} identified a set of local axioms that 
give a solution to this problem for $\gl_n$-crystals.
Gillespie, Hawkes, Poh, and Schilling \cite{GHPS}, building on work of Assaf and Oguz \cite{AssafOguz}, have recently extended Stembridge's results to
$\q_n$-crystals.

The prototypical example of a normal $\gl_n$-crystal is the set of \emph{semistandard Young tableaux} 
with entries in $\{1,2,\dots,n\}$. The $\q_n$-analogues of these crystals have two combinatorial models,
either as \emph{semistandard decomposition tableaux} \cite{GJKKK} or \emph{semistandard shifted tableaux} \cite{AssafOguz, HPS, Hiroshima2018}. We focus primarily on the second model in this article.
For the definitions, see Section~\ref{shtab-sect}.

Another source of crystal constructions comes from reduced words.
A \emph{reduced word} for a permutation $\pi$ is a minimal-length integer sequence $i_1i_2\cdots i_l$
such that $\pi = s_{i_1}s_{i_2}\cdots s_{i_l}$ where $s_i:=(i,i+1)$.
One can divide a word $i_1i_2\cdots i_l$ into a sequence of $n$ strictly increasing, possibly empty subwords (which we call an \emph{$n$-fold increasing factorization}) if and only if there are fewer than $n$ indices $j$ with $i_j \geq i_{j+1}$. 
Let $\cR_n(\pi)$ denote the set of all $n$-fold increasing factorizations 
of reduced words for a fixed permutation $\pi$.
Morse and Schilling \cite{MorseSchilling}
identified a natural $\gl_n$-crystal structure on this set (see Section~\ref{reduced-sect})
and by checking Stembridge's local axioms, proved the following:

\begin{theorem*}[Morse and Schilling \cite{MorseSchilling}; see Corollary~\ref{ms-cor1}]
The $\gl_n$-crystal $\cR_n(\pi)$ is normal.
\end{theorem*}

Morse and Schilling also showed something more specific.
The \emph{Edelman-Greene correspondence} \cite{EG} is a well-known map 
that sends each $w \in \cR_n(\pi)$ to a pair of  tableaux $(P_\EG(w),Q_\EG(w))$
with the same shape. This correspondence has many interesting and desirable properties. 

\begin{theorem*}[Morse and Schilling \cite{MorseSchilling}; see Theorem~\ref{eg-thm}]
The full subcrystals of $\cR_n(\pi)$ are the fibers of $w \mapsto P_\EG(w)$.
Moreover, the map $w \mapsto Q_\EG(w)$ defines an isomorphism from each full subcrystal of $\cR_n(\pi)$ to a (normal) $\gl_n$-crystal of semistandard tableaux.
\end{theorem*}

One application of this result is to give a crystal theoretic interpretation of the positive coefficients in the Schur expansion of the \emph{Stanley symmetric functions}; see Corollary~\ref{ms-cor2}.
The precise definitions of  $\cR_n(\pi)$ and the Edelman-Greene correspondence appear in Section~\ref{factor-sect}.

In the recent paper \cite{Hiroshima},
Hiroshima has constructed two $\q_n$-analogues of Morse and Schilling's $\gl_n$-crystals of factorized reduced words.
The elements of 
Hiroshima's $\q_n$-crystals
 are $n$-fold increasing factorizations of the \emph{involution words} and \emph{fpf-involution words} 
 of a self-inverse permutation $\pi$. Such words have been studied under various names by several authors \cite{CJW,HMP2,HanssonHultman,HuZhang1,RichSpring}.

Whereas  reduced words for permutations may be identified with maximal chains in the weak order on the symmetric group,
involution words and fpf-involution words correspond to maximal chains in an analogous weak order on
the finite set of orbits of the orthogonal and symplectic groups acting on the complete flag variety.
For this reason, we denote Hiroshima's $\q_n$-crystals by $\iR_n(\pi)$ and $\iRfpf_n(\pi)$.
We review the definitions of these crystals in Sections~\ref{o-fact-sect} and \ref{sp-fact-sect}.
One of our main new results is the following theorem concerning these objects.

\begin{theorem*}[Corollaries~\ref{normal-cor1} and \ref{normal-cor2}]
For each $\K \in \{ \O,\Sp\}$,
the $\q_n$-crystal $\cR^\K_n(\pi)$ is normal.
\end{theorem*}

This result is a corollary of a more effective theorem, which we sketch as follows.
It turns out that one obtains natural
``orthogonal'' and ``symplectic'' analogues of the Edelman-Greene correspondence
by restricting the \emph{shifted Hecke insertion} and \emph{symplectic Hecke insertion} algorithms
introduced in \cite{Marberg2019a, PatPyl2014}.
We denote these maps by $w \mapsto (\PO(w), \QO(w))$ and $w \mapsto (\PSp(w), \QSp(w))$.
Both are generalizations of \emph{Sagan-Worley insertion} \cite{Sag87,Worley},
and assign increasing factorizations to pairs of shifted tableaux with the same shape;
the definitions are given in Section~\ref{ins-sect}.

The following theorem
provides more substance to the strong formal analogy between Morse and Schilling's $\gl_n$-crystals $\cR_n(\pi)$
and Hiroshima's $\q_n$-crystals $\iR_n(\pi)$ and $\iRfpf_n(\pi)$.

\begin{theorem*} [Theorems~\ref{main-thm1} and \ref{main-thm2}]
Let $\K \in \{ \O,\Sp\}$.
The full subcrystals of $\cR^\K_n(\pi)$ are the fibers of $w \mapsto P_\EG^\K(w)$,
and the map $w \mapsto Q_\EG^\K(w)$ defines an isomorphism from each full subcrystal of $\cR^\K_n(\pi)$ to a (normal) $\q_n$-crystal of semistandard shifted tableaux.
\end{theorem*}

An application of this result is to give a crystal theoretic interpretation of the positive coefficients in the Schur $P$-expansion of the \emph{involution Stanley symmetric functions} studied in \cite{HMP4,HMP5}; see Corollaries~\ref{invF-cor1} and \ref{invF-cor2}.
As we work to show this result, we will end up proving a few conjectures from \cite{HMP4,Hiroshima,Marberg2019a}.
In particular, the claim that $w \mapsto Q_\EG^\Sp(w)$ is a crystal morphism is equivalent to \cite[Conjecture 5.1]{Hiroshima}, as we explain
in Section~\ref{main-sect}.
As another application, we use our results in Section~\ref{dual-sect} to recover a theorem of Assaf \cite{Assaf14} that constructs a \emph{dual equivalence} on standard shifted tableaux.

Our strategy to prove these results is to identify certain crystal isomorphisms between different
instances of $\cR^\K_n(\pi)$,
which commute with the map $ w\mapsto Q_\EG^\K(w)$. We use these isomorphisms to translate our crystals to a simpler form. 
The relevant maps will be composed of the \emph{involution Little bumping operators} introduced in \cite{HMP3}.
Another way to frame the main results of this article is as an in-depth study of these operators.
Ignoring all applications to crystals, our work serves to generalize several theorems of Hamaker and Young \cite{HamakerYoung}
and to clarify the relationship between involution Little bumps and shifted forms of the Edelman-Greene correspondence.

We organize the rest of this paper as follows.
First, there is a section of preliminaries on crystals, words, and tableaux. Section~\ref{factor-sect} reviews the definitions of the crystals $\cR_n(\pi)$, $\iR_n(\pi)$, and $\iRfpf_n(\pi)$, along with Edelman-Greene insertion
and its orthogonal and symplectic variants. At the end of Section~\ref{factor-sect}, we give the precise statements of our main results. 
Section~\ref{little-sect} includes our investigation of various bumping operators.
Section~\ref{proofs-sect} contains our proofs of the theorems sketched
in this introduction, along with a discussion of applications and open problems.
Finally, a mostly expository appendix provides some extra background on 
crystals of shifted tableaux.

\subsection*{Acknowledgements}

This work was partially supported by Hong Kong RGC Grant ECS 26305218.
I am grateful to
Sami Assaf and Brendan Pawlowski
for useful conversations
over the course of several productive visits to the University of Southern California.
I am also thankful
to
Dan Bump,
Maria Gillespie,
Jake Levinson,
Oliver Pechenik,
and Travis Scrimshaw for helpful discussions about crystals,
and to
 Zach Hamaker and Jake Potter for 
answering questions about Little bumping operators.

\section{Preliminaries}\label{prelim-sect}

In this section, we review some
general background on crystals, words, and tableaux from
\cite{BumpSchilling,GJKKK}.
Let  $\ZZ$ be the set of integers and write $\NN$ and $\PP$ for the sets of nonnegative and positive integers.
For each $m \in \NN$, let $[m] = \{ i \in \PP : i \leq m\} =\{1,2,\dots,m\}$, 
so that $[0] = \varnothing$.

\subsection{Crystals}\label{crystal-sect}

Fix a positive integer $n$ and
let $\e_1,\e_2,\dots,\e_n$ be the standard unit basis vectors in $\ZZ^n$.

\begin{definition}[{\cite[\S2.2]{BumpSchilling}}]
\label{gl-def}
An \emph{abstract $\gl_n$-crystal} is a set $\cB$ with maps
$\weight : \cB\to \NN^n$
and $ e_i,f_i : \cB \to \cB \sqcup \{0\}$ for $i \in [n-1]$,
where $0 \notin \cB$ is an auxiliary element, such that if $i \in [n-1]$ then:
\ben
\item[(1)] If $b,c \in \cB$ then $e_i(b) = c$ if and only if $f_i(c) = b$, in which case
$\weight(c) = \weight(b) + \e_{i} -\e_{i+1}.$

\item[(2)] If  $b \in \cB$  then
$
\varepsilon_i(b) := \max\left\{ k\geq 0 : e_i^k(b) \neq 0\right\}
$
and
$
\varphi_i(b) := \max\left\{ k \geq 0: f_i^k(b) \neq 0\right\}
$
are both finite, and $\varphi_i(b) - \varepsilon_i(b) = \weight(b)_i - \weight(b)_{i+1}$.
\een
We refer to the function $\weight$ as the \emph{weight map}, to $e_i$ and $f_i$ as the \emph{raising} and \emph{lowering crystal operators},
and to $\varepsilon_i$ and $\varphi_i$ as the \emph{string lengths}  of $\cB$.
\end{definition}

This is slightly more specialized than the definition of a $\gl_n$-crystal in \cite{BumpSchilling}. 
In the terminology of \cite{BumpSchilling},
our definition describes the $\gl_n$-crystals 
that are \emph{seminormal}.



\begin{definition}[{\cite[\S1.3]{GJKKK}}]
\label{q-def}
An \emph{abstract $\q_n$-crystal} (for $n\geq 2$) is an abstract $\gl_n$-crystal $\cB$ 
with \emph{queer raising} and \emph{lowering operators} $e_{\overline 1},f_{\overline 1} : \cB \to \cB\sqcup\{0\}$
satisfying the following conditions:
\ben
\item[(1)] If $b,c \in \cB$ then $e_{\overline 1}(b)=c$ if and only if $f_{\overline 1}(c) = b$, in which case 
\[\weight(b) = \weight(c)+\e_2 - \e_1,
\quad
\varepsilon_i(b)=\varepsilon_i(c),\quand
\varphi_i(b) = \varphi_i(c)\]
for all $3 \leq i \leq n-1$,
where $\varepsilon_i$, $\varphi_i$ are
the string lengths from Definition~\ref{gl-def}.

\item[(2)] The operators $e_{\overline 1}$ and $f_{\overline 1}$ commute with $e_i$ and $f_i$
for each $3 \leq i \leq n-1$,
under the convention that
$e_{\overline 1}(0) = f_{\overline 1}(0) = e_i(0) = f_i(0)=0.$

\item[(3)] If $b \in \cB$ and  we define
$
\varepsilon_{\overline 1}(b) := \max\bigl\{ k \geq 0 : e_{\overline 1}^k(b) \neq 0\bigr\}
$
and
$\varphi_{\overline 1}(b) := \max\bigl\{ k \geq 0: f_{\overline 1}^k(b) \neq 0\bigr\},
$
then we have $\varepsilon_{\overline 1}(b) + \varphi_{\overline 1}(b) \leq 1$, with equality if 
$\weight(b)_1\neq 0$ or $\weight(b)_2\neq 0$.
\een
\end{definition}

The original definition of an \emph{abstract $\q_n$-crystal} in \cite[\S1.3]{GJKKK}
omits condition (3).
 This condition holds in all examples of interest and will imply a desirable symmetry property. 
To simplify some later statements, we 
consider the empty set to be an abstract $\gl_n$- and $\q_n$-crystal and 
define an abstract $\q_1$-crystal to be any set with a weight map $\weight$ taking values in $\NN$.

The \emph{crystal graph} of an abstract $\q_n$-crystal $\cB$ is the weighted directed graph with vertex set $\cB$
that has an edge $x \xrightarrow{i} y$ whenever $y = f_i(x)$ for some $i\in \{\overline 1,1,2,\dots,n-1\}$.
A weakly connected component of the crystal graph of an abstract $\gl_n$- or $\q_n$-crystal
is called a \emph{full subcrystal}.

\begin{example}\label{bb-ex2}
The \emph{standard $\q_n$-crystal} $\BB_n$ has
weight function $\weight(\boxed{i})=\e_i$ and crystal graph
  \begin{center}
    \begin{tikzpicture}[xscale=1.6, yscale=1,>=latex]]
      \node at (0,0) (T0) {$\boxed{1}$};
      \node at (1,0) (T1) {$\boxed{2}$};
      \node at (2,0) (T2) {$\boxed{3}$};
      \node at (3,0) (T3) {${\cdots}$};
      \node at (4,0) (T4) {$\boxed{n}$};
      \draw[->,thick]  (T0.15) -- (T1.165) node[midway,above,scale=0.75] {$\overline 1$};
      \draw[->,thick]  (T0.345) -- (T1.195) node[midway,below,scale=0.75] {$1$};
      \draw[->,thick]  (T1) -- (T2) node[midway,above,scale=0.75] {$2$};
      \draw[->,thick]  (T2) -- (T3) node[midway,above,scale=0.75] {$3$};
      \draw[->,thick]  (T3) -- (T4) node[midway,above,scale=0.75] {$n-1$};
     \end{tikzpicture}
  \end{center}
\end{example}

The \emph{character} of a finite $\gl_n$-crystal $\cB$ is the polynomial $\ch(\cB) = \sum_{b \in \cB} x_1^{\weight(b)_1}x_2^{\weight(b)_2}\cdots x_n^{\weight(b)_n}$.
Let $\Lambda_n$ be the ring of symmetric polynomials in $\ZZ[x_1,x_2,\dots,x_n]$.

\begin{proposition}[{\cite[\S2.6]{BumpSchilling}}]\label{lambda-prop}
The character of a finite $\gl_n$-crystal is in $\Lambda_n$.
\end{proposition}

An element $f \in \Lambda_n$ 
is \emph{supersymmetric} if 
$f(x_1,-x_1,x_3,\dots,x_n)\in \ZZ[x_3,\dots,x_n].$
Let $\Gamma_n$ denote the subring of supersymmetric polynomials in $\Lambda_n$ for $n\geq 2$,
and set $\Gamma_1=\Lambda_1$.

\begin{proposition}\label{gamma-prop}
The character of a finite $\q_n$-crystal is in $\Gamma_n$.
\end{proposition}

\begin{proof}
Fix $n\geq 2$.
Let $R$ be the set of $f \in \ZZ[x_1,\dots,x_n]$
with  $f(x_1,-x_1,x_3,\dots,x_n) \in \ZZ[x_3,\dots,x_n]$.
Suppose $\cB$ is a finite $\q_n$-crystal and $b \in \cB$. If $\weight(b)_1 = \weight(b)_2 = 0$ then $x^{\weight(b)} \in R_n$.
Otherwise, Definition~\ref{q-def}(3)
implies that there exists a unique $c \in \cB$
with $e_{\overline 1}(b) = c$ or $f_{\overline 1}(b) = c$,
and $x^{\weight(b)} + x^{\weight(c)} \in (x_1+x_2) \ZZ[x_1,x_2,\dots,x_n] \subset R_n$.
We conclude that $\ch(\cB) \in 
R \cap \Lambda_n=\Gamma_n$.
\end{proof}

A \emph{(strict) morphism} $\cB \to \cC$ of abstract $\gl_n$- or $\q_n$-crystals
is a map $\cB\sqcup\{0\} \to \cC\sqcup\{0\}$ with $0\mapsto 0$ 
that preserves weights and string lengths
and commutes with all crystal operators. 
A morphism that is also a bijection is an \emph{isomorphism};
such a map induces an isomorphism of 
crystal graphs.

\subsection{Words}\label{words-sect}
 
 Given two abstract $\gl_n$- or $\q_n$-crystals $\cB$ and $\cC$,
 one can form the \emph{tensor product crystal} $\cB \otimes \cC$;
 see \cite[\S.3]{BumpSchilling} and \cite[\S1.3]{GJKKK}
 for the precise definitions.
For our applications, it will suffice to describe 
the $m$-fold tensor product $(\BB_n)^{\otimes m}$ 
of the standard $\q_n$-crystal from Example~\ref{bb-ex2}.
One can realize this object as the following crystal of {words}.
 
A \emph{word} is a finite sequence $w_1w_2\cdots w_m$ of integers. 
Fix $m \in \PP$ and let $\cW_{n}(m)$ be the set of $m$-letter words in the alphabet $[n]=\{1,2,\dots,n\}$.
%
Given $w \in \cW_n(m)$,
define $\weight(w)  \in (\NN)^n$ to be the $n$-tuple
whose $i$th entry is the number of occurrences of $i$ in $w$.
For any $i \in \ZZ$,
there are operators $f_i$ and $e_i$ acting on $\cW_n(m)$ as follows.
 Consider the sequence formed by replacing each $i$  in a word $w$ by a right 
parenthesis
and each $i+1$ in $w$ by a left parenthesis.
\begin{itemize}
\item 
If all right parentheses in this sequence 
belong to a balanced pair of left and right parentheses,
then $f_i(w) = 0$.
Otherwise, form $f_i(w)$ from $w$ by
changing the letter $i$ corresponding to the last unbalanced 
right parenthesis to $i+1$.

\item Similarly, if all left parentheses 
belong to a balanced pair,
then $e_i(w) = 0$.
Otherwise,
form $e_i(w)$ by changing the $i+1$ in $w$ corresponding to the
first unbalanced  left parenthesis to $i$.
\end{itemize}
For example, if $w= 1223313212$ and $i=2$ then the parenthesized word is $1))((1()1)$,
so
$f_2(w) = 12 \underline{3}3313212$
and
$e_2(w) = 122 \underline{2}313212.$
We also 
define
$f_{\overline 1}(w)$ and $e_{\overline 1}(w)$ for words $w \in \cW_n(m)$:
\begin{itemize}
\item If $w$ has no $1$'s or if its first $1$ appears after its first $2$,
then $f_{\overline 1}(w) =0$.
 Otherwise, $f_{\overline 1}(w)$ is the word formed by changing the first $1$ in $w$ to $2$.

\item If $w$ has no $2$'s or if its first $2$ appears after its first $1$,
then $e_{\overline 1}(w) =0$. 
Otherwise, $e_{\overline 1}(w)$ is the word formed by changing the first $2$ in $w$ to  $1$.
\end{itemize}
If $w= 1223313212$ then $f_{\overline 1}(w) = \underline{2}223313212$ and $e_{\overline 1}(w) =0$.

\begin{proposition}[{\cite[Remarks 2.3 and 2.4]{GHPS}}]
Relative to the maps $\weight$, $e_i$, $f_i$ just given,
$\cW_n(m)$ is an abstract $\q_n$-crystal
and there is a $\q_n$-crystal isomorphism $\cW_n(m) \cong (\BB_n)^{\otimes m} $.
\end{proposition}


 \subsection{Tableaux}

The Young diagram of an integer partition $\lambda = (\lambda_1  \geq \lambda_2\geq \dots \geq \lambda_n > 0)$ is the set of pairs 
$\D_\lambda = \{ (i,j) \in [n]\times [\lambda_1] : 1 \leq j \leq \lambda_i\}.$
We use the term \emph{tableau}
to mean a map  $\D_\lambda \to \ZZ$ for some partition $\lambda$;
 such a map is said to have shape $\lambda$.

We draw tableaux in French notation, so that row indices increase going up. 
For example,
\be\label{tab-ex}
\ytab{ 3 & 4 & \none  \\ 
 2 & 2 & 4}
 \quand
  \ytab{
 3 & 4 & \none  \\ 
 2 & 3 & 4}
 \quand
\ytab{ 3 & 5 & \none  \\ 
 1 & 2 & 4}
\ee
all have shape $\lambda=(3,2)$.
The pairs in the domain of a tableau are its \emph{boxes}.

A tableau is \emph{semistandard} if its rows are weakly increasing and its columns are strictly increasing.
A tableau is \emph{increasing} if its rows and columns are both strictly increasing.
 A tableau with $m$ boxes is \emph{standard} if 
it is increasing and
it contains each of the numbers $1,2,\dots,m$ exactly once.
The three tableaux drawn above are respectively semistandard, increasing, and standard.
Let $\Tab_n(m)$ denote the set of semistandard tableaux with $m$ boxes and entries in $[n]$.
Let $\Tab_n(\lambda)$ denote the subset of $T \in \Tab_{n}(|\lambda|)$ of shape $\lambda$.

The \emph{row reading word} of a tableau $T$ is 
the sequence $\row(T)$
formed by listing the entries of  $T$  row-by-row from left to right, starting with the top row.
The row reading words of the tableaux in \eqref{tab-ex} are $34224$, $34234$, and $35124$.
The \emph{column reading word} of $T$ is the sequence $\col(T)$ formed by listing the entries of $T$
down each column, starting with the first column.
The column reading words of the tableaux in \eqref{tab-ex} are $32424$, $32434$, and $31524$.

We introduce the term \emph{quasi-isomorphism}
to mean
a morphism $\psi : \cB \to \cC$ between abstract $\gl_n$- or $\q_n$-crystals 
with the property that
 for each full subcrystal $\tilde \cB \subset \cB$, there is a full subcrystal $\tilde \cC \subset \cC$
such that $\psi$ restricts to an isomorphism $\tilde\cB \to \tilde\cC$.
Let $m,n \in \PP$.

\begin{thmdef}[{\cite[\S3.1]{BumpSchilling}}]
\label{tab-thmdef} 
There is a unique abstract $\gl_n$-crystal structure on $\Tab_n(m)$
that makes the injective map $\row:\Tab_n(m)\to \cW_n(m)$ into a quasi-isomorphism.
The full $\gl_n$-subcrystals of $\Tab_n(m)$  
 are the sets 
$\Tab_n(\lambda)$ as $\lambda$ ranges over all partitions of $m$ with $\leq n$ parts.
\end{thmdef}


\begin{remark}
If $s_\lambda$ denotes the \emph{Schur function} of a partition $\lambda$,
then 
the character of the abstract $\gl_n$-crystal
$\Tab_n(\lambda)$ is  the Schur polynomial $s_\lambda(x_1,x_2,\dots,x_n) \in \Lambda_n$ \cite[Eq.\ (3.3)]{BumpSchilling}.
\end{remark}

\subsection{Shifted tableaux}\label{shtab-sect}

If $i \in \ZZ$ then we define $i' := i - \frac{1}{2}$. Using this notation,
 we can write \[\tfrac{1}{2}\ZZ = \{\ldots  <0' < 0 < 1' < 1 < 2'<2<\ldots\}.\]
Since $1' = \frac{1}{2}$, we have $(i+1)' = i' + 1 = i+1'$ for all $i\in \ZZ$.
An element of $\frac{1}{2}\ZZ$ is a \emph{primed number} if it has the form $i' \in \ZZ -\frac{1}{2}$ for some $i \in \ZZ$.
We sometimes refer to integers $i \in \ZZ$ as \emph{unprimed numbers}.


The shifted diagram of a strict partition $\mu = (\mu_1 > \dots > \mu_n > 0)$ 
is the set of pairs
\[\SD_\mu := \{ (i,j) \in [n]\times [\mu_1] : i \leq j  \leq \mu_i + i - 1\}.\]
We use the term \emph{shifted tableau} to mean a map from the shifted diagram
$\SD_\mu$ of some strict to partition to $\frac{1}{2}\ZZ$.
A shifted tableau with domain $\SD_\mu$
has shape $\mu$.
The \emph{(main) diagonal}
of a shifted tableau consists of the boxes $(i,j)$ in its domain with $i=j$.

A shifted tableau with positive entries is \emph{semistandard}
if its rows and columns are weakly increasing,
no unprimed number appears more than once in a column, and no
primed number appears in a diagonal position or more than once  in a row.
A shifted tableau is \emph{increasing} if it contains no primed entries 
and its rows and columns are strictly increasing.
A shifted tableau with $m$ boxes is \emph{standard} if it is semistandard
and its boxes contain exactly one of $i$ or $i'$ for each $i \in [m]$.
The following examples are respectively semistandard, increasing, and standard:
\[
\ytab{
 \none & 3 & 4'  \\ 
 2 & 2 & 4'
}
 \quand
\ytab{
 \none & 4 & 5  \\ 
 2 & 3 & 4
}
 \quand
\ytab{
  \none & 3 & 5'   \\ 
 1 & 2' & 4
}.
\]
Fix $m,n \in \PP$ and 
let $\STab_n(m)$ denote the set of semistandard shifted tableaux with $m$ boxes and entries in $\{1'<1<2'<2\dots<n'<n\}$.
For each strict partition $\mu$,
let $\STab_n(\mu)$ denote the subset of 
shifted tableaux in $ \STab_{n}(|\mu|)$ of shape $\mu$.
The row and column reading words of a shifted tableau are defined in the same way as for ordinary tableaux.

Results in \cite{AssafOguz,HPS,Hiroshima2018} show that
the set $\STab_n(m)$ carries a natural $\q_n$-crystal structure, 
which we describe in Section~\ref{app-sect}.
Except for the proof of Theorem~\ref{dual-thm2},
we will not need to work with the explicit formulas for the relevant crystal operators the appear in this appendix.
Instead, it will suffice to use the following characterization of the
$\q_n$-crystal structure on $\STab_n(m)$ in terms of Haiman's notion of \emph{mixed (shifted) insertion} \cite[Definition 6.7]{HaimanMixed}.

Mixed insertion
is defined as an iterative
algorithm, where at each stage a number $x$ is ``inserted'' 
into a row or column of a tableau. When this happens, 
$x$ replaces some other (usually larger) entry $y$, and we say that
the number $y$ is ``bumped.''

\begin{definition}[\cite{HaimanMixed}]
\label{hm-def}
Given a word $w=w_1w_2\cdots w_m \in \cW_n(m)$,
let $\emptyset = T_0,T_1,\dots,T_m=P_\HM(w)$
be the sequence of shifted tableaux in which
 $T_i$ for $i \in [m]$ is formed from $T_{i-1}$ by inserting $w_i$ according to the following procedure:
\ben
\item[] Start by inserting $w_i$ into the first row.
At each stage, an entry $x$ is inserted into a row or column.
Let $y$ be 
the first entry in the row going left to right (respectively, column going bottom to top) with $x< y$.
If no such entry $y$ exists then $x$ is added to the end of the row or column.
Otherwise, $x$ replaces $y$ and we 
continue by inserting $y$ into
the next row if $y$ is unprimed or into the next column otherwise,
with the exception that if $y$ is on the main diagonal (and therefore unprimed) then we insert 
the primed number $y'$ into the next column.
\een
We call $P_\HM(w)$ the \emph{mixed insertion tableau} of $w$. The \emph{mixed recording tableau}
$Q_\HM(w)$ is the shifted tableau with the same shape as $P_\HM(w)$
which contains $i$ in the box added to $T_{i-1}$. 
\end{definition}

\begin{example}\label{hm-ex} We compute $P_\HM(w)$ and $Q_\HM(w)$ for $w=332332$:
\[ 
{\small
\barr{rl}
\ytab{
  3 
}
   \leadsto 
\ytab{
  3 & 3
}
   \leadsto 
\ytab{
\none & 3 \\
  2 &3' 
}
      \leadsto 
\ytab{
\none & 3 \\
  2 &3'  & 3
}
      \leadsto 
\ytab{
\none & 3 \\
  2 &3'  & 3 & 3 
}
        \leadsto 
\ytab{
  \none & 3 & 3 \\
  2 & 2 & 3' & 3
} 
  = P_\HM(w)
\quand
\ytab{
  \none & 3 & 6 \\
  1 & 2 & 4 & 5
}  = Q_\HM(w).
\earr}
\]
\end{example}

The following theorem is due to Haiman \cite{HaimanMixed},
but the cited result in \cite{HPS} matches our notation.

\begin{theorem}[{\cite{HaimanMixed}; see \cite[Theorem 3.12]{HPS}}]
\label{hm-bijection-thm}
The map $w\mapsto (P_\HM(w), Q_\HM(w))$ is a bijection from 
$\cW_n(m)$ to the set of pairs $(P,Q)$
of shifted tableaux of the same shape in which $P \in \STab_n(m)$ 
and $Q$ is standard with no primed entries.
\end{theorem}

The  $\q_n$-crystal structure on $\STab_n(m)$
is now determined by the following result.

\begin{thmdef}[\cite{AssafOguz,HPS,Hiroshima2018}]
\label{stab-thmdef}
Fix integers $m,n \in \PP$. Then:
\ben
\item[(a)] The full $\q_n$-subcrystals of  $\cW_n(m)$
are the sets on which $Q_\HM$ is constant.

\item[(b)] There is a unique abstract $\q_n$-crystal structure on the set $\STab_n(m)$
that makes the surjective map  
$P_\HM : \cW_n(m) \to \STab_n(m)$ into a quasi-isomorphism.

\item[(c)] The full $\q_n$-subcrystals of $\STab_n(m)$ in this structure are
the sets 
$\STab_n(\mu)$ where $\mu$ ranges over all strict partitions of $m$ with at most $n$ parts.
\een
\end{thmdef}

\begin{proof}[Proof sketch]
One can check directly
 that $Q_\HM(w) = Q_\HM(f_{\overline 1}(w))$ 
 if $f_{\overline 1}(w)\neq 0$.
The other properties follow from \cite[Theorems 3.13 and 4.3]{HPS}, \cite[Theorem 3.2]{Hiroshima2018},
 and  \cite[Theorem 4.8]{AssafOguz}.
%
\end{proof}

It follows that the weight map for  $\STab_n(m)$
assigns to a shifted tableau $T$ the sequence $\weight(T) \in \NN^n$
whose $i$th entry is the number of times $i$ or $i'$ appears in $T$.
For example, 
\be\label{stab-weight-eq} \weight\(\ytab{
 \none & 3 & 4  \\ 
 2 & 2 & 4'
}\) = (0,  2,1,2,0)
\ee
when $n=5$.
For an example of the $\q_n$-crystal $\STab_n(m)$, see Figure~\ref{stab-fig}.
For more direct formulas for the operators $e_{i}$
and $f_{i}$ acting on this crystal, see Section~\ref{app-sect}.

\begin{remark}
If $P_\mu$ denotes the \emph{Schur $P$-function} of a strict partition $\mu$ (see \cite[\S{III.8}]{Macdonald}),
then 
the character of the abstract $\q_n$-crystal
$\STab_n(\mu)$ is
the Schur $P$-polynomial
$P_\mu(x_1,x_2,\dots,x_n) \in \Gamma_n$.
\end{remark}

\begin{figure}[h]
  \begin{center}
    \begin{tikzpicture}[xscale=2.4, yscale=2.4,>=latex]
      \node at (2,0) (T02) {$\ytab{\none & 3 \\ 2 & 3' & 3}$};
      \node at (1,1) (T11) {$\ytab{\none & 3 \\ 1 & 3' & 3}$};
      \node at (2,1) (T12) {$\ytab{\none & 3 \\ 2 & 2 & 3}$};
      \node at (3,1) (T13) {$\ytab{\none &3 \\ 2 & 2 & 3'}$};
      \node at (0,2) (T20) {$\ytab{\none & 3 \\1 & 2' & 3}$};
      \node at (1,2) (T21) {$\ytab{\none & 3 \\ 1 & 2 & 3}$};
      \node at (2,2) (T22) {$\ytab{\none & 3 \\ 2 & 2  & 2}$};
      \node at (3,2) (T23) {$\ytab{\none & 3 \\ 1 & 2 & 3'}$};
      \node at (5,2) (T25) {$\ytab{\none & 3 \\ 1 & 2' & 3'}$};
      \node at (0,3) (T30) {$\ytab{\none & 2 \\ 1 & 2' & 3}$};
      \node at (1,3) (T31) {$\ytab{\none & 3 \\ 1 & 1 & 3}$};
      \node at (2,3) (T32) {$\ytab{\none & 3 \\ 1 & 2 & 2 }$};
      \node at (3,3) (T33) {$\ytab{\none & 3 \\ 1 & 2' & 2}$};
      \node at (4,3) (T34) {$\ytab{\none & 3 \\ 1 & 1 & 3'}$};
      \node at (5,3) (T35) {$\ytab{\none & 2 \\ 1 & 2' & 3'}$};
      \node at (0,4) (T40) {$\ytab{\none & 2 \\ 1 & 2' & 2}$};
      \node at (1,4) (T41) {$\ytab{\none & 2 \\ 1 & 1 & 3}$};
      \node at (2,4) (T42) {$\ytab{\none & 3 \\ 1 & 1 & 2 }$};
      \node at (3,4) (T43) {$\ytab{\none & 3 \\ 1 & 1 & 2'}$};
      \node at (5,4) (T45) {$\ytab{\none & 2 \\ 1 & 1 & 3'}$};
      \node at (1,5) (T51) {$\ytab{\none & 2 \\ 1 & 1 & 2 }$};
      \node at (2,5) (T52) {$\ytab{\none & 3 \\ 1 & 1 & 1}$};
      \node at (3,5) (T53) {$\ytab{\none & 2 \\ 1 & 1 & 2'}$};
      \node at (2,6) (T62) {$\ytab{ \none & 2 \\ 1 & 1 & 1 }$};
      \draw[->,thick]  (T62) -- (T51) node[midway,above,scale=0.75] {$1$};
      \draw[->,thick]  (T62) -- (T52) node[midway,right,scale=0.75] {$2$};
      \draw[->,thick]  (T62) -- (T53) node[midway,above,scale=0.75] {$\overline 1$};
      \draw[->,thick]  (T51.218) -- (T40.52) node[midway,above,scale=0.75] {$\overline 1$};
      \draw[->,thick]  (T51.232) -- (T40.38) node[midway,below,scale=0.75] {$1$};
      \draw[->,thick]  (T51) -- (T41) node[midway,right,scale=0.75] {$2$};
      \draw[->,thick]  (T52) -- (T42) node[midway,right,scale=0.75] {$1$};
      \draw[->,thick]  (T52) -- (T43) node[midway,above,scale=0.75] {$\overline 1$};
      \draw[->,thick]  (T53) -- (T43) node[midway,right,scale=0.75] {$2$};
      \draw[->,thick]  (T40) -- (T30) node[midway,left,scale=0.75] {$2$};
      \draw[->,thick]  (T41.218) -- (T30.52) node[midway,above,scale=0.75] {$\overline 1$};
      \draw[->,thick]  (T41.232) -- (T30.38) node[midway,below,scale=0.75] {$1$};
      \draw[->,thick]  (T41) -- (T31) node[midway,right,scale=0.75] {$2$};
      \draw[->,thick]  (T42) -- (T32) node[midway,right,scale=0.75] {$1$};
      \draw[->,thick]  (T42) -- (T33) node[midway,above,scale=0.75] {$\overline 1$};
      \draw[->,thick]  (T43) -- (T33) node[midway,right,scale=0.75] {$1$};
      \draw[->,thick]  (T43) -- (T34) node[midway,above,scale=0.75] {$2$};
      \draw[->,thick]  (T45.260) -- (T35.100) node[midway,left,scale=0.75] {$\overline 1$};
      \draw[->,thick]  (T45.280) -- (T35.80) node[midway,right,scale=0.75] {$1$};
      \draw[->,thick]  (T30) -- (T20) node[midway,left,scale=0.75] {$2$};
      \draw[->,thick]  (T31) -- (T20) node[midway,above,scale=0.75] {$\overline 1$};
      \draw[->,thick]  (T31) -- (T21) node[midway,right,scale=0.75] {$1$};
      \draw[->,thick]  (T32) -- (T21) node[midway,above,scale=0.75] {$2$};
      \draw[->,thick]  (T32.260) -- (T22.100) node[midway,left,scale=0.75] {$\overline 1$};
      \draw[->,thick]  (T32.280) -- (T22.80) node[midway,right,scale=0.75] {$1$};
      \draw[->,thick]  (T33) -- (T23) node[midway,right,scale=0.75] {$2$};
      \draw[->,thick]  (T34) -- (T23) node[midway,above,scale=0.75] {$1$};
      \draw[->,thick]  (T34) -- (T25) node[midway,above,scale=0.75] {$\overline 1$};
      \draw[->,thick]  (T35) -- (T25) node[midway,right,scale=0.75] {$2$};
      \draw[->,thick]  (T20) -- (T11) node[midway,above,scale=0.75] {$2$};
      \draw[->,thick]  (T21.322) -- (T12.128) node[midway,above,scale=0.75] {$\overline 1$};
      \draw[->,thick]  (T21.308) -- (T12.142) node[midway,below,scale=0.75] {$1$};
      \draw[->,thick]  (T22) -- (T12) node[midway,right,scale=0.75] {$2$};
      \draw[->,thick]  (T23.260) -- (T13.100) node[midway,left,scale=0.75] {$\overline 1$};
      \draw[->,thick]  (T23.280) -- (T13.80) node[midway,right,scale=0.75] {$1$};
      \draw[->,thick]  (T11.322) -- (T02.128) node[midway,above,scale=0.75] {$\overline 1$};
      \draw[->,thick]  (T11.308) -- (T02.142) node[midway,below,scale=0.75] {$1$};
      \draw[->,thick]  (T12) -- (T02) node[midway,right,scale=0.75] {$2$};
     \end{tikzpicture}
  \end{center}
\caption{The $\q_3$-crystal graph of $\STab_3(\lambda)$ for $\lambda=(3,1)$.}
\label{stab-fig}
\end{figure}

\section{Crystals of factorizations}\label{factor-sect}

In this section, we first review three families of abstract crystals introduced in \cite{Hiroshima,MorseSchilling}.
The elements of each crystal are increasing factorizations of reduced words for certain permutations.
We then present our main new results
 in Section~\ref{main-sect}.
Throughout, $n\in \PP$ is a positive integer.

\subsection{Reduced factorizations}\label{reduced-sect}

Let $S_\ZZ$ denote the group of permutations of $\ZZ$ fixing all but finitely many integers.
The simple transpositions $s_i := (i,i+1) \in S_\ZZ$ for $i \in \ZZ$
generate $S_\ZZ$ and with
respect to this generating set $S_\ZZ$ is a Coxeter group,
whose length function $\ell : S_\ZZ \to \NN$ counts the
inversions of a permutation.

\begin{definition}
A \emph{reduced word} for a permutation $\pi \in S_\ZZ$
is a minimal-length sequence of integers $i_1i_2\cdots i_l$
with $\pi = s_{i_1}s_{i_2}\cdots s_{i_l}$.
Let $\cR(\pi)$ denote the set of reduced words for $\pi \in S_\ZZ$.
\end{definition}

Let $\pi \in S_\ZZ$.
It is well-known that $\ell(\pi)$ is the length every  word in $ \cR(\pi)$,
which is a single equivalence class under the transitive closure of the \emph{Coxeter braid relations}
defined by
$ \cdots ij\cdots \sim \cdots ji \cdots$ for $|i-j|>1$ 
and 
$\cdots i(i+1)i\cdots \sim \cdots (i+1)i(i+1)\cdots$
for all $i \in \ZZ$ \cite[Lemma 6.18]{EG}.

\begin{definition}\label{reduced-fac-def}
An \emph{increasing factorization} of a word $w$
is a finite sequence $(w^1,w^2,\dots,w^n)$ 
in which each $w^i$ is a strictly increasing, possibly empty word and $w=w^1w^2\cdots w^n$.
A \emph{reduced factorization}
is an increasing factorization of a reduced word for some 
$\pi \in S_\ZZ$.
Let $\cR_n(\pi)$ denote the set of increasing factorizations
with $n$ factors of reduced words for $\pi \in S_\ZZ$.
\end{definition}

Given increasing words $a=a_1a_2\cdots a_p$ and $b=b_1b_2\cdots b_q$,
define a set $\pair(a,b)$ inductively as follows.
If no $(i,j) \in [p] \times [q]$ exists with $a_i > b_j$,
then  $\pair(a,b) = \varnothing$. Otherwise, choose $j$ to be maximal with $\{ i  : a_i > b_j\}\neq\varnothing$, then choose $i $ to be minimal with $a_i>b_j$,
and finally set 
\be\pair(a,b) = \{(a_i, b_j)\} \sqcup \pair(a_1\cdots a_{i-1}a_{i+1}\cdots a_p, b_1\cdots b_{j-1}b_{j+1}\cdots b_q).\ee
Equivalently, we form $\pair(a,b)$ by iterating over the letters in the second word from largest to smallest; at each iteration,
the current letter $b_j$ is paired with the smallest unpaired letter $a_i$ in the first word with $a_i>b_j$, if such a letter exists. For example,
if
$u = 1 , 3, 4, 5, 8, 10, 11
$
and
$
v = 2, 6, 9 , 12 , 13
$
then we have
$\pair(u,v) = \{ (10,9), (8,6), (3,2) \}$.

The set $\cR_n(\pi)$ has an abstract $\gl_n$-crystal structure \cite{MorseSchilling},
which we can describe as follows.
Fix $\pi \in S_\ZZ$ and $w=(w^1,w^2,\dots,w^n) \in \cR_n(\pi)$.
Let
$\weight(w) = (\ell(w^1),\ell(w^2),\dots,\ell(w^n)) \in \NN^n$.
To define $f_i(w)$ and $e_i(w)$
for $i \in [n-1]$,
we examine the set $\pair(w^i,w^{i+1})$:
\begin{itemize}
\item If every letter in $w^i$ belongs to $\{ a : (a,b) \in \pair(w^i,w^{i+1})\}$,
then $f_i(w) = 0$. Otherwise, form
$f_i(w)$ by removing from $w^i$ its largest unpaired letter $x$,
and then adding to $w^{i+1}$ the smallest integer $y\geq x$ that is not already a letter
(in the position yielding an increasing word).

\item If every letter in $w^{i+1}$ belongs to $\{ b : (a,b) \in \pair(w^i,w^{i+1})\}$,
then $e_i(w) = 0$. Otherwise, form 
$e_i(w)$ by removing  from $w^{i+1}$
its smallest unpaired letter $y$,
and then adding to $w^i$ the largest integer $x\leq y$ that is not already a letter
(in the position yielding an increasing word).

\end{itemize}
The following is equivalent to \cite[Theorem 3.5]{MorseSchilling};
see also \cite[\S10.2]{BumpSchilling}.
Both references work with factorizations 
into \emph{decreasing} subwords that are indexed in 
reverse order as $(w^n,\dots,w^2,w^1)$. 
Reading everything backwards translates the relevant statements to
what is given here.

\begin{proposition}[\cite{MorseSchilling}]
Relative to the maps $\weight$, $e_i$, $f_i$ just given, 
the set of reduced factorizations
$\cR_n(\pi)$ is an abstract $\gl_n$-crystal  for all  $\pi \in S_\ZZ$.
(We call this a \emph{Morse-Schilling crystal}.)
\end{proposition}

For examples of Morse-Schilling crystals,
see Figures~\ref{ir-fig} and \ref{irfpf-fig}, ignoring any arrows $x\xrightarrow{\overline 1} y$.

\begin{remark}\label{nonempty-rmk1}
In \cite{Stanley}, Stanley associates to each permutation $\pi$
a certain homogeneous symmetric function $F_\pi$. (In later references, the 
indexing convention for this power series is often inverted.)
It is clear from the definition in \cite{Stanley} that the character of the crystal $\cR_n(\pi)$
is the polynomial $F_\pi(x_1,x_2,\dots,x_n) \in \Lambda_n$
obtained by specializing the \emph{Stanley symmetric function} $F_\pi$
to $n$ variables.

Stanley derives a formula for the maximal term in the Schur expansion of $F_\pi$ relative to dominance order on
partitions \cite[Theorem 4.1]{Stanley}. Because $s_\lambda(x_1,x_2,\dots,x_n)$ is nonzero if and only if 
the number of parts of $\lambda$ satisfies $\ell(\lambda)\leq n$, and because one has
 $\mu \leq \lambda$ in dominance order only if $\ell(\mu)\geq \ell(\lambda)$,
Stanley's formula implies that $F_\pi(x_1,x_2,\dots,x_n)$ is nonzero if and only if the numbers
$c_i(\pi) := |\{ i < j \in \ZZ : \pi(i) > \pi(j)\}|$  that make up the \emph{code} of $\pi$ are at most $n$ for all $i \in \ZZ$.
Necessarily, the set $\cR_n(\pi)$ is nonempty if and only if the same condition holds.
\end{remark}
 
\subsection{Orthogonal factorizations}\label{o-fact-sect}

Let 
$I_\ZZ = \{ \pi \in S_\ZZ : \pi = \pi^{-1}\}$
be the set of involutions in $S_\ZZ$.
Such permutations are in bijection with the (incomplete) matchings on $\ZZ$
with finitely many edges.
Given $\pi \in S_\ZZ$ and $ i \in \ZZ$, define 
\[
\pi\circ s_i = \begin{cases} \pi s_i &\text{if }\pi(i) < \pi(i+1) \\
\pi &\text{if }\pi(i)>\pi(i+1)\end{cases}
\quand
\pi \rtimes s_i = \begin{cases} s_i \pi s_i & \text{if $\pi s_i \neq s_i \pi$} \\
\pi s_i &\text{if $\pi s_i = s_i \pi$}.
\end{cases}
\]
The operation $\circ$ extends to an associative product $S_\ZZ \times S_\ZZ \to S_\ZZ$ \cite[\S7.1]{Humphreys}.
The operation $\rtimes $ is not associative, 
but one has $\pi \rtimes s_i \in I_\ZZ$ if and only if $\pi \in I_\ZZ$.

\begin{lemma}[{\cite[\S2]{HMP2}}]
\label{rtimes-lem}
Let $\pi \in S_\ZZ$ and $i_1,i_2,\dots, i_l \in \ZZ$. The following are equivalent:
\ben
\item[(a)]   $i_1i_2\cdots i_l$ is a minimal-length word with $\pi =s_{i_l} \circ \cdots \circ s_{i_2}\circ s_{i_1} \circ 1 \circ s_{i_1} \circ s_{i_2} \circ \cdots \circ s_{i_l}.$

\item[(b)] $i_1i_2\cdots i_l$ is a minimal-length word with $\pi =(\cdots ((1 \rtimes   s_{i_1}) \rtimes s_{i_2}) \rtimes \cdots  )\rtimes s_{i_l}.$

\item[(c)] $i_j$ is not a descent of $ (\cdots ((1 \rtimes   s_{i_1}) \rtimes s_{i_2}) \rtimes \cdots  )\rtimes s_{i_{j-1}}$
for each $j \in [l]$.
\een
Moreover, if these properties hold then we must have $\pi \in I_\ZZ$.
\end{lemma}
  
\begin{definition}
An \emph{involution word} for $\pi \in I_\ZZ$
is a sequence of integers $i_1i_2\cdots i_l$ with the equivalent properties in Lemma~\ref{rtimes-lem}.
Let $\iR(\pi)$ denote the set of involution words for $\pi \in S_\ZZ$.
\end{definition}

One can view  
 $\iR(\pi)$ as an ``orthogonal'' analogue of 
the usual set of reduced words $\cR(\pi)$. 
Elements of the latter correspond to maximal chains in weak order on the set of Borel orbits in the type A flag variety $\Fl_n$. The orbits of the orthogonal group $\O_n(\CC)$ in $\Fl_n$ are in bijection with 
the set of permutations $\pi \in I_\ZZ$ with support in $[n]$.
Involution words may be identified with the set of maximal chains in an analogous weak order on these orbits; see \cite{CJW,RichSpring,WyserYong}.

Let $\isim$
denote
the transitive closure of
 the Coxeter braid relations 
plus the relation 
on words that has
$
w_1w_2w_3\cdots w_m \isim w_2w_1w_3\cdots w_m$ 
for any choice of $w_1,w_2,w_3,\dots,w_m\in\ZZ$.

\begin{theorem}[{\cite[Theorem 3.1]{HuZhang1}}]
\label{isim-thm}
If $\pi \in I_\ZZ$
then $\iR(\pi)$ is 
a single equivalence class under
$\isim$.
Moreover, an equivalence class of words under $\isim$ is
equal to $\iR(\pi)$ for some $\pi \in I_\ZZ$
if and only if no word in the class has equal adjacent letters.
\end{theorem}

The theorem implies that there is a finite set $\cA(\pi)\subset S_\ZZ$
with 
$
\iR(\pi) = \bigsqcup_{\sigma \in \cA(\pi)} \cR(\sigma)
$ for each $\pi \in I_\ZZ$.
In prior related work, the sets $\iR(\pi)$ and $\cA(\pi)$ have usually been denoted ``$\hat{\mathcal{R}}(\pi)$'' and ``$\mathcal{A}(\pi)$'' or sometimes ``$\mathcal{W}(\pi)$''.
Our present convention follows \cite{Pawlowski2019}.

\begin{definition}
An \emph{orthogonal factorization}
is an increasing factorization of an involution word for some element of $I_\ZZ$.
Let $\iR_n(\pi)$ denote the set of increasing factorizations with $n$ factors
of involution words for $\pi \in I_\ZZ$.
\end{definition}

Let $\pi \in I_\ZZ$.
Then $\iR_n(\pi) = \bigsqcup_{\sigma \in \cA(\pi)} \cR_n(\sigma)$ is a finite, disjoint union of Morse-Schilling crystals,
so is an abstract $\gl_n$-crystal.
Hiroshima \cite[Appendix B]{Hiroshima} has identified two new
 operators $\fO$ and $\eO$ acting on this set.
Given $w=(w^1,w^2,\dots,w^n) \in \iR_n(\pi)$, define $\fO(w)$ and $\eO(w)$ as follows:
\begin{itemize}
\item If $w^1\neq \emptyset$ and its first letter $x$ is smaller than every letter in $w^2$,
then $\fO(w)$ is formed from $w$ by moving $x$ from $w^1$ to $w^2$.
Otherwise, $\fO(w) := 0$.

\item If $w^2\neq \emptyset$ and its first letter $x$ is smaller than every letter in $w^1$,
then $\eO(w)$ is formed from $w$ by moving $x$ from $w^2$ to $w^1$.
Otherwise, $\eO(w) := 0$.
\end{itemize}
Note that Theorem~\ref{isim-thm} implies that the smallest letters of $w^1$ and $w^2$ can never be equal.

\begin{proposition}[{\cite[Theorem B.2]{Hiroshima}}]
Relative to the queer crystal operators $e_{\overline 1}=\eO$ and $f_{\overline 1}=\fO$ just given, 
the abstract $\gl_n$-crystal 
$\iR_n(\pi)$ is an abstract $\q_n$-crystal for all $\pi \in I_\ZZ$.
\end{proposition}

For an example of the crystal $\iR_n(\pi)$, see Figure~\ref{ir-fig}.

\begin{remark}\label{iF-rmk}
The character of the crystal
$\iR_n(\pi)$ 
is the polynomial $\iF_\pi(x_1,x_2,\dots,x_n) \in \Gamma_n$,
where
$
\iF_\pi := \sum_{\sigma \in \cA(\pi)} F_\sigma
$
is the \emph{involution Stanley symmetric function} studied in  \cite{HMP4}.
Because the definition of $F_\sigma$ in \cite{HMP4} differs from the one in \cite{Stanley}
by inverting indices, it is not obvious that the formula for $\iF_\pi$ just given matches what is in \cite{HMP4}.
However, this follows from \cite[Lemma 5.3 and Corollary 5.10]{Marberg2019a},
which show that $ \sum_{\sigma \in \cA(\pi)} F_\sigma =  \sum_{\sigma \in \cA(\pi)} F_{\sigma^{-1}}$.

The maximal term in the Schur $P$-expansion of $\iF_\pi$ relative to dominance order is given in 
\cite[Theorem 1.13]{HMP4}. Similar to Remark~\ref{nonempty-rmk1}, since $P_\mu(x_1,x_2,\dots,x_n)\neq 0$ if and only if $\ell(\mu)\leq n$,
this result characterizes when
$\iF_\pi(x_1,x_2,\dots,x_n)\neq 0$ and hence also when $\iR_n(\pi)\neq \varnothing$.
Specifically, \cite[Theorem 1.13]{HMP4} implies that $\iR_n(\pi)$ is nonempty if and only if the numbers $\hat c_i(\pi) := | \{ i < j \in \ZZ :  \min\{i,\pi(i)\} \geq  \pi(j) \}|$
in the \emph{involution code} of $\pi$ \cite[Definition 4.6]{HMP4} are at most $n$ for all $i\in \ZZ$.
\end{remark}

\begin{figure}[h]
  \begin{center}
  {
    \begin{tikzpicture}[xscale=2.4, yscale=2.4,>=latex]
      \node at (2,0) (T02) {$\boxed{\emptyset \pipe 3 \pipe 124}$};
      \node at (1,1) (T11) {$\boxed{3 \pipe \emptyset \pipe 124}$};
      \node at (2,1) (T12) {$\boxed{\emptyset \pipe 13 \pipe 24}$};
      \node at (3,1) (T13) {$\boxed{\emptyset \pipe 34 \pipe 12}$};
      \node at (0,2) (T20) {$\boxed{3 \pipe 1 \pipe 24}$};
      \node at (1,2) (T21) {$\boxed{1 \pipe 3 \pipe 24}$};
      \node at (2,2) (T22) {$\boxed{\emptyset \pipe 134 \pipe 2}$};
      \node at (3,2) (T23) {$\boxed{3 \pipe 4 \pipe 12}$};
      \node at (5,2) (T25) {$\boxed{4 \pipe 3 \pipe 12}$};
      \node at (0,3) (T30) {$\boxed{3 \pipe 12 \pipe 4}$};
      \node at (1,3) (T31) {$\boxed{13 \pipe \emptyset \pipe 24}$};
      \node at (2,3) (T32) {$\boxed{1 \pipe 34 \pipe 2}$};
      \node at (3,3) (T33) {$\boxed{3 \pipe 14 \pipe 2}$};
      \node at (4,3) (T34) {$\boxed{34 \pipe \emptyset \pipe 12}$};
      \node at (5,3) (T35) {$\boxed{4 \pipe 13 \pipe 2}$};
      \node at (0,4) (T40) {$\boxed{3 \pipe 124 \pipe \emptyset}$};
      \node at (1,4) (T41) {$\boxed{13 \pipe 2 \pipe 4}$};
      \node at (2,4) (T42) {$\boxed{13 \pipe 4 \pipe 2}$};
      \node at (3,4) (T43) {$\boxed{34 \pipe 1 \pipe 2}$};
      \node at (5,4) (T45) {$\boxed{14 \pipe 3 \pipe 2}$};
      \node at (1,5) (T51) {$\boxed{13 \pipe 24 \pipe \emptyset}$};
      \node at (2,5) (T52) {$\boxed{134 \pipe \emptyset \pipe 2}$};
      \node at (3,5) (T53) {$\boxed{34 \pipe 12 \pipe \emptyset}$};
      \node at (2,6) (T62) {$\boxed{134 \pipe 2 \pipe \emptyset}$};
      \draw[->,thick]  (T62) -- (T51) node[midway,above,scale=0.75] {$1$};
      \draw[->,thick]  (T62) -- (T52) node[midway,right,scale=0.75] {$2$};
      \draw[->,thick]  (T62) -- (T53) node[midway,above,scale=0.75] {$\overline 1$};
      \draw[->,thick]  (T51.218) -- (T40.52) node[midway,above,scale=0.75] {$\overline 1$};
      \draw[->,thick]  (T51.232) -- (T40.38) node[midway,below,scale=0.75] {$1$};
      \draw[->,thick]  (T51) -- (T41) node[midway,right,scale=0.75] {$2$};
      \draw[->,thick]  (T52) -- (T42) node[midway,right,scale=0.75] {$1$};
      \draw[->,thick]  (T52) -- (T43) node[midway,above,scale=0.75] {$\overline 1$};
      \draw[->,thick]  (T53) -- (T43) node[midway,right,scale=0.75] {$2$};
      \draw[->,thick]  (T40) -- (T30) node[midway,left,scale=0.75] {$2$};
      \draw[->,thick]  (T41.218) -- (T30.52) node[midway,above,scale=0.75] {$\overline 1$};
      \draw[->,thick]  (T41.232) -- (T30.38) node[midway,below,scale=0.75] {$1$};
      \draw[->,thick]  (T41) -- (T31) node[midway,right,scale=0.75] {$2$};
      \draw[->,thick]  (T42) -- (T32) node[midway,right,scale=0.75] {$1$};
      \draw[->,thick]  (T42) -- (T33) node[midway,above,scale=0.75] {$\overline 1$};
      \draw[->,thick]  (T43) -- (T33) node[midway,right,scale=0.75] {$1$};
      \draw[->,thick]  (T43) -- (T34) node[midway,above,scale=0.75] {$2$};
      \draw[->,thick]  (T45.260) -- (T35.100) node[midway,left,scale=0.75] {$\overline 1$};
      \draw[->,thick]  (T45.280) -- (T35.80) node[midway,right,scale=0.75] {$1$};
      \draw[->,thick]  (T30) -- (T20) node[midway,left,scale=0.75] {$2$};
      \draw[->,thick]  (T31) -- (T20) node[midway,above,scale=0.75] {$\overline 1$};
      \draw[->,thick]  (T31) -- (T21) node[midway,right,scale=0.75] {$1$};
      \draw[->,thick]  (T32) -- (T21) node[midway,above,scale=0.75] {$2$};
      \draw[->,thick]  (T32.260) -- (T22.100) node[midway,left,scale=0.75] {$\overline 1$};
      \draw[->,thick]  (T32.280) -- (T22.80) node[midway,right,scale=0.75] {$1$};
      \draw[->,thick]  (T33) -- (T23) node[midway,right,scale=0.75] {$2$};
      \draw[->,thick]  (T34) -- (T23) node[midway,above,scale=0.75] {$1$};
      \draw[->,thick]  (T34) -- (T25) node[midway,above,scale=0.75] {$\overline 1$};
      \draw[->,thick]  (T35) -- (T25) node[midway,right,scale=0.75] {$2$};
      \draw[->,thick]  (T20) -- (T11) node[midway,above,scale=0.75] {$2$};
      \draw[->,thick]  (T21.322) -- (T12.128) node[midway,above,scale=0.75] {$\overline 1$};
      \draw[->,thick]  (T21.308) -- (T12.142) node[midway,below,scale=0.75] {$1$};
      \draw[->,thick]  (T22) -- (T12) node[midway,right,scale=0.75] {$2$};
      \draw[->,thick]  (T23.260) -- (T13.100) node[midway,left,scale=0.75] {$\overline 1$};
      \draw[->,thick]  (T23.280) -- (T13.80) node[midway,right,scale=0.75] {$1$};
      \draw[->,thick]  (T11.322) -- (T02.128) node[midway,above,scale=0.75] {$\overline 1$};
      \draw[->,thick]  (T11.308) -- (T02.142) node[midway,below,scale=0.75] {$1$};
      \draw[->,thick]  (T12) -- (T02) node[midway,right,scale=0.75] {$2$};
     \end{tikzpicture}
     }
  \end{center}
\caption{The $\q_3$-crystal graph of $\iR_3(\pi)$ for $\pi=(1,3)(2,5) \in I_\ZZ$.}
\label{ir-fig}
\end{figure}

\subsection{Symplectic factorizations}\label{sp-fact-sect}

Let $\Ifpf_\ZZ$ denote the $S_\ZZ$-conjugacy class of the 
 permutation $\minfpf$ of $\ZZ$ 
that maps $
i \mapsto i -(-1)^i.
$
The elements $\pi \in \Ifpf_\ZZ$ are the fixed-point-free involutions of $\ZZ$
with $\pi(i) = \minfpf(i)$ whenever $|i|$ is sufficiently large. 
If $n$ is even and $\pi \in I_\ZZ$ restricts to a fixed-point-free involution of $[n]$,
then there is a unique element $\check \pi \in \Ifpf_\ZZ$ that restricts to $\pi$ on $[n]$ and to $\minfpf$
on $\ZZ\setminus [n]$.
In examples, for convenience,
we usually make no distinction between the permutations $\pi$ and $\check \pi$.

\begin{definition}
An \emph{fpf-involution word} for 
an element $\pi \in \Ifpf_\ZZ$
is a minimal-length sequence of integers $i_1i_2\cdots i_l$
such that $\pi = s_{i_l}\cdots s_{i_2} s_{i_1} \cdot \minfpf\cdot  s_{i_1}s_{i_2}\cdots s_{i_l}.$
Let $\iRfpf(\pi)$ denote the set of fpf-involution words for $\pi \in \Ifpf_\ZZ$.
\end{definition}

One can view  
 $\iRfpf(\pi)$ as a ``symplectic'' analogue of $\cR(\pi)$. 
When $n$ is even, the orbits of the symplectic group $\Sp_n(\CC)$ acting on the type A flag variety $\Fl_n$ are in bijection with 
the set of involutions $\pi \in \Ifpf_\ZZ$ with $\pi(i) =\minfpf(i)$ for all $i \notin [n]$.
Words in $\iRfpf(\pi)$ may be identified with maximal chains in a weak order on these orbits \cite{CJW,WyserYong}.
Fpf-involution words are also instances of reduced words for \emph{quasiparabolic sets} as defined in \cite{RainsVazirani}.

 The following is helpful for checking whether a given word is an fpf-involution word:
 
 \begin{lemma}[{\cite[\S2.3]{HMP5}}]
\label{fpf-rtimes-lem}
An integer sequence $i_1i_2\cdots i_l$ is an fpf-involution word for some element of $\Ifpf_\ZZ$ if and only
$i_j$ is not a descent of 
$s_{i_{j-1}}\cdots s_{i_2} s_{i_1} \cdot \minfpf\cdot  s_{i_1}s_{i_2}\cdots s_{i_{j-1}}$ for all $j\in [l]$.
\end{lemma}

Since every odd integer is a descent of $\minfpf$, 
an fpf-involution word must start with an even letter.
Let $\fsim$
denote
the transitive closure of
 the Coxeter braid relations 
plus the symmetric relation 
on words that has
$w_1(w_1- 1)w_2\cdots w_m \fsim w_1(w_1+ 1)w_2\cdots w_m$
for any choice of $w_1,w_2,\dots,w_m$.

\begin{theorem}[{\cite[Theorem 2.4]{Marberg2019a}}]
\label{fsim-thm}
If $\pi \in \Ifpf_\ZZ$
then $\iRfpf(\pi)$ is 
an equivalence class under the relation
$\fsim$.
Moreover, an equivalence class of words under $\fsim$ is
equal to $\iRfpf(\pi)$ for some $\pi \in \Ifpf_\ZZ$
if and only if no word in the class has equal adjacent letters or starts with an odd letter.
\end{theorem}

This theorem implies that
there is a finite set $\cAfpf(\pi)\subset S_\ZZ$
with 
$
\iRfpf(\pi) = \bigsqcup_{\sigma \in \cAfpf(\pi)} \cR(\sigma)
$
for each $\pi \in \Ifpf_\ZZ$.
Our notation again follows the convention of \cite{Pawlowski2019}, which differs 
from prior related work where $\iRfpf(\pi)$ and $\cAfpf(\pi)$ have usually been denoted ``$\hat\cR^\fpf(\pi)$'' and ``$\mathcal{A}^\fpf(\pi)$.''

\begin{definition}
A \emph{symplectic factorization}
is a reduced factorization of an fpf-involution word
for some element of $\Ifpf_\ZZ$.
Let $\iRfpf_n(\pi)$ denote the set of increasing factorizations with $n$ factors of 
fpf-involution words for $\pi \in \Ifpf_\ZZ$.
\end{definition}

Fix $\pi \in \Ifpf_\ZZ$.
Then $\iRfpf_n(\pi)$ is again a disjoint union of Morse-Schilling crystals.
Hiroshima \cite[\S5]{Hiroshima} has defined two other operators $\fSp$ and $\eSp$ acting on this set, analogous to $\fO$ and $\eO$ in the previous section.
Given $w=(w^1,w^2,\dots,w^n) \in \iRfpf_n(\pi)$, form $\fSp(w)$ and $\eSp(w)$ as follows:
\begin{itemize}
\item If $w^1=\emptyset$ or $\min(w^2) \leq \min(w^1)$, then $\fSp(w):=0$.
Otherwise, let $x=\min(w^1)$.
If $x+1$ is not in $w^1$ 
then form $\fSp(w)$ from $w$ by moving $x$ from $w^1$ to $w^2$.
If $x+1$ is in $w^1$ then form $\fSp(w)$ from $w$ by removing $x+1$ from $w^1$
and adding $x-1$ to the start of $w^2$.

\item If $w^2=\emptyset$ or $\min(w^1) \leq \min(w^2)$, then $\eSp(w):=0$.
Otherwise, let $x=\min(w^2)$.
If $x$ is even 
then form $\eSp(w)$ from $w$ by moving $x$ from $w^2$ to $w^1$.
If $x$ is odd then form $\eSp(w)$ from $w$ by removing $x$ from $w^2$
and adding $x+2$ to $w^1$ (in the position yielding an increasing word).

\end{itemize}
Some observations are warranted.
As in the orthogonal case, if
$w^1$ and $w^2$ are both nonempty then
it follows from Theorem~\ref{fsim-thm} that
$\min(w^1) \neq \min(w^2)$.
Likewise, in the definition of $\eSp(w)$, if $x$ is odd and $\eSp(w)\neq 0$, then $x+2$ 
cannot be a letter in $w^1$.

\begin{proposition}[{\cite[Theorem 5.1]{Hiroshima}}]
Relative to the queer crystal operators $e_{\overline 1}=\eSp$ and $f_{\overline 1}=\fSp$ just given, 
the abstract $\gl_n$-crystal 
$\iRfpf_n(\pi)$ is an abstract $\q_n$-crystal for all $\pi \in \Ifpf_\ZZ$.
\end{proposition}

See Figure~\ref{irfpf-fig} for an example of the crystal $\iRfpf_n(\pi)$.

\begin{remark}\label{Ffpf-rmk}
The character of the crystal
$\iRfpf_n(\pi)$ 
is the polynomial $\Ffpf_\pi(x_1,x_2\dots,x_n) \in \Gamma_n$,
where
$\Ffpf_\pi := \sum_{\sigma \in \cAfpf(\pi)} F_\sigma$ is 
the \emph{fpf-involution Stanley symmetric function} studied in \cite{HMP5}.
As in Remark~\ref{iF-rmk},
it is not obvious that this definition of $\Ffpf_\pi$ is consistent with what is in \cite{HMP5}, but this holds as
 \cite[Lemma 5.3 and Corollary 5.10]{Marberg2019a} imply that $\sum_{\sigma \in \cAfpf(\pi)} F_\sigma =\sum_{\sigma \in \cAfpf(\pi)} F_{\sigma^{-1}}$.

The maximal term in the Schur $P$-expansion of $\Ffpf_\pi$ in dominance order is given in 
\cite[Theorem 1.4]{HMP5}. As in Remarks~\ref{nonempty-rmk1} and \ref{iF-rmk}, this result leads to a characterization of 
when $\iRfpf_n(\pi)$ is nonempty: namely, $\iRfpf_n(\pi)\neq \varnothing$ if and only if
the numbers $\hat c^\fpf_i(\pi) := | \{ i < j \in \ZZ :  \min\{i,\pi(i)\} >  \pi(j) \}|$
in the \emph{FPF-involution code} of $\pi$ (see \cite[\S1]{HMP5}) are at most $n$ for all $i \in \ZZ$.
\end{remark}

\begin{figure}[h]
  \begin{center}
  {
    \begin{tikzpicture}[xscale=2.4, yscale=2.4,>=latex]
      \node at (2,0) (T02) {$\boxed{ \emptyset \pipe 4 \pipe 235}$};
      \node at (1,1) (T11) {$\boxed{ 4 \pipe \emptyset \pipe 235}$};
      \node at (2,1) (T12) {$\boxed{\emptyset \pipe 24 \pipe 35 }$};
      \node at (3,1) (T13) {$\boxed{\emptyset \pipe 45 \pipe 23}$};
      \node at (0,2) (T20) {$\boxed{4 \pipe 2 \pipe 35 }$};
      \node at (1,2) (T21) {$\boxed{2 \pipe 4 \pipe 35 }$};
      \node at (2,2) (T22) {$\boxed{\emptyset \pipe 245 \pipe 3 }$};
      \node at (3,2) (T23) {$\boxed{ 4 \pipe 5 \pipe 23 }$};
      \node at (5,2) (T25) {$\boxed{4 \pipe 3 \pipe 23 }$};
      \node at (0,3) (T30) {$\boxed{4 \pipe 23 \pipe 5 }$};
      \node at (1,3) (T31) {$\boxed{ 24 \pipe \emptyset \pipe 35}$};
      \node at (2,3) (T32) {$\boxed{ 2 \pipe 45 \pipe 3}$};
      \node at (3,3) (T33) {$\boxed{ 4 \pipe 25 \pipe 3}$};
      \node at (4,3) (T34) {$\boxed{ 45 \pipe \emptyset \pipe 23}$};
      \node at (5,3) (T35) {$\boxed{4 \pipe 23 \pipe 2 }$};
      \node at (0,4) (T40) {$\boxed{4 \pipe 235 \pipe \emptyset}$};
      \node at (1,4) (T41) {$\boxed{24 \pipe 3 \pipe 5 }$};
      \node at (2,4) (T42) {$\boxed{ 24 \pipe 5 \pipe 3}$};
      \node at (3,4) (T43) {$\boxed{ 45 \pipe 2 \pipe 3}$};
      \node at (5,4) (T45) {$\boxed{24 \pipe 3 \pipe 2 }$};
      \node at (1,5) (T51) {$\boxed{24 \pipe 35 \pipe \emptyset}$};
      \node at (2,5) (T52) {$\boxed{245\pipe \emptyset \pipe 3 }$};
      \node at (3,5) (T53) {$\boxed{45 \pipe 23 \pipe \emptyset }$};
      \node at (2,6) (T62) {$\boxed{245 \pipe 3 \pipe \emptyset}$};
      \draw[->,thick]  (T62) -- (T51) node[midway,above,scale=0.75] {$1$};
      \draw[->,thick]  (T62) -- (T52) node[midway,right,scale=0.75] {$2$};
      \draw[->,thick]  (T62) -- (T53) node[midway,above,scale=0.75] {$\overline 1$};
      \draw[->,thick]  (T51.218) -- (T40.52) node[midway,above,scale=0.75] {$\overline 1$};
      \draw[->,thick]  (T51.232) -- (T40.38) node[midway,below,scale=0.75] {$1$};
      \draw[->,thick]  (T51) -- (T41) node[midway,right,scale=0.75] {$2$};
      \draw[->,thick]  (T52) -- (T42) node[midway,right,scale=0.75] {$1$};
      \draw[->,thick]  (T52) -- (T43) node[midway,above,scale=0.75] {$\overline 1$};
      \draw[->,thick]  (T53) -- (T43) node[midway,right,scale=0.75] {$2$};
      \draw[->,thick]  (T40) -- (T30) node[midway,left,scale=0.75] {$2$};
      \draw[->,thick]  (T41.218) -- (T30.52) node[midway,above,scale=0.75] {$\overline 1$};
      \draw[->,thick]  (T41.232) -- (T30.38) node[midway,below,scale=0.75] {$1$};
      \draw[->,thick]  (T41) -- (T31) node[midway,right,scale=0.75] {$2$};
      \draw[->,thick]  (T42) -- (T32) node[midway,right,scale=0.75] {$1$};
      \draw[->,thick]  (T42) -- (T33) node[midway,above,scale=0.75] {$\overline 1$};
      \draw[->,thick]  (T43) -- (T33) node[midway,right,scale=0.75] {$1$};
      \draw[->,thick]  (T43) -- (T34) node[midway,above,scale=0.75] {$2$};
      \draw[->,thick]  (T45.260) -- (T35.100) node[midway,left,scale=0.75] {$\overline 1$};
      \draw[->,thick]  (T45.280) -- (T35.80) node[midway,right,scale=0.75] {$1$};
      \draw[->,thick]  (T30) -- (T20) node[midway,left,scale=0.75] {$2$};
      \draw[->,thick]  (T31) -- (T20) node[midway,above,scale=0.75] {$\overline 1$};
      \draw[->,thick]  (T31) -- (T21) node[midway,right,scale=0.75] {$1$};
      \draw[->,thick]  (T32) -- (T21) node[midway,above,scale=0.75] {$2$};
      \draw[->,thick]  (T32.260) -- (T22.100) node[midway,left,scale=0.75] {$\overline 1$};
      \draw[->,thick]  (T32.280) -- (T22.80) node[midway,right,scale=0.75] {$1$};
      \draw[->,thick]  (T33) -- (T23) node[midway,right,scale=0.75] {$2$};
      \draw[->,thick]  (T34) -- (T23) node[midway,above,scale=0.75] {$1$};
      \draw[->,thick]  (T34) -- (T25) node[midway,above,scale=0.75] {$\overline 1$};
      \draw[->,thick]  (T35) -- (T25) node[midway,right,scale=0.75] {$2$};
      \draw[->,thick]  (T20) -- (T11) node[midway,above,scale=0.75] {$2$};
      \draw[->,thick]  (T21.322) -- (T12.128) node[midway,above,scale=0.75] {$\overline 1$};
      \draw[->,thick]  (T21.308) -- (T12.142) node[midway,below,scale=0.75] {$1$};
      \draw[->,thick]  (T22) -- (T12) node[midway,right,scale=0.75] {$2$};
      \draw[->,thick]  (T23.260) -- (T13.100) node[midway,left,scale=0.75] {$\overline 1$};
      \draw[->,thick]  (T23.280) -- (T13.80) node[midway,right,scale=0.75] {$1$};
      \draw[->,thick]  (T11.322) -- (T02.128) node[midway,above,scale=0.75] {$\overline 1$};
      \draw[->,thick]  (T11.308) -- (T02.142) node[midway,below,scale=0.75] {$1$};
      \draw[->,thick]  (T12) -- (T02) node[midway,right,scale=0.75] {$2$};
     \end{tikzpicture}
     }
  \end{center}
\caption{The $\q_3$-crystal graph of $\iRfpf_3(\pi)$ for $\pi=(1,4)(2,6)(3,5) \in \Ifpf_\ZZ$.}
\label{irfpf-fig}
\end{figure}

\subsection{Insertion algorithms}\label{ins-sect}

In this section, we describe two shifted variants of the \emph{Edelman-Greene correspondence},
which will be closely related to 
$\iR_n(\pi)$ and $\iRfpf_n(\pi)$.
We start with some well-known results from \cite{EG}.

\begin{definition}[\cite{EG}]
\label{eg-def}
Let $\pi \in S_\ZZ$ and 
$w=(w^1,w^2,\dots,w^n)\in\cR_n(\pi)$.
 Suppose concatenating the factors in $w$ gives the $m$-letter word
 $w_1w_2\cdots w_m$.
Let $\emptyset = T_0,T_1,\dots,T_m=P_\EG(w)$
be the sequence of (unshifted) tableaux in which
 $T_i$ for $i \in [m]$ is formed from $T_{i-1}$ by inserting $w_i$ as follows:
\ben
\item[] Start by inserting $w_i$ into the first row. 
At each stage, an entry $x$ is inserted into a row.
Let $y$ be the smallest entry in the row with $x \leq y$.
If no such entry $y$ exists then $x$ is added to the end of the row.
If $x=y$ then the current row is unchanged and $y+1$ is inserted into the next row.
If $x<y$ then $y$ is replaced by $x$ and $y$ is inserted into the next row.
\een
We call $P_\EG(w)$ the \emph{EG-insertion tableau} of $w$. The \emph{EG-recording tableau}
$Q_\EG(w)$ is the tableau with the same shape as $P_\EG(w)$
that contains $j$ in each of the boxes added by inserting the letters in 
the factor $w^j$, for each $j \in [n]$.
\end{definition}


\begin{example} We compute $P_\EG(w)$ and $Q_\EG(w)$ for $w=(4,23,2)$:
\[ 
\barr{rl}
\ytab{
  4 
}
  \ \leadsto\ 
\ytab{
  4 \\
  2 
}
   \ \leadsto\ 
\ytab{
  4 \\
  2 & 3
}
     \ \leadsto\ 
\ytab{
  4 \\
  3 \\
   2 & 3
} & = P_\EG(w)
  \earr
\quand
  \barr{rl}
\ytab{
  3 \\
  2 \\
  1 & 2
}& = Q_\EG(w).
\earr
\]
\end{example}

When $w=w_1w_2\cdots w_m \in \cR(\pi)$,
 we set 
$
P_\EG(w) := P_\EG((w_1,w_2,\dots,w_m))
$
and
$
Q_\EG(w) := Q_\EG((w_1,w_2,\dots,w_m)).
$
That is, we treat words as factorizations with all factors of length one.

\begin{theorem}[\cite{EG}]
If $\pi \in S_\ZZ$
then 
$w \mapsto (P_\EG(w), \Q_\EG(w))$
is a bijection 
\[
\cR_n(\pi)\xrightarrow{\sim}
\left\{\
\ba &\text{pairs $(P,Q)$ of tableaux of the same shape with} \\ 
&\text{$P$ increasing, $\row(P) \in \cR(\pi)$, and $Q \in \Tab_n(\ell(\pi))$}
\ea\
\right\}.
\]
Moreover, $Q_\EG(w)$ is a standard tableau if and only if all factors of $w$ have size one.
\end{theorem}

We turn to our first shifted analogue of Definition~\ref{eg-def}.

\begin{definition}[\cite{HKPWZZ,Marberg2019a,PatPyl2014}]
\label{o-eg-def}
Let $\pi \in I_\ZZ$
and $w=(w^1,\dots,w^n) \in \iR_n(\pi)$.
 Suppose concatenating the factors in $w$ gives the $m$-letter word
 $w_1w_2\cdots w_m$.
Let $\emptyset = T_0,T_1,\dots,T_m=\PO(w)$
be the sequence of shifted tableaux in which
 $T_i$ for $i \in [m]$ is formed from $T_{i-1}$ by inserting $w_i$ as follows:
\ben
\item[] Start by inserting $w_i$ into the first row. 
At each stage, an entry $x$ is inserted into a row or column.
Let $y$ be 
the smallest entry in the row or column with $x\leq y$.
If no such entry $y$ exists then add $x$ to the end of the row or column.
If $x=y$ then the row (respectively, column) is unchanged
and $y+1$ is inserted into the next row (respectively, column),
with one exception.
If $x < y$ then $y$ is replaced by $x$ and $y$ is inserted into the next row (respectively, column),
again with one exception.
The exceptions are that if $x$ is inserted into a row and $y$ is on the main diagonal,
then we insert $y+1$ (if $x=y$) or $y$ (if $x<y$) into the next column.

\een
If the orientation of insertion changes from rows to columns during this process,
then we say that $w_i$ is \emph{column-inserted}; otherwise, $w_i$ is \emph{row-inserted}.
We call $\PO(w)$ the \emph{orthogonal-EG-insertion tableau} of $w$. The \emph{orthogonal-EG-recording tableau}
$\QO(w)$ is the shifted tableau with the same shape as $\PO(w)$
that contains $j$ (respectively, $j'$) in each of the boxes added by a row-inserted (respectively, column-inserted) letter from 
the factor $w^j$, for each $j \in [n]$.
\end{definition}

This algorithm is called \emph{involution Coxeter-Knuth insertion} in \cite[\S4.3]{Marberg2019a}.

\begin{example} We compute $\PO(w)$ and $\QO(w)$ for $w=(4,23,2,1)$:
\[ 
\barr{rl}
\ytab{
  4 
}
  \ \leadsto\ 
\ytab{
  2 & 4
}
  \ \leadsto\ 
\ytab{
  \none & 4 \\
  2 & 3
}
  \ \leadsto\ 
\ytab{
  \none & 4 \\
   2 & 3 &4
}
  \ \leadsto\ 
\ytab{
  \none &  4 \\
  1 & 2 & 3 & 4
} & = \PO(w)
\earr
\quand
\barr{rl}
\ytab{
  \none & 2 \\
  1 & 2' & 3'& 4' 
} & = \QO(w).
\earr
\]
\end{example}

If $w=w_1\cdots w_m \in\iR(\pi)$ is an involution word
then we let 
$\PO(w) := \PO((w_1,\dots,w_m))$
and
$ \QO(w) := \QO((w_1,\dots,w_m)).$
Define $\ellhat(\pi)$ for $\pi \in I_\ZZ$
to be the length of any word in $\iR(\pi)$.
One has 
$\ellhat(\pi) = \tfrac{1}{2}(\ell(\pi) + \kappa(\pi))$
where $\kappa(\pi)$ is the number of 2-cycles of $\pi$  \cite[\S3]{Hultman2}.

\begin{theorem}[{\cite[Theorem 5.19]{HMP4}}]
\label{oeg-thm}
If $\pi \in I_\ZZ$ then
$w \mapsto (\PO(w), \QO(w))$
is a bijection 
\[
\iR_n(\pi)\xrightarrow{\sim}
\left\{\
\ba &\text{pairs $(P,Q)$ of shifted tableaux of the same shape with} \\ 
&\text{$P$ increasing, $\row(P) \in \iR(\pi)$, and $Q \in \STab_n(\ellhat(\pi))$}
\ea\
\right\}.
\]
Moreover,  $\QO(w)$ is a standard shifted tableau if and only if all factors of $w$ have size one.
\end{theorem}

Our second shifted analogue of Definition~\ref{eg-def} relates to fpf-involution words.

\begin{definition}[\cite{Marberg2019a}]
\label{sp-eg-def}
Let $ \pi \in \Ifpf_\ZZ$ and $w=(w^1,w^2,\dots,w^n)\in \iRfpf_n(\pi)$.
 Suppose concatenating the factors in $w$ gives the $m$-letter word
 $w_1w_2\cdots w_m$.
Let $\emptyset = T_0,T_1,\dots,T_m=\PSp(w)$
be the sequence of shifted tableaux in which
 $T_i$ for $i \in [m]$ is formed from $T_{i-1}$ by inserting $w_i$ as follows:
\ben
\item[] Start by inserting $w_i$ into the first row. 
At each stage, an entry $x$ is inserted into a row or column.
Let $y$ be 
the smallest entry in the row or column with $x\leq y$.
If no such entry $y$ exists then $x$ is added to the end of the row or column.
If $x=y$ then the current row (respectively, column) is unchanged,
and $y+1$ is inserted into the next row (respectively, column).
If $x < y$ then $y$ is replaced by $x$ and $y$ is inserted into the next row (respectively, column),
except
when $x$ is inserted into a row and $y$ is on the main diagonal.
In this case, if $y>x+1$
 then $y$ is replaced by $x$ and $y$ is inserted into the next column,
 while  if $y=x+1$ then
the current row is unchanged and $y+1$ is inserted into the next column.
\een
If the orientation of insertion changes from rows to columns during this process,
then we again say that $w_i$ is column-inserted; otherwise, $w_i$ is row-inserted.
We call $\PSp(w)$ the \emph{symplectic-EG-insertion tableau} of $w$. The \emph{symplectic-EG-recording tableau}
$\QSp(w)$ is the shifted tableau with the same shape as $\PSp(w)$
that contains $j$ (respectively, $j'$) in each of the boxes added by a row-inserted (respectively, column-inserted) letter from 
the factor $w^j$, for each $j \in [n]$.
\end{definition}

This algorithm is called \emph{FPF-involution Coxeter-Knuth insertion} in \cite[\S4.3]{Marberg2019a}.

\begin{example} 
\label{sp-eg-ex}
We compute $\PSp(w)$ and $\QSp(w)$ for $w=(4,23,12)$:
\[ 
\barr{rl}
\ytab{
  4 
}
  \ \leadsto\ 
\ytab{
  2 & 4
}
  \ \leadsto\ 
\ytab{
  \none & 4 \\
  2 & 3
}
  \ \leadsto\ 
\ytab{
  \none & 4 \\
  2 & 3 & 4
}
  \ \leadsto\ 
\ytab{
  \none & 4 & 5 \\
  2 & 3 & 4
}& = \PSp(w)
\earr
\quand
\barr{rl}
\ytab{
  \none & 2 & 3' \\
  1 & 2' & 3'
} & =\QSp(w).
\earr
\]
\end{example}

 \begin{remark}\label{insertion-remarks}
Consider the procedures adding $w_i$ to $T_{i-1}$ in Definitions~\ref{eg-def}, \ref{o-eg-def}, and \ref{sp-eg-def}.
Each iteration of these processes inserts an entry $x$ into a row or column of a tableau of the same shape as $T_{i-1}$.
The analysis of these algorithms in \cite{EG,HKPWZZ,Marberg2019a,PatPyl2014}
implies some additional properties:
 \begin{itemize}
\item Each intermediate (shifted) tableau is itself increasing. 
\item If the current row or column contains $x$ prior to insertion, then it must also contain $x+1$.
\item In both shifted algorithms, if $x$ is inserted into a row and the row's smallest entry $y$ is on the diagonal, 
then $x=y$ only if the diagonal position in the next row is occupied by $y+2$.
\item In the symplectic algorithm, 
if $x$ is inserted into a row and the row's smallest entry $y\geq x$ is on the main diagonal, 
then 
$y$ must be even, and either $x$ is even or $x=y-1$. The case with $x=y-1$ can only occur if the diagonal position in the next row is occupied by $y+2$.
\end{itemize}
\end{remark}

If $w=w_1\cdots w_m$ is an fpf-involution word then we define
$
\PSp(w) :=\PSp((w_1,\dots,w_m))
$
and
$ \QSp(w) := \QSp((w_1,\dots,w_m)).
$
Given 
 $\pi \in \Ifpf_\ZZ$,
 write $\ellfpf(\pi)$ 
for the length of any word in $\iRfpf(\pi)$.
To compute this number, choose
$m \in \NN$ 
with $\pi(i) = \minfpf(i)$ for all $ i > m$ and $i \leq -m$.
If 
$\sigma \in I_\ZZ$ has $\sigma(i) = \pi(i)$ for $-m < i \leq m$ and $\sigma(i) =i$ otherwise,
then  $\ellfpf(\pi) = \tfrac{1}{2}(\ell(\sigma) - m)$ \cite[\S2.3]{HMP5}.

\begin{theorem}[{\cite[Theorem 4.5]{Marberg2019a}}]
\label{speg-thm}
If $\pi \in \Ifpf_\ZZ$
then $w \mapsto (\PSp(w), \QSp(w))$
is a bijection 
\[
\iRfpf_n(\pi)\xrightarrow{\sim}
\left\{\
\ba &\text{pairs $(P,Q)$ of shifted tableaux of the same shape with} \\ 
&\text{$P$ increasing, $\row(P) \in \iRfpf(\pi)$, and $Q \in \STab_n(\ellfpf(\pi))$}
\ea\
\right\}.
\]
Moreover, $\QSp(w)$ is a standard shifted tableau if and only if all factors of $w$ have size one.
\end{theorem}

Ending this subsection, we include a few remarks about the relationship between our insertion algorithms and similar maps in the literature.
To start,  Edelman-Greene insertion is a special case of \emph{Hecke insertion} \cite{BKSTY}
and coincides with the \emph{RSK correspondence} \cite[\S7.1]{BumpSchilling} when restricted to (factorizations of) \emph{partial permutations}, i.e., words without repeated letters.

Orthogonal Edelman-Greene insertion, in turn, is a special case of \emph{shifted Hecke insertion} \cite{HKPWZZ,PatPyl2014}
and coincides with \emph{Sagan-Worley insertion} \cite{Sag87,Worley}
when restricted to (factorizations of) partial permutations.
 Symplectic Edelman-Greene insertion is a special case of \emph{symplectic Hecke insertion} \cite{Marberg2019a}
 and gives another generalization of Sagan-Worley insertion.

Haiman's mixed insertion    is also closely related to Sagan-Worley insertion
(and therefore to our shifted forms of  Edelman-Greene insertion).
We will say more about this connection in Section~\ref{reduction-sect}.

\subsection{Main results}\label{main-sect}

This section summarizes our main results, which relate our two shifted forms of Edelman-Greene insertion
to the abstract $\q_n$-crystals $\iR_n(\pi)$ and $\iRfpf_n(\pi)$.

As motivation and to provide some needed background, we start by reviewing two 
 theorems about Edelman-Greene insertion.
Let $\ck$ be the involution on 3-letter words that acts by
\[
acb \leftrightarrow cab,\quad bca \leftrightarrow bac,\quand
a(a+1)a\leftrightarrow (a+1)a(a+1),\quad\text{if }a<b<c,
\]
while fixing all other words.
Given a word $w$ and $i \in [\ell(w)-2]$,
define $\ck_i(w)$ to be the word obtained from $w$ by replacing the subword
$w_iw_{i+1}w_{i+2}$ by its image under $\ck$.
For integers $i \notin [\ell(w) - 2]$ set $\ck_i(w) := w$.
For example, $13541 = \ck_1(13541)=\ck_2(15341)  =  \ck_4(13541).$
\emph{Coxeter-Knuth equivalence} $\simCK$ is the transitive closure   of the reflexive relation
with $w \simCK \ck_i(w)$ for all $w$ and $i$.

\begin{theorem}[{\cite[\S6]{EG}}]
\label{simCK-thm}
If $v$ and $w$ are reduced words, 
then
 $v\simCK w$ if and only if $P_\EG(v) = P_\EG(w)$.
\end{theorem}

Morse and Schilling \cite{MorseSchilling} have shown  that $w \mapsto Q_\EG(w)$
is a crystal morphism:

\begin{theorem}[{\cite[Theorem 4.11]{MorseSchilling}; see \cite[\S10]{BumpSchilling}}]
\label{eg-thm}
Consider a permutation $\pi\in S_\ZZ$. Then:
\ben
\item[(a)] The full $\gl_n$-subcrystals of $\cR_n(\pi)$ are the subsets on which $P_\EG$ is constant.
\item[(b)] The map
 $Q_\EG$
is a quasi-isomorphism 
$\cR_n(\pi) \to \Tab_n(\ell(\pi))$. 
\een
\end{theorem}
 
This theorem has some noteworthy corollaries.
An abstract $\gl_n$-crystal is \emph{normal} (or \emph{Stembridge} \cite{Stembridge2003}) if each of its full subcrystals is isomorphic to a full $\gl_n$-subcrystal of
$\cW_n(m)$ for some $m$.

\begin{corollary}[\cite{MorseSchilling}]
\label{ms-cor1}
Each set of factorizations $\cR_n(\pi)$ for $\pi \in S_\ZZ$ is a normal $\gl_n$-crystal.
\end{corollary}

An element $b \in \cB$ in an abstract $\gl_n$-crystal
is a \emph{highest weight}
if $e_i(b) = 0$ for all $i \in [n-1]$.
A connected normal $\gl_n$-crystal $\cB$ has a unique highest weight $b$,
and the value of $\weight(b) $ for this element (after discarding trailing zeros) is always an integer partition \cite[\S4.4]{BumpSchilling}.

Edelman and Greene \cite{EG} proved that each Stanley symmetric function $F_\pi$ is a
nonnegative integer linear combination of Schur functions.
One can  interpret these coefficients as follows:

\begin{corollary}[\cite{MorseSchilling}]
\label{ms-cor2}
Suppose $\pi \in S_\ZZ$ and $n \geq \ell(\pi)$.
Then $F_\pi = \sum_\lambda a_{\pi\lambda} s_{\lambda}$ where $ a_{\pi\lambda}$ is the number of highest weights $w$ in the abstract $\gl_n$-crystal $\cR_n(\pi)$
with $\weight(w) = \lambda$.
\end{corollary}

Our main new results are ``orthogonal'' and ``symplectic'' analogues of Theorems~\ref{simCK-thm} and \ref{eg-thm}.
We start with the orthogonal case.
 Let $\ock$ denote the operator on words given by
\be\label{ck0-eq}
\ock(w_1w_2w_3w_4\cdots w_m) := w_2w_1 w_3w_4\cdots w_m
\ee
for any letters $w_i \in \ZZ$. If $\ell(w)\leq 1$ then set $\ock(w):=w$.
Define \emph{orthogonal Coxeter-Knuth equivalence} 
$\simICK$ to be
the transitive closure of $\simCK$ and the relation with $w\simICK \ock(w)$ for all words $w$.

\begin{theorem}\label{o-eg-thm}
If $v$ and $w$ are involution words, 
then $v\simICK w$ if and only if $\PO(v) = \PO(w)$.
\end{theorem}

This affirms \cite[Conjecture 5.24]{HMP4}, which is also stated as \cite[Conjecture 4.13]{Marberg2019a}.

\begin{theorem}\label{main-thm1}
Consider an involution $\pi\in I_\ZZ$. Then:
\ben
\item[(a)] The full $\q_n$-subcrystals of $\iR_n(\pi)$ are the subsets on which $\PO$ is constant.

\item[(b)] The map
$\QO$
is a quasi-isomorphism 
$\iR_n(\pi) \to \STab_n(\ellhat(\pi))$. 
\een
\end{theorem}

We give the proofs of both theorems in Section~\ref{proofs-o-sect}.
These results lead to interesting 
analogues of 
Corollaries~\ref{ms-cor1} and \ref{ms-cor2}.
Following \cite{AssafOguz,GHPS}, we say that
an abstract $\q_n$-crystal is \emph{normal} if each of its full subcrystals is isomorphic to a full $\q_n$-subcrystal of
$\cW_n(m) \cong (\BB_n)^{\otimes m}$ for some $m$.

\begin{corollary}\label{normal-cor1}
Each set of factorizations $\iR_n(\pi)$ for $\pi \in I_\ZZ$ is a normal $\q_n$-crystal.
\end{corollary}

The notion of a \emph{highest weight} of an abstract $\q_n$-crystal is 
 subtler than for $\gl_n$-crystals; 
see \cite[\S1.3]{GJKKK} for the precise definition.
Each connected normal $\q_n$-crystal $\cB$ has a unique highest weight $b$,
and it holds that  $\weight(b)$ is a strict partition and 
$\ch(\cB)=P_{\weight(b)}(x_1,x_2,\dots,x_n)$
 \cite[Theorem 1.14]{GJKKK}.

It is shown in \cite{HMP4} that 
the symmetric functions $\iF_\pi$ from Remark~\ref{iF-rmk} are $\NN$-linear combinations of Schur $P$-functions.
We can give a new interpretation of these coefficients:

\begin{corollary}\label{invF-cor1}
Suppose $\pi \in I_\ZZ$ and $n \geq \ellhat(\pi)$.
Then $\iF_\pi = \sum_\lambda b_{\pi\lambda} P_{\lambda}$ where $ b_{\pi\lambda}$ is the number of highest weights $w$ in the abstract $\q_n$-crystal $\iR_n(\pi)$
with $\weight(w) = \lambda$.
\end{corollary}

Parallel results hold in the symplectic case.
For a word $w=w_1w_2\cdots w_m$ with $m\geq 2$, define 
\be\label{spck-eq}
\spck(w) := \begin{cases} w_1(w_1\mp 1) w_3w_4\cdots w_m &\text{if $w_2 = w_1 \pm 1$} \\
w_2 w_1 w_3w_4\cdots w_m &\text{if $w_1- w_2$ is even} \\
w &\text{otherwise}.
\end{cases}
\ee
For words with $\ell(w) \leq 1$ we set $\spck(w) :=w$.
Define
\emph{symplectic Coxeter-Knuth equivalence}
 $\simFCK$ to be the transitive closure of 
$\simCK$ and the relation with $w \simFCK \spck(w)$ for all words $w$.

Hiroshima \cite[Theorem 4.4]{Hiroshima} has recently proved the following analogue of Theorem~\ref{o-eg-thm}.
His methods consist of a self-contained but very complicated analysis of 
symplectic-EG insertion.
We give a  conceptually simpler (but less self-contained) alternate proof in Section~\ref{proofs-sp-sect}.

\begin{theorem}[{\cite{Hiroshima}}]
\label{sp-eg-thm}
If $v$, $w$ are fpf-involution words,
then $v\simFCK w$ if and only if $\PSp(v) = \PSp(w)$.
\end{theorem}

Next, we have a symplectic version of Theorem~\ref{main-thm1}.
This result implies \cite[Conjecture 5.1]{Hiroshima}.

\begin{theorem}\label{main-thm2}
Consider an involution $\pi\in \Ifpf_\ZZ$. Then:
\ben
\item[(a)] The full $\q_n$-subcrystals of $\iRfpf_n(\pi)$ are
the subsets on which $\PSp$ is constant.

\item[(b)]
The map
$\QSp$
is a quasi-isomorphism 
$\iRfpf_n(\pi) \to \STab_n(\ellfpf(\pi))$. 
\een
\end{theorem}

The 
proof is also in Section~\ref{proofs-sp-sect}.
There are again a few corollaries worth noting.
 
\begin{corollary}\label{normal-cor2}
Each set of factorizations $\iRfpf_n(\pi)$ for $\pi \in \Ifpf_\ZZ$ is a normal $\q_n$-crystal.
\end{corollary}

It is shown in \cite{HMP5} that 
the symmetric functions $\Ffpf_\pi$ 
from Remark~\ref{Ffpf-rmk} 
are $\NN$-linear combinations of Schur $P$-functions.
We can give a new interpretation of these coefficients:

\begin{corollary}\label{invF-cor2}
Suppose $\pi \in \Ifpf_\ZZ$ and $n \geq \ellfpf(\pi)$.
Then $\Ffpf_\pi = \sum_\lambda c_{\pi\lambda} P_{\lambda}$ where $ c_{\pi\lambda}$ is the number of highest weights $w$ in the abstract $\q_n$-crystal $\iRfpf_n(\pi)$
with $\weight(w) = \lambda$.
\end{corollary}

\section{Bumping operators}\label{little-sect}

Rather than attempting to prove our main theorems directly, our strategy is to construct certain $\q_n$-crystal isomorphisms
that translate each result to an equivalent but more tractable setting.
These isomorphisms will be formed as compositions of the \emph{involution Little bumping operators}
introduced in \cite{HMP3}. We study the properties of these operators in this section.

\subsection{Little bumps}

We begin by reviewing some ``classical'' results about Little bumping operators
from \cite{HamakerYoung,Little}. For related background, see also \cite{LamShim}, whose notational conventions we follow.

Let $w$ be a word, choose an index $i \in[\ell(w)]$,
and define 
$\del_i(w)$ to be the subword of $w$ formed by omitting the $i$th letter.
The pair $(w,i)$ is a \emph{$\pi$-marked word} 
for a permutation $\pi\in S_\ZZ$ if
 $\del_i(w) \in \cR(\pi)$.
A marked word   $(w,i)$ is \emph{reduced} if $w$ is a reduced word.

\begin{lemma}[{\cite[Lemma 21]{LamShim}}]
\label{lamshim-lem}
If $(w,i)$ is an unreduced $\pi$-marked word, 
then 
there exists a unique index $i \neq j \in [\ell(w)]$
such that $(w,j)$ is a $\pi$-marked word.
\end{lemma}

\begin{definition}
Fix $\pi \in S_\ZZ$ and suppose $(w,i)$ is a $\pi$-marked word of length $m$.
If $w$ is reduced then let $j=i$, and otherwise 
let $ j \in [m]$ be the unique index with $i\neq j$ such that $(w,j)$ is a $\pi$-marked word.
Then define $\push(w,i):=(v, j)$ where 
 $v  := w_1\cdots w_{j-1} (w_j + 1) w_{j+1}\cdots w_m.$
 \end{definition}
 
If $(w,i)$ is a $\pi$-marked word then $\push^N(w,i)$
 is reduced for some sufficiently large $N>0$ \cite[Lemma 5]{Little}.
The
Strong Exchange Condition \cite[Theorem 5.8]{Humphreys},
moreover,
implies that if
$\pi \in S_\ZZ$ and
$w$ is a fixed reduced word, then
  $(w,i)$ is a $\pi$-marked word for at most one choice of $i$.

\begin{definition}[{\cite[\S5]{Little}}]
For each $\pi \in S_\ZZ$, 
the \emph{Little bumping operator} $\fkb_{\pi}$ acts on reduced words $w$ as follows.
If there exists an index $i$ such that
 $(w,i)$ is a $\pi$-marked word
 and 
 $N>0$ is minimal such that $\push^N(w,i)=: (v,j)$
is reduced, then
$\fkb_{\pi}(w) := v$.
Otherwise, $\fkb_{\pi}(w) := w$.
\end{definition}

Our definition of $\fkb_\pi$ gives the inverse of operator described in \cite[Algorithm 2]{Little}.
In \cite{Little}, bumping operators act by decrementing letters, but here it is convenient to adopt the opposite convention.

The \emph{descent set} of a word $w$ is $\Des(w) := \{ i \in [\ell(w)-1] : w_i > w_{i+1}\}$.
The following theorem gathers together several properties of $\fkb_\pi$ from papers of Little \cite{Little} and of Hamaker and Young \cite{HamakerYoung}.

\begin{theorem}[\cite{HamakerYoung,Little}]
\label{little-thm}
Let $\pi,\sigma \in S_\ZZ$ and $w \in \cR(\sigma)$. 
\ben
\item[(a)] The operator $\fkb_\pi$ is a bijection $\bigsqcup_{\tau \in S_\ZZ} \cR(\tau) \to \bigsqcup_{\tau \in S_\ZZ} \cR(\tau)$.
\item[(b)] It holds that $\Des(\fkb_{\pi}(w)) = \Des(w)$.
\item[(c)] For all $i\in \PP$ it holds that $\ck_i(\fkb_\pi(w)) = \fkb_\pi(\ck_i(w))$.
\item[(d)] It holds that $Q_\EG(\fkb_\pi(w)) = Q_\EG(w)$.
\een
\end{theorem} 

\begin{proof}
Part (a) is a weaker form of \cite[Lemma 7]{Little}.
Parts (b), (c), and (d) 
are equivalent to \cite[Corollary 1, Lemma 3, and Proposition 1]{HamakerYoung}, respectively.
\end{proof}

Since $\fkb_\pi$ preserve descents, 
it extends to an operator on reduced factorizations
as follows. Given 
$\pi, \sigma \in S_\ZZ$ and
$w =(w^1,w^2,\dots,w^n) \in \cR_n(\sigma)$,
define
$
 \fkb_\pi(w) = (v^1,v^2,\dots,v^n)
$ 
where the words $v^i$ are such that $\fkb_\pi(w^1w^2\cdots w^n) = v^1v^2\cdots v^n$
and $\ell(v^i) = \ell(w^i)$.
Since $\Des(v^i) = \Des(w^i)$,
each $v^i$ is again strictly increasing.
The following theorem appears to be new.

\begin{theorem}\label{qi-thm1}
Let $\pi \in S_\ZZ$. Then
$\fkb_\pi$ is an isomorphism of $\gl_n$-crystals
$\ds\bigsqcup_{\sigma \in S_\ZZ} \cR_n(\sigma) \to \bigsqcup_{\sigma \in S_\ZZ} \cR_n(\sigma)$.
\end{theorem}

\begin{proof}
Suppose $\cB$ is a full $\gl_n$-subcrystal of $\cX : = \bigsqcup_{\sigma \in S_\ZZ} \cR_n(\sigma)$.
Theorems~\ref{simCK-thm} and \ref{eg-thm} imply that $\cB$ consists of all 
$n$-fold increasing factorizations of reduced words in a
single Coxeter-Knuth equivalence class.
 Theorem~\ref{little-thm}(c)
implies that $\cC := \fkb_\pi(\cB)$ is another full subcrystal of $\cX$.
Let $\lambda$ be the partition that is 
the common shape of the EG-recording tableaux for all factorizations in $\cB$ and $\cC$.
By Theorem~\ref{little-thm}(d), 
the map $\cB \xrightarrow{\fkb_\pi} \cC \xrightarrow{Q_\EG} \Tab_n(\lambda)$
coincides with $\cB \xrightarrow{Q_\EG}\Tab_n(\lambda)$.
Since both
$\cB \xrightarrow{Q_\EG}\Tab_n(\lambda)$
and
$ \cC \xrightarrow{Q_\EG} \Tab_n(\lambda)$
are crystal isomorphisms
by Theorem~\ref{eg-thm},
we conclude that $\fkb_\pi : \cB \to \cC$ is a crystal isomorphism.
By Theorem~\ref{little-thm}(a), $\fkb_\pi$ is an isomorphism $ \cX \to \cX$.
\end{proof}

\subsection{Involution Little bumps}

Our next goal is 
to prove analogues of Theorems~\ref{little-thm} and \ref{qi-thm1} for involution words.
 Fix $\pi \in I_\ZZ$ and recall that $\cA(\pi)\subset S_\ZZ$ is the set with
$\iR(\pi) = \bigsqcup_{\sigma \in \cA(\pi)} \cR(\sigma)$.
A \emph{$\pi$-marked involution word} is a
pair $(w,i)$ in which $w$ is word and $i$ is an index such that $\del_i(w) \in \iR(\pi)$.
Equivalently, this is just an $\alpha$-marked word for some $\alpha \in \cA(\pi)$.
A $\pi$-marked involution word $(w,i)$ is \emph{inv-reduced} if $w$ is  an involution word
for some element of $I_\ZZ$.

\begin{lemma}[{\cite[Lemma 3.34]{HMP3}}]
\label{toggle-lem}
If $(w,i)$ is a $\pi$-marked involution word
that is not inv-reduced,
then there is a unique index $i\neq j \in [\ell(w)]$ such that $(w,j)$ is also a $\pi$-marked involution word.
\end{lemma}

\begin{remark}
A $\pi$-marked involution word $(w,i)$ may fail to be inv-reduced in two ways:
either $w$ is a reduced word that is not an involution word, or $w$ is not reduced.
In the latter case,
the index $j$ identified in Lemma~\ref{toggle-lem} is necessarily the same as the one in Lemma~\ref{lamshim-lem}.
\end{remark}

\begin{definition}
Let $(w,i)$ be a $\pi$-marked involution word of length $m$.
If $(w,i)$ is inv-reduced then let $j=i$,
and otherwise let $i\neq j \in [m]$ be 
the unique index  such that $(w,j)$ is 
a $\pi$-marked involution word.
Then define $\ipush(w,i):=(v, j)$
where $v := w_1\cdots w_{j-1} (w_j + 1) w_{j+1}\cdots w_m.$
\end{definition}

As with the earlier $\push$ operator,
one can show that if 
$(w,i)$ is a $\pi$-marked involution word then 
$\ipush(w,i)$ is inv-reduced for some sufficiently large $N>0$
\cite[Lemma 3.37]{HMP3}.
By \cite[Theorem 3.4]{HMP3} (see also \cite[Theorem 2.8]{Hultman3}),
it holds that 
if $\pi \in I_\ZZ$ and $w$ is a fixed involution word, then there exists at most one index $i$ 
such that $(w,i)$ is a $\pi$-marked involution word.

\begin{definition}[{\cite[\S3.3]{HMP3}}]
The \emph{involution Little bumping operator} $\ifkb_{\pi}$ of $\pi \in I_\ZZ$ acts on involution words $w$ as follows.
If there exists $i$ such that 
 $(w,i)$ is a $\pi$-marked involution word
and $N>0$ is minimal such that 
$ \ipush^N(w,i) =: (v,j)$ is inv-reduced, then
$\ifkb_{\pi}(w) := v$.
Otherwise, $\ifkb_{\pi}(w) := w$.
\end{definition}

The map $\ifkb_\pi$ is the inverse of the operator 
${\hat\cB}_\pi$  in 
\cite[Theorem 3.40]{HMP3}.

\begin{example}\label{ibb-ex}
Let $\pi = (2,5) \in I_\ZZ$, $\sigma = (1,5) \in I_\ZZ$,
and $w=2134 \in \iR(\sigma)$. 
Then $234 \in \iR(\pi)$, so to compute $\ifkb_\pi(w)$ we must find the minimal $N >0$ such that 
 $\ipush^N(2134,2)$ is inv-reduced. The following pictures show that $N=4$:
\[
\arraycolsep=0.7pt\def\arraystretch{0.7}
\barr{cccc}
 \cdot & \cdot & \cdot & \cdot \\
 \cdot & \cdot & \cdot & \times \\
 \cdot & \cdot & \times & \cdot \\
 \times & \cdot & \cdot & \cdot \\
 \cdot & \ttimes & \cdot & \cdot \\
 \hline \\[-6pt]
 2 & 1 & 3 & 4  
\earr
\ \
\xrightarrow{\ipush}
\ \
\barr{cccc}
 \cdot & \cdot & \cdot & \cdot \\
 \cdot & \cdot & \cdot & \times \\
 \cdot & \cdot & \times & \cdot \\
 \times & \ttimes & \cdot & \cdot \\
 \cdot & \cdot & \cdot & \cdot \\
 \hline \\[-6pt]
 2 & 2 & 3 & 4  
\earr
\ \
\xrightarrow{\ipush}
\ \
\barr{cccc}
 \cdot & \cdot & \cdot & \cdot \\
 \cdot & \cdot & \cdot & \times \\
 \ttimes & \cdot & \times & \cdot \\
 \cdot & \times & \cdot & \cdot \\
 \cdot & \cdot & \cdot & \cdot \\
 \hline \\[-6pt]
 3 & 2 & 3 & 4  
\earr
\ \
\xrightarrow{\ipush}
\ \
\barr{cccc}
 \cdot & \cdot & \cdot & \cdot \\
 \cdot & \cdot & \ttimes & \times \\
 \times & \cdot & \cdot & \cdot \\
 \cdot & \times & \cdot & \cdot \\
 \cdot & \cdot & \cdot & \cdot \\
 \hline \\[-6pt]
 3 & 2 & 4 & 4  
\earr
\ \
\xrightarrow{\ipush}
\ \
\barr{cccc}
 \cdot & \cdot & \cdot & \ttimes \\
 \cdot & \cdot & \times & \cdot \\
 \times & \cdot & \cdot & \cdot \\
 \cdot & \times & \cdot & \cdot \\
 \cdot & \cdot & \cdot & \cdot \\
 \hline \\[-6pt]
 3 & 2 & 4 & 5
\earr
\]
The third marked word is reduced but not inv-reduced, as $3234 \isim 2334$.
Thus $\ifkb_\pi(w) = 3245$. 
\end{example}

\begin{lemma}\label{easy-lem}
Let $\pi \in I_\ZZ$.
For any involution word $w$,
there is a finite sequence of elements $\alpha_1,\alpha_2,\dots,\alpha_l \in \cA(\pi)$
such that $\ifkb_\pi(w) = \fkb_{\alpha_l}\cdots \fkb_{\alpha_2}\fkb_{\alpha_1}(w)$.
Moreover, this sequence is the same for all involution words in a single Coxeter-Knuth equivalence class.
\end{lemma}

In Example~\ref{ibb-ex} 
where $\pi = (2,5)$ and $w=2134$ 
we have $\ifkb_\pi(w) =\fkb_{\alpha_2}\fkb_{\alpha_1}(w)$
for $\alpha_1 = s_2s_3s_4 = (2,3,4,5)$
and $\alpha_2 = s_3s_2s_4 = (2,4,5,3)$.

\begin{proof}
The first assertion holds by the remark after Lemma~\ref{toggle-lem}.
Write $\lessdot $ for the covering relation in the Bruhat order on $S_\ZZ$.
Fix an involution word $w$ and let $\alpha_1,\alpha_2,\dots,\alpha_l \in \cA(\pi)$
be the sequence with
$\ifkb_\pi(w) = \fkb_{\alpha_l}\cdots \fkb_{\alpha_2}\fkb_{\alpha_1}(w)$.
Define $\sigma_i \in S_\ZZ$ such that 
$w \in \cR(\sigma_1)$ and $\fkb_{\alpha_{i-1}}\cdots \fkb_{\alpha_2}\fkb_{\alpha_1}(w) \in \cR(\sigma_i)$ for $1<i\leq l$.
Lemma~\ref{toggle-lem} implies that
$\alpha_1$ is the unique element of $\cA(\pi)$ with $\alpha_1 \lessdot \sigma_1$
while $\alpha_i$ for $i>1$ is the unique element of $\cA(\pi)\setminus\{\alpha_{i-1}\}$
with $\alpha_i \lessdot \sigma_i$.   
It follows that any word $v \in \cR(\sigma_1)$
with $\fkb_{\alpha_{i-1}}\cdots \fkb_{\alpha_2}\fkb_{\alpha_1}(v) \in \cR(\sigma_i)$ for all $1<i \leq l$
also has $\ifkb_\pi(v) =  \fkb_{\alpha_l}\cdots \fkb_{\alpha_2}\fkb_{\alpha_1}(v)$.
These conditions hold
if $v \simCK w$ by Theorem~\ref{little-thm}.
\end{proof}

The following lemma collects several observations
used in the proof of Theorem~\ref{o-little-thm},
which gives an analogue of Theorem~\ref{little-thm}
 for the operator $\ifkb_\pi$

\begin{lemma}\label{16lem}
Let $\mu$ be a strict partition.

\begin{itemize}
\item[(1)] Suppose we know that $w$ is the row (respectively, column) reading word of an increasing shifted tableau of shape $\mu$.
Then we can recover $\mu$ from $\ell(w)$ and $\Des(w)$.

\item[(2)] For each strict partition $\mu$, there are indices $i_1,i_2,\dots,i_p \in \PP$ (depending only on $\mu$)
such that
$\ck_{i_1}\ck_{i_2}\cdots \ck_{i_p}(\row(T)) = \col(T)$
and
$\ck_{i_p}\cdots \ck_{i_2}\ck_{i_1}(\col(T)) = \row(T)$
 for all increasing shifted tableaux $T$ of shape $\mu$.
We define
\be\label{tau-eq} \tau_\mu^\col := \ck_{i_1}\ck_{i_2}\cdots \ck_{i_p}
\quand
\tau_\mu^\row := \ck_{i_p}\cdots \ck_{i_2}\ck_{i_1}.
\ee
\end{itemize}
Now suppose $u=u_1u_2\cdots u_{n}$ is a strictly increasing word with $n>0$ and $ x \in \ZZ$.
\begin{itemize}
\item[(3)] 
Assume $ux$ is a reduced word and $ u_{n}> x$.
If $i$ is minimal such that  $y:=u_i\geq x$
then
\[
\ck_1\ck_2\cdots \ck_{n-1}(ux) = \begin{cases}
(y+1)\cdot u_1u_2\cdots u_{n}&\text{if $x=y$} \\
y \cdot u_1\cdots u_{i-1} \cdot x\cdot u_{i+1}\cdots u_{n} &\text{if $x<y$}.
\end{cases}
\]
In particular, if $x<u_1$ then $\ck_1\ck_2\cdots \ck_{n-1}(ux) = u_1 \cdot x \cdot u_2\cdots u_n$.

\item[(4)] If $ux$ is an involution word (respectively, fpf-involution word)
then  
 $x<u_1$ if and only if 
 a descent in position $n$
 occurs in $\ck_{n-2}\cdots\ck_2 \ck_1\ock(ux)$
 (respectively, $\ck_{n-2}\cdots\ck_2 \ck_1\spck(ux)$).

\end{itemize}
Finally, let $v=v_1v_2\cdots v_{n}$ be a strictly decreasing word with $n>0$ and $ x \in \ZZ$.
\begin{itemize}

\item[(5)] Assume $xv$ is a reduced word
and
 $x<v_1$.
If $i$ is maximal such that  $y:=v_i\geq x$
then
\[
\ck_{n-1}\cdots \ck_2\ck_1(xv) 
=
 \begin{cases}
v_1v_2\cdots v_{n}\cdot (y+1)&\text{if $x=y$} \\
v_1\cdots v_{i-1} \cdot x\cdot v_{i+1}\cdots v_{n}\cdot y &\text{if $x<y$}.
\end{cases}
\]

\item[(6)] If $w:=v_1  x v_2 \cdots v_n$ is reduced  and $v_1<x$,
then 
$
\ck_{n-1}\cdots \ck_2\ck_1(w) = vx.
$
\end{itemize}
\end{lemma}

\begin{proof}
The derivation of each of these assertions is a straightforward exercise.
The claims in part (2) follow from the discussion in \cite[\S2.2]{Marberg2019a}.
Part (4) is a consequence of Theorems~\ref{isim-thm} and \ref{fsim-thm}.
\end{proof}

\begin{theorem}\label{o-little-thm}
Let $\pi,\sigma \in I_\ZZ$ and $w \in \iR(\sigma)$. Then:
\ben
\item[(a)] The operator $\ifkb_\pi$ is a bijection $\bigsqcup_{z \in I_\ZZ} \iR(z) \to \bigsqcup_{z \in I_\ZZ} \iR(z)$.
\item[(b)] It holds that $\Des(\ifkb_{\pi}(w)) = \Des(w)$.
\item[(c)] It holds that 
$\ock(\ifkb_\pi(w)) = \ifkb_\pi(\ock(w))$ and $\ck_i(\ifkb_\pi(w)) = \ifkb_\pi(\ck_i(w))$ for all $i>0$.
\item[(d)] It holds that $\QO(\ifkb_\pi(w)) = \QO(w)$.
\een
\end{theorem} 

\begin{proof}
Part (a) is equivalent to \cite[Theorem 3.40]{HMP3}.
Part (b) is immediate from Theorem~\ref{little-thm}(b) and Lemma~\ref{easy-lem}.

When $i$ is a positive integer, the identity in part (c) follows 
from 
Theorem~\ref{little-thm}(c) and Lemma~\ref{easy-lem}.
It remains to show that  $\ifkb_{\pi}$ commutes with $\ock$.
Write $w=w_1w_2\cdots w_m$ where $m\geq 2$.
For the rest of this paragraph, whenever $a \in \PP$, we  define $\underline a := 3-a$ if $a \in\{1,2\}$
and otherwise set $\underline a:=a$.
It follows from Theorem~\ref{isim-thm} that $\del_a(w) \in \iR(\pi)$ if and only if $\del_{\underline a}(\ock(w)) \in \iR(\pi)$. If such an integer $a$ exists and
$\push(w,a) = (v,b)$, then  Theorem~\ref{isim-thm} and the uniqueness asserted 
in Lemma~\ref{toggle-lem} imply
that
$\push(\ock(w),\underline a) = (\ock(v), \underline b)$.
In turn, if $(v,b)$ is not inv-reduced and $\push^2(w,a) = \push(v,b) = (u,c)$
then it follows by the same reasoning 
that $\push^2(\ock(w),\underline a) = \push(\ock(v), \underline b)  = (\ock(u),\underline c)$.
Continuing in this way, we conclude that $\ifkb_\pi(\ock(w)) =\ock(\ifkb_\pi(w))$, which is what we needed to prove part (c).

Thus $\ifkb_\pi$ preserves descents and 
commutes with $\ock$ and $\ck_i$ for  $i> 0$.
Hence, to prove part (d), it suffices to give an algorithm to compute $\QO(w)$
that only relies on the descent sets of words in the $\simICK$-equivalence class of 
$w$ as inputs.
Such an algorithm will produce the same output for $\ifkb_\pi(w)$ as for $w$ by parts (b) and (c),
showing that $\QO(\ifkb_\pi(w)) = \QO(w)$.

Suppose  $w$ has the form $w=w_1w_2\cdots w_{n+1}$ where $n \in \NN$.
Specifically, we will describe an inductive procedure to
produce
the tableau $\QO(w)$ along with an operator $\fk p$ such that $\fk p(w) = \row(\PO(w))$.
Each step in this algorithm will depend only on $n$ and the descent sets of words 
obtained by applying certain fixed sequences orthogonal Coxeter-Knuth operators to 
$w$.
To follow our discussion it will be helpful to consult
Examples~\ref{worked-ex1}, \ref{worked-ex2}, and \ref{worked-ex3},
which illustrate our notation in some particular cases.

 If $n=0$ then we always have $\QO(w) = 
 \begin{ytableau}
 1
 \end{ytableau} 
 $
 and  $\fk p $ is given by the identity operator.
 Assume $n>0$ and let 
 $P = \PO(w_1w_2\cdots w_n)$
 and
$Q = \QO(w_1w_2\cdots w_n)$.
By induction, we may assume that $Q$ is given and that we have an operator $ {\fk o}$ such that ${\fk o}(w) = \row(P)w_{n+1}$.

Let $\mu=(\mu_1>\mu_2>\dots>\mu_r>0)$ be the strict partition of $n$ that is
the shape of $Q$. By Lemma~\ref{16lem}(1), this partition can be computed from $\Des(\fk o(w))$.
For each $i \in [r]$, let 
\[
d_i := n - \mu_1 - \mu_2 - \dots - \mu_{i-1}
\quand
\rho_i := \ck_{d_i+1} \ck_{d_i+2}\cdots \ck_{n - 1}.
\]
Also set 
\[ \rho_{r+1} :=\begin{cases} \rho_r &\text{if }\mu_r = 1 \\
 \ck_{\mu_r - 2}\cdots \ck_2\ck_1\ock\circ  \rho_{r} &\text{if }\mu_r \geq 2.\end{cases}\]
It follows from Lemma~\ref{16lem}(3)
that if  $d_i \notin \Des(\rho_i\circ {\fk o}(w))$
 for some $i\in[r]$, and $i$ is the minimal index with this property,
then we have
$\row(\PO(w))= \fk p(w)$ for the operator \[\fk p := \rho_i\circ {\fk o}\] and
$\QO(w)$ is formed from $Q$ by adding $n+1$ to box $(i, i+\mu_{i})$.

Assume $d_i \in \Des(\rho_i\circ {\fk o}(w))$
for all $i \in [r]$.
By Lemma~\ref{16lem}(3)-(4), 
adding $w_{n+1}$ to $P$ ends in row-insertion if and only if 
$d_r \notin\Des\( \rho_{r+1}\circ {\fk o}(w)\),$
 in which case  $\row(\PO(w))= \fk p(w)$ for  
 \[\fk p := \rho_{r+1}\circ {\fk o}\]
 and
 $\QO(w)$ is formed from $Q$ by adding $n+1$ to the diagonal box $(r+1, r+1)$.
 Otherwise, 
  adding $w_{n+1}$ to $P$ must end in column-insertion.
  Assume we are in this case.
For each $i \in [r-1]$,
let $\triangle(i) :=1 + 2 + \dots +i$ and define
\[
\psi_i := \ck_1\ck_2\cdots \ck_{d_{i} + i - 2}
\quand
\phi_i :=  (\ck_{ \triangle(i)+1}\ck_{\triangle(i)+ 2}\cdots \ck_{ \triangle(i) + r  - 1})^i.
\]
We compose these operators to define 
\[\ba
\ROWINSERT &:=  \psi_{r} \cdots \psi_2 \psi_1, \\
\REVERSE  &:= (\ck_{r-1} \cdots \ck_2\ck_1\ock)^r, \\
\COLUMNINSERT &:=  \phi_{r-1} \cdots \phi_2 \phi_1.
\ea\]
Finally, let $\nu$ be the strict partition formed by subtracting one from each part of $\mu$,
suppose $i_1,i_2,\dots,i_p \in \NN$ are such that $\tau_\nu^\col =  \ck_{i_1}\ck_{i_2}\cdots \ck_{i_p}$, and define
 \[
\REORIENT := \ck_{r + 1 + i_1}\ck_{r + 1 +  i_2}\cdots \ck_{r + 1 +  i_p} .
\]

Here is what these operators do.
Consider the insertion process  outlined in Definition~\ref{o-eg-def} that adds $w_{n+1}$ to $P$.
At exactly one iteration in this process,
a number $x$ is inserted into a row (of an increasing shifted tableau $T$ with shape $\mu$)
whose smallest entry $y$ has $x\leq y$, and the next iteration proceeds by inserting either $y$ or $y+1$ into the next column.
Suppose this diagonal bump happens in row $j \in [r]$. 
Let $t_1t_2\cdots t_r$ be the main diagonal of $T$ read bottom-to-top, so $x<y=t_j$.
Write $U$ for the tableau of shape $\nu$ formed by removing the main diagonal of $T$
and let 
\be\label{tildetilde-eq}  \tilde x=\tilde{\tilde x} = \begin{cases} x & \text{if }x <y \\ y&\text{if }x=y\end{cases}
\quand
\tilde y=\tilde{\tilde y} = \begin{cases} y & \text{if }x <y \\ y+1&\text{if }x=y.\end{cases}
\ee
(We introduce the symbols $\tilde{\tilde x}$ and $\tilde{\tilde y}$ in order to
reuse the following identities in the symplectic case, where these variables 
will have a different meaning.)
Using Lemma~\ref{16lem}(3), we compute that
\[
\ba
{\fk o}(w) &= \row(P)w_{n+1},\\
\ROWINSERT \circ {\fk o}(w) &= t_rt_{r-1}  \cdots t_{j+1} \cdot \tilde y \cdot \tilde x \cdot t_{j-1}t_{j-2} \cdots t_1 \cdot \row(U), \\
\REVERSE\circ \ROWINSERT \circ {\fk o}(w) &= t_1t_2  \cdots t_{j-1} \cdot \tilde{\tilde x} \cdot \tilde{\tilde y} \cdot t_{j+1}t_{j+2} \cdots t_r \cdot \row(U), \\
\REORIENT\circ \REVERSE\circ \ROWINSERT \circ {\fk o}(w) &= t_1t_2  \cdots t_{j-1} \cdot \tilde{\tilde x} \cdot \tilde{\tilde y} \cdot t_{j+1} t_{j+2}\cdots t_r \cdot \col(U).
\ea
\]
The insertion process adding $w_{n+1}$ to $P$ lasts for at least  $r+1$ iterations, and every iteration after the $j$th 
proceeds by column insertion.
Suppose $z$ is the number and $V$ is the shifted tableau of shape $\mu$ such that 
iteration $r+1$ inserts
$z$ into column $r+1$ of $V$.
If we write $v_1v_2\cdots v_n =\col(V)$
and set $\triangle := 1 + 2 + \dots + r$
then
it follows from Lemma~\ref{16lem}(5)-(6) 
 that
\be\label{vVv-eq}
\COLUMNINSERT\circ \REORIENT\circ \REVERSE\circ \ROWINSERT \circ {\fk o}(w) = v_1v_2\cdots v_{\triangle} \cdot z\cdot v_{\triangle+1}v_{\triangle+2}\cdots v_n
.\ee
See Example~\ref{worked-ex3} for a demonstration of this.

Let $q = \mu_1$ and 
write $h_i$ for the number of boxes in column $i$ of $\SD_\mu$,
so that $h_i = i$ for $i \in [r]$ and $h_i=0$ for $i >q$.
Continue to let $\triangle := 1 + 2 + \dots + r $.
For $r+1 \leq i \leq q+1$
 define
\[\INSERTUPTO(i) :=  \ck_{h_1 + h_2 +\dots + h_{i-1}-1}\cdots \ck_{\triangle+2}\ck_{\triangle+1}.\]
Note that this gives the identity operator if $i=r+1$ or if $i=r+2$ and $h_{r+1}=1$.
Applying $\INSERTUPTO(i)$  to 
\eqref{vVv-eq} has the effect of continuing the column insertion
process described in Definition~\ref{o-eg-def} past columns $r+1,r+2,\dots,i-1$.
and then taking the column reading word; compare again with Example~\ref{worked-ex3}.
Define
\[\TOTAL(i) := \INSERTUPTO(i)  \circ  \COLUMNINSERT\circ \REORIENT\circ \REVERSE\circ \ROWINSERT.
\]
Let $m \in\{r+1,r+2,\dots,q\}$ be minimal
with \[h_1 + h_2 +\dots + h_{m-1}+1 \in \Des(\TOTAL(m)\circ {\fk o}(w) ),\]
or if no such $m$ exists then set $m:=q+1$.
It follows from Lemma~\ref{16lem}(5) and \eqref{vVv-eq} that
$\QO(w)$ is formed from $Q$ by adding $n+1'$ to the box $(h_{m}+1, m)$,
and  if $\lambda$ is the strict partition shape of $\QO(w)$ then
$\row(\PO(w)) = \fk p (w)$ for the operator 
\[\fk p := \tau^\row_\lambda\circ  \TOTAL(m)\circ {\fk o}.\]
The shifted tableau $\QO(w)$ and the operator $\fk p$ have now been determined in all cases.
As the outputs  of the preceding algorithm are unchanged if the starting word
$w$ is replaced by $\ifkb_\pi(w)$, we conclude that
$ \QO(w) =\QO(\ifkb_\pi(w)) $ as desired.
\end{proof}

The following worked examples demonstrate the algorithm just given to compute $\QO(w)$.
Our notation maintains the conventions in the proof above.
The first example shows the relatively simple case 
when adding the final letter $w_{n+1}$ to $\PO(w_1w_2\cdots w_n)$ ends in row insertion.

\begin{example}\label{worked-ex1}
Suppose $w = 354251234$ 
so that
$n = \ell(w) - 1 = 8$.
We have 
\[P:=\PO(w_1w_2\cdots w_8) = \ytab{\none & \none & 5 \\ \none &3 & 4 \\ 1 & 2 &3 &4 & 5}
\quand
Q:=\QO(w_1w_2\cdots w_8) = \ytab{\none & \none & 8 \\ \none &3 & 7' \\ 1 & 2 &4' &5 & 6'}\]
so
$\mu=(5,2,1)$,
$r=3$, $q=5$, 
and
$d_1 = 8 > d_2 = 3 > d_3 = 1 $.
We assume that the operator $\fk o$ with 
\[
 \fk o (w) =\row(P)w_{n+1} = \row\(\ytab{\none & \none & 5 \\ \none &3 & 4 \\ 1 & 2 &3 &4 & 5& \none & 4}\)= 534123454
 \]
 is given.
 Then
\[
\ba
\rho_1 \circ \fk o (w) &= \row\(\ytab{\none & \none & 5 \\ \none &3 & 4 \\ 1 & 2 &3 &4 & 5& \none & 4}\)= 534123454 \text{ has a descent at $d_1=8$, but}
\\
\rho_2 \circ \fk o (w) &=  \row\(\ytab{\none & \none & 5 \\ \none &3 & 4 & \none & \none& \none  & 5\\  1 &2 &3 & 4&  5}\)=534512345\text{ has no descent at $d_2=3$}.
\ea
\]
Therefore $\row(\PO(w)) = \fk p(w)$ for  $\fk p := \rho_2 \circ \fk o$ and
$\QO(w)$
is formed from $Q$ by adding $9=n+1$ to box $(2,4)=(2,2+\mu_2)$.
This is consistent with Definition~\ref{o-eg-def}, which gives
\[
\PO(w) = \ytab{\none & \none & 5 \\ \none &3 & 4 &  5\\  1 &2 &3 & 4&  5}
\quand
\QO(w) = \ytab{\none & \none & 8 \\ \none &3 & 7' & 9 \\ 1 & 2 &4' &5 & 6'}.
\]
\end{example}

The next example shows the case where inserting $w_{n+1}$ into $\PO(w_1w_2\cdots w_n)$
adds a new row.

\begin{example}\label{worked-ex3}
Suppose $w = 35425123$ 
so that
$n = \ell(w) - 1 = 7$.
We have 
\[P:=\PO(w_1w_2\cdots w_7) = \ytab{ \none &3 & 5 \\ 1 & 2 &3 &4 & 5}
\quand
Q:=\QO(w_1w_2\cdots w_7) = \ytab{ \none &3 & 7' \\ 1 & 2 &4' &5 & 6'}\]
so
$\mu=(5,2)$,
$r=2$, $q=5$, 
and
$d_1 = 7 > d_2 = 2 $.
We assume that the operator $\fk o$ with 
\[
 \fk o (w) =\row(P)w_{n+1} = \row\(\ytab{ \none &3 & 5 \\ 1 & 2 &3 &4 & 5& \none & 3}\)= 3512345
 \]
 is given.
 Then
\[
\ba
\rho_1 \circ \fk o (w) &= \row\(\ytab{\none &3 & 5 \\ 1 & 2 &3 &4 & 5& \none & 3}\)= 35123453 \text{ has a descent at $d_1=7$,}
\\
\rho_2 \circ \fk o (w) &=  \row\(\ytab{\none &3 & 5&  \none &\none & \none & 4 \\ 1 & 2 &3 &4 & 5}\)=35412345\text{ has a descent at $d_2=2$, but}
\\
\rho_3 \circ \fk o (w) &=  \row\(\ytab{\none &5 & 3&  \none &\none & \none & 4 \\ 1 & 2 &3 &4 & 5}\)=53412345\text{ has no descent at $d_2=2$}.
\ea
\]
Therefore $\row(\PO(w)) = \fk p(w)$ for  $\fk p := \rho_3 \circ \fk o$ and
$\QO(w)$
is formed from $Q$ by adding $8=n+1$ to the diagonal box $(3,3)=(r+1,r+1)$.
This is consistent with Definition~\ref{o-eg-def}, which gives
\[
\PO(w) = \ytab{\none & \none & 5 \\ \none &3 & 4 \\  1 &2 &3 & 4&  5}
\quand
\QO(w) =\ytab{ \none & \none & 8  \\ \none &3 & 7' \\ 1 & 2 &4' &5 & 6'}.
\]
\end{example}

We now consider a  case 
where adding  $w_{n+1}$ to $\PO(w_1w_2\cdots w_n)$ ends in column insertion.

\begin{example}\label{worked-ex2}
Suppose $w = 1528639742$ so that $n=\ell(w)-1=9$. We have
\[
P := \PO(w_1w_2\cdots w_9)= \ytab{ \none & \none & 8 & 9  \\ \none &5 & 6  & 7\\ 1 & 2 &3 &4 }
\quand
Q :=\QO(w_1w_2\cdots w_9)  =\ytab{\none & \none & 6 & 9  \\ \none &3 & 5  & 8\\ 1 & 2 &4 &7},
\]
so
$\mu=(4,3,2)$,
$r=3$, $q=4$,
$d_1 = 9 > d_2 = 5 > d_3 = 2 $.
We assume that the operator $\fk o$ with 
\[
\fk o (w) = \row(P)w_{n+1}=\row\(\ytab{\none & \none & 8 & 9  \\ \none &5 & 6  & 7\\ 1 & 2 &3 &4& \none &2  }\)= 8956712342
\]
is given. 
Then
\[
\ba
\rho_1 \circ \fk o (w) &= \row\(\ytab{\none & \none & 8 & 9  \\ \none &5 & 6  & 7\\ 1 & 2 &3 &4& \none &2  }\)= 8956712342\text{ has a descent at $d_1=9$}, \\
\rho_2 \circ \fk o (w) &=  \row\(\ytab{\none & \none & 8 &  9  \\ \none &5 & 6  & 7 & \none &3\\ 1 & 2 &3 &4 }\)=8956731234\text{ has a descent at $d_2=5$,} \\
\rho_3 \circ \fk o (w) &=  \row\(\ytab{\none & \none & 8 & 9 & \none  &5 \\ \none &3 & 6  & 7 \\ 1 & 2 &3 &4  }\)=8953671234\text{ has a descent at $d_3=2$, and} \\ 
\rho_4 \circ \fk o (w) &=  \row\(\ytab{\none & \none & 9 & 8 & \none  &5 \\ \none &3 & 6  & 7 \\ 1 & 2 &3 &4  }\)=9853671234\text{ has a descent at $d_3=2$}.
\ea\]
  This means that the process inserting $w_{n+1}$ into $P$ will end in column-insertion rather than row-insertion.
Successively applying the operators $\psi_i$ to $ \fk o (w) $ gives
\[
\ba
\psi_1 \circ \fk o (w) &=  \row\(\ytab{
8  & \none& \none & \none& \none & 5 & 9    \\ 
\none & \none& \none& \none &3 & 6  & 7 \\ 
\none& \none & \none & 1 & 2 &3 &4  }\)= 8593671234,
\\
\psi_2\psi_1 \circ \fk o (w) &=  \row\(\ytab{
8  & 5 & \none& \none   & \none& 3 & 9    \\ 
\none& \none& \none& \none &1 & 6  & 7 \\ 
\none& \none  & \none& \none & 2 &3 &4 }\) = 8539167234,
\\
\psi_3\psi_2\psi_1 \circ \fk o (w) &=  \row\(\ytab{
8  & 5 & 3   & \none& \none& 1 & 9    \\ 
\none& \none& \none& \none &\none & 6  & 7 \\ 
\none  & \none& \none& \none & 2 &3 &4 }\) =8531967234,
\ea
\]
so we have
\[
\ba
\ROWINSERT\circ \fk o(w) &=  \row\(\ytab{8  & 5 & 3   & 1 &  \none& \none & 9    \\ \none& \none& \none& \none &\none & 6  & 7 \\ \none  & \none& \none& \none & 2 &3 &4 }\) =8531967234, \\
\REVERSE\circ \ROWINSERT\circ \fk o(w) &= \row\(\ytab{1  & 3 & 5   & 8 &  \none& \none & 9    \\ \none& \none& \none& \none &\none & 6  & 7 \\ \none  & \none& \none& \none & 2 &3 &4 }\) = 1358967234,
\\
\REORIENT\circ \REVERSE\circ \ROWINSERT \circ {\fk o}(w)
&=
 \col\(\ytab{1 \\  3 &  \none& \none & 9    \\  5 &\none & 6  & 7 \\  8 & 2 &3 &4 }\) = 1358263974.
 \ea
 \]
 Successively applying the operators $\phi_i$ to the last word gives
\[
\ba
\REORIENT\circ \REVERSE\circ \ROWINSERT \circ {\fk o}(w)
&=
 \col\(\ytab{ \none & 3 \\  \none &  5& \none & 9    \\  \none &8 & 6  & 7 \\  1 & 2 &3 &4}\) =1358263974,
 \\
 \phi_1\circ \REORIENT\circ \REVERSE\circ \ROWINSERT \circ {\fk o}(w)
&=
 \col\(\ytab{ \none & \none & 5 \\  \none &  \none& 8 & 9    \\  \none &3 & 6  & 7 \\  1 & 2 &3 &4}\) =1325863974,
  \\
 \phi_2\phi_1\circ \REORIENT\circ \REVERSE\circ \ROWINSERT \circ {\fk o}(w)
&=
 \col\(\ytab{ \none & \none & \none & 6 \\  \none &  \none& 8 & 9    \\  \none &3 & 5  & 7 \\  1 & 2 &3 &4}\) =1328536974,
 \ea
 \]
so we have 
 \[
 \COLUMNINSERT\circ \REORIENT\circ \REVERSE\circ \ROWINSERT \circ {\fk o}(w)
 =
  \col\(\ytab{ \none & \none & \none & 6  \\  \none &  \none& 8 & 9    \\  \none &3 & 5  & 7 \\  1 & 2 &3 &4}\) =1328536974.
  \]
  Since $r+1=q=4$ in this example, the operator $\TOTAL(i)$ is only defined for $i\in\{4,5\}$, and this gives
     \[
     \ba
   \TOTAL(4)\circ \fk o(w) &=   \col\(\ytab{ \none & \none & \none & 6  \\  \none &  \none& 8 & 9    \\  \none &3 & 5  & 7 \\  1 & 2 &3 &4}\) =1328536974,
   \\  
\TOTAL(5)  \circ {\fk o}(w)
 &=
  \col\(\ytab{  \none &  \none& 8 & 9    \\  \none &3 & 5  & 6 \\  1 & 2 &3 &4 & 7}\) =1328539647.
  \ea
  \]
As $h_1+h_2+ \dots+h_r + 1 = 1 +2+3+1=7$ is not a descent of $\TOTAL(4) \circ {\fk o}(w)$, we have $m=q+1=5$.
Thus $\QO(w)$ is formed from $Q$ by adding $10'=n+1'$ to the box $(1,5)=(h_{m}+1, m)$, meaning that
\[
\ytableausetup{boxsize = .6cm,aligntableaux=center}
\QO(w) = \begin{ytableau} \none & \none & 6 & 9  \\ \none &3 & 5  & 8\\ 1 & 2 &4 &7 & 10'\end{ytableau}.
\]
Additionally, we have
$\row(\PO(w)) = \fk p (w)$ for  
$\fk p := \tau^\row_{\lambda}\circ  \TOTAL(5)\circ {\fk o}$
where $\lambda = (5,3,2)$ is the shape of $\QO(w)$.
This agrees with 
Definition~\ref{o-eg-def}, which gives
$
\PO(w) =\ytab{  \none &  \none& 8 & 9    \\  \none &3 & 5  & 6 \\  1 & 2 &3 &4 & 7}.
$
\end{example}

Our next result is a variant of Theorem~\ref{qi-thm1}.
\begin{theorem}\label{qi-thm2}
Let $\pi \in I_\ZZ$.
Then $\ifkb_\pi$ is an isomorphism of $\q_n$-crystals
$\ds\bigsqcup_{\sigma\in I_\ZZ} \iR_n(\sigma) \to \bigsqcup_{\sigma\in I_\ZZ} \iR_n(\sigma)$.
\end{theorem}

\begin{proof}
Theorems~\ref{isim-thm}, \ref{simCK-thm},
and
\ref{eg-thm}, 
show that each full $\gl_n$-subcrystal of
$\cY := \bigsqcup_{\sigma\in I_\ZZ} \iR_n(\sigma)$ 
is the set of all $n$-fold increasing factorizations of reduced words in a single Coxeter-Knuth equivalence class.
Thus, it follows by combining Theorem~\ref{qi-thm1}, Lemma~\ref{easy-lem},
and Theorem~\ref{o-little-thm}(a) that $\ifkb_\pi : \cY \to \cY$ is 
at least an isomorphism of abstract $\gl_n$-crystals.

Let $w=(w^1,w^2,\dots,w^n) \in \cY$.
To show that $\ifkb_\pi$ is a $\q_n$-crystal morphism,
 it is enough to check
that 
 $\fO(w) \neq 0$ if and only if $\fO(\ifkb_\pi(w)) \neq 0$,
and that in this case $\ifkb_\pi(\fO(w))  = \fO (\ifkb_\pi(w))$.
Let $p=\ell(w^1)$ and $q=\ell(w^2)$.
From the definitions in Section~\ref{o-fact-sect},
it is easy to work out that  $\fO(w) \neq 0$ if and only if 
$p$ is neither zero nor a descent of 
$v:= \ck_{p-2}\cdots \ck_1\ock(w^1w^2)$.
In
this 
case,
if $v^1$ and $v^2$ are the words of length $p-1$ and $q+1$ such that $v=v^1v^2$,
then $\fO(w) = (v^1,v^2,w^3,\dots w^n)$.
As $\ifkb_\pi$ preserves descents and commutes with every $\ock$ and $\ck_i$ by Theorem~\ref{o-little-thm},
the claim follows.
\end{proof}

\subsection{Fixed-point-free Little bumps}

The results in the previous section have a parallel story for fpf-involution words,
which we present here. 
Fix $\pi \in \Ifpf_\ZZ$ and recall that $\cAfpf(\pi)\subset S_\ZZ$ is such that
$\iRfpf(\pi) = \bigsqcup_{\sigma \in \cAfpf(\pi)} \cR(\sigma)$.
A \emph{$\pi$-marked fpf-involution word}
is a pair $(w,i)$
in which $w$ is a word and $i$ is an index
 such that $\del_i(w) \in \iRfpf(\pi)$.
 Equivalently, this is just an $\alpha$-marked word for some $\alpha \in \cAfpf(\pi)$.
 
A $\pi$-marked fpf-involution word is \emph{fpf-reduced} if $w$ is  an fpf-involution word.
If $(w,i)$ is not fpf-reduced but
$w \in \cR(\sigma)$ for some $\sigma \in S_\ZZ$ with
$\sigma^{-1} \cdot \minfpf \cdot \sigma = \pi$,
then $(w,i)$ is \emph{semi-reduced}.

\begin{lemma}[{\cite[Lemma 4.21]{HMP3}}]
\label{fpf-toggle-lem}
If $(w,i)$ is a $\pi$-marked fpf-involution word of length $m$ that is neither fpf-reduced nor semi-reduced,
then there is a unique index $i\neq j \in [\ell(w)]$ such that $(w,j)$ is also a $\pi$-marked fpf-involution word; moreover,
in this event $(w,j)$ is also not semi-reduced.
\end{lemma}

\begin{definition}
Let $(w,i)$ be a $\pi$-marked fpf-involution word of length $m$.
If $(w,i)$ is semi- or fpf-reduced, then let $j=i$,
and otherwise let $i\neq j \in [m]$ be 
such that $(w,j)$ is a $\pi$-marked fpf-involution word.
Then define $\fpush(w,i) := (v,j)$ where
 $v := w_1\cdots w_{j-1} (w_j + 1) w_{j+1}\cdots w_m.$
 \end{definition}
 
As with $\push$ and $\ipush$,
if
$(w,i)$ is a $\pi$-marked fpf-involution word then 
$\fpush(w,i)$ is fpf-reduced for some sufficiently large $N>0$
\cite[Lemma 4.26]{HMP3}.
Also, for any fixed
$\pi \in \Ifpf_\ZZ$ and fpf-involution word $w$, 
at most one index $i$ exists
such that $(w,i)$ is a $\pi$-marked fpf-involution word.

\begin{definition}[{\cite[\S4.3]{HMP3}}]
The \emph{fpf-involution Little bumping operator} $\ffkb_\pi$ of $\pi \in \Ifpf_\ZZ$
acts on fpf-involution words $w$ as follows.
If $(w,i)$ is a marked fpf-involution word for some $i$, and $N>0$
is minimal such that 
$ \fpush^N(w,i) =: (v,j)$ is fpf-reduced, then
$\ffkb_{\pi}(w) := v$.
Otherwise, $\ffkb_{\pi}(w) := w$.
\end{definition}

The map $\ffkb_\pi$ is the inverse of the operator 
${\hat\cB}^\fpf_\pi$  in 
\cite[Theorem 4.29]{HMP3}.

\begin{example}\label{fbb-ex} 
Let $\pi = (1,2)(3,6)(4,5) \in \Ifpf_\ZZ$, $\sigma =(1,4)(2,5)(3,6) \in \Ifpf_\ZZ$,
and $w=243 \in \iRfpf(\sigma)$, so that
 $43 \in \iRfpf(\pi)$. The values of
$\fpush^N(243,1)$ for $0 \leq N \leq 6$ are as follows:
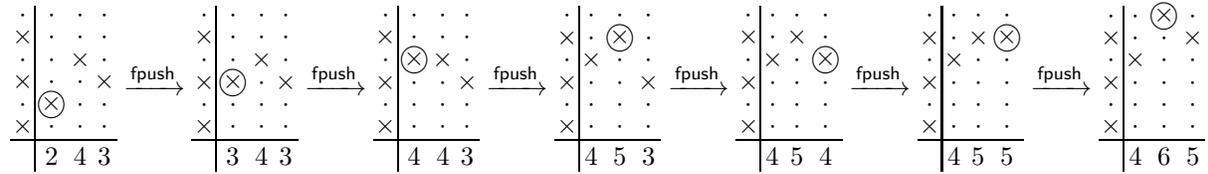
\def\dims{\arraycolsep=0.7pt\def\arraystretch{0.7}}
\[{\small
\ba
\dims
\barr{c|ccc}
\cdot & \cdot & \cdot & \cdot \\
\times & \cdot & \cdot & \cdot\\
\cdot & \cdot & \times & \cdot  \\
\times & \cdot & \cdot & \times  \\
\cdot & \ttimes & \cdot & \cdot  \\
\times & \cdot & \cdot & \cdot  \\
 \hline \\[-6pt]
&  2 & 4 & 3
\earr
& 
\xrightarrow{\fpush}
\dims
\barr{c|ccc}
\cdot & \cdot & \cdot & \cdot \\
\times & \cdot & \cdot & \cdot\\
\cdot & \cdot & \times & \cdot  \\
\times & \ttimes & \cdot & \times  \\
\cdot & \cdot & \cdot & \cdot  \\
\times & \cdot & \cdot & \cdot  \\
 \hline \\[-6pt]
&  3 & 4 & 3
\earr
\xrightarrow{\fpush} 
\dims
\barr{c|ccc}
\cdot & \cdot & \cdot & \cdot \\
\times & \cdot & \cdot & \cdot\\
\cdot & \ttimes & \times & \cdot  \\
\times & \cdot & \cdot & \times  \\
\cdot & \cdot & \cdot & \cdot  \\
\times & \cdot & \cdot & \cdot  \\
 \hline \\[-6pt]
&  4 & 4 & 3
\earr
\xrightarrow{\fpush}
\dims
\barr{c|ccc}
\cdot & \cdot & \cdot & \cdot \\
\times & \cdot & \ttimes & \cdot \\
\cdot & \times & \cdot & \cdot  \\
\times & \cdot & \cdot & \times  \\
\cdot & \cdot & \cdot & \cdot  \\
\times & \cdot & \cdot & \cdot  \\
 \hline \\[-6pt]
&  4 & 5 & 3
\earr
\xrightarrow{\fpush}
\dims
\barr{c|ccc}
\cdot & \cdot & \cdot & \cdot \\
\times & \cdot & \times & \cdot \\
\cdot & \times & \cdot & \ttimes  \\
\times & \cdot & \cdot & \cdot  \\
\cdot & \cdot & \cdot & \cdot  \\
\times & \cdot & \cdot & \cdot  \\
 \hline \\[-6pt]
&  4 & 5 & 4
\earr
\xrightarrow{\fpush}
\dims
\barr{c|ccc}
\cdot & \cdot & \cdot & \cdot \\
\times & \cdot & \times & \ttimes \\
\cdot & \times & \cdot & \cdot  \\
\times & \cdot & \cdot & \cdot  \\
\cdot & \cdot & \cdot & \cdot  \\
\times & \cdot & \cdot & \cdot  \\
 \hline \\[-6pt]
&  4 & 5 & 5
\earr
\xrightarrow{\fpush}
\barr{c|ccc}
\cdot & \cdot & \ttimes & \cdot \\
\times & \cdot & \cdot & \times \\
\cdot & \times & \cdot & \cdot  \\
\times & \cdot & \cdot & \cdot  \\
\cdot & \cdot & \cdot & \cdot  \\
\times & \cdot & \cdot & \cdot  \\
 \hline \\[-6pt]
&  4 & 6 & 5
\earr
\ea}
\]
The last marked word in this sequence is fpf-reduced, the second and fifth are semi-reduced,
and the fourth is reduced but not fpf-reduced. We conclude that $\ffkb_\pi(w) = 465$.
\end{example}

We have an analogue of Lemma~\ref{easy-lem}, with almost the same proof. 

\begin{lemma}\label{fpf-easy-lem}
Let $\pi \in \Ifpf_\ZZ$.
For any fpf-involution word $w$, 
there is a finite sequence of elements $\alpha_1,\alpha_2,\dots,\alpha_l \in \cAfpf(\pi)$
such that $\ffkb_\pi(w) = \fkb_{\alpha_l}\cdots \fkb_{\alpha_2} \fkb_{\alpha_1}(w)$.
Moreover, this sequence is the same for all fpf-involution words in a single Coxeter-Knuth equivalence class.
\end{lemma}

In Example~\ref{fbb-ex} 
where $\pi =  (1,2)(3,6)(4,5)$ and $w=243 $ 
we have $\ffkb_\pi(w) =\fkb_{\alpha_4}\fkb_{\alpha_3}\fkb_{\alpha_2}\fkb_{\alpha_1}(w)$
for $\alpha_1 =\alpha_2=\alpha_3= s_4s_3 = (3,5,4)$ and
$\alpha_4 = s_4s_5 = (4,5,6)$.

\begin{proof}
Fix an fpf-involution word $w$ and let $\alpha_1,\dots,\alpha_l \in \cAfpf(\pi)$
be such that
$\ffkb_\pi(w) = \fkb_{\alpha_l}\cdots\fkb_{\alpha_1}(w)$.
Define $\sigma_i \in S_\ZZ$ for $i =1,2,\dots,l$ as in the proof of Lemma~\ref{easy-lem}.
Then $\alpha_1$ is the unique element of $\cAfpf(\pi)$ with $\alpha_1 \lessdot \sigma_1$,
and $\alpha_i$ for $i>1$ is 
either $\alpha_{i-1}$ when $\sigma_i^{-1} \cdot \minfpf \cdot \sigma_i =\pi$
or else
the unique element of $\cAfpf(\pi)\setminus\{\alpha_{i-1}\}$
with $\alpha_i \lessdot \sigma_i$.   
It follows that any word $v \in \cR(\sigma_1)$
with $\fkb_{\alpha_{i-1}}\cdots \fkb_{\alpha_1}(v) \in \cR(\sigma_i)$ for all $1<i \leq l$
also has $\ffkb_\pi(v) =  \fkb_{\alpha_l}\cdots \fkb_{\alpha_1}(v)$.
This holds
if $v \simCK w$ by Theorem~\ref{little-thm}.
\end{proof}

Recall the definition of $\spck$ from \eqref{spck-eq}.
Theorem~\ref{little-thm}
has another analogue for the maps $\ffkb_\pi$.

\begin{theorem}\label{sp-little-thm}
Let $\pi,\sigma \in \Ifpf_\ZZ$ and $w \in \iRfpf(\sigma)$. Then:
\ben
\item[(a)] The operator $\ffkb_\pi$ is a bijection $\bigsqcup_{z \in \Ifpf_\ZZ} \iRfpf(z) \to \bigsqcup_{z \in \Ifpf_\ZZ} \iRfpf(z)$.
\item[(b)] It holds that $\Des(\ffkb_{\pi}(w)) = \Des(w)$.
\item[(c)] For all $i\in\PP$ it holds that 
$\spck(\ffkb_\pi(w)) = \ffkb_\pi(\spck(w))$ and $\ck_i(\ffkb_\pi(w)) = \ffkb_\pi(\ck_i(w)).$
\item[(d)] It holds that $\QSp(\ffkb_\pi(w)) = \QSp(w)$.
\een
\end{theorem} 

\begin{proof}
Part (a) is equivalent to \cite[Theorem 4.29]{HMP3},
while part (b) is immediate from Theorem~\ref{little-thm}(b) and Lemma~\ref{fpf-easy-lem}.
Our proof of part (c)
 is similar to argument given for Theorem~\ref{o-little-thm}(c), but the details to check 
 are more complicated.
It is clear from Theorem~\ref{little-thm}(c) and Lemma~\ref{fpf-easy-lem}
that if $i>0$ then
 $\ffkb_{\pi}(\ck_i(w)) = \ck_i(\ffkb_\pi(w))$ for all fpf-involution words $w$.
 
It remains to consider the operator $\spck$.
Fix a word $w=w_1w_2\cdots w_m$ with $m\geq 2$,
and suppose $i \in [m]$ is  such that $(w,i)$ is a $\pi$-marked fpf-involution word.
If $i=1$ then $w_2$ must be even; if $i>1$ then $w_1$ must be even;
and if $i>2$ then $w_2$ must either be even or equal to $w_1\pm 1$.
Let $w^\bullet := \spck(w)$
and define $i^\bullet$ to be $3-i$ if $w_2$ is even and $i \in \{1,2\}$, and otherwise set $i^\bullet:=i$.
If $w$ is an fpf-involution word, so that 
$(w,i)$ is fpf-reduced, then it follows from Theorem~\ref{fsim-thm}
that $(w^\bullet,i^\bullet)$ is another $\pi$-marked fpf-involution word
 that is fpf-reduced.

Suppose the subword $w_1w_2$ is an fpf-involution word. 
It suffices to show that there exist $N,N^\bullet>0$
with the following properties: neither $\fpush^M(w,i)$
nor $\fpush^{M^\bullet}(w^\bullet,i^\bullet)$
 is fpf-reduced for any $0<M<N$ or $0<M^\bullet<N^\bullet$,
and if $\fpush^N(w,i) = (v,j)$ then $v_1v_2$ is an fpf-involution word
and $\fpush^{N^\bullet}(w^\bullet,i^\bullet) =(v^\bullet,j^\bullet)$
where $v^\bullet:= \spck(v)$ and $j^\bullet$ is defined to be either $3-j$
when $j \in \{1,2\}$ and $v_2$ is even or else $j^\bullet:=j$.
If this holds then 
by repeating the same claim for $(v,j)$ and $(v^\bullet,j^\bullet)$,
we deduce by induction
that $\ffkb_\pi(\spck(w)) = \spck(\ffkb_\pi(w))$.

It is a straightforward exercise to check this 
claim.
There are four possibilities for $N$ and $N^\bullet$:
\begin{itemize}
\item $N=N^\bullet=1$ if $i=i^\bullet \notin \{1,2\}$.
\item $N=N^\bullet=2$ if $i=3-i^\bullet \in \{1,2\}$ and  either $w_i > w_{i^\bullet}+1$ or $w_i < w_{i^\bullet}-2$.
\item $N=3$ and $N^\bullet=1$ if $i=1$ and $w_1=w_{2}-2$, or $i=2$ and $w_2=w_1-1$.
\item $N=1$ and $N^\bullet=3$ if $i=2$ and $w_2\in \{w_1-2,w_1+1\}$.
\end{itemize}
If $i=1$ then $w_1 \notin\{ w_2\pm 1\}$
so these cases are exhaustive.
The last two cases are more interesting.
For example, 
$\fpush^3(24,1) = \fpush^2(34,1)=\fpush(44,1) = (45,2)$
and
$\fpush(42,2) = (43,2)$ as
\[{
\arraycolsep=0.7pt\def\arraystretch{0.7}
\barr{c|cc}
\times&\cdot & \cdot \\
\cdot &\cdot & \times \\
\times&\cdot & \cdot \\
\cdot &\ttimes & \cdot \\
\times&\cdot & \cdot\\
 \hline \\[-6pt]
 & 2 & 4
\earr
\hs\xrightarrow{\fpush}\hs
\barr{c|cc}
\times&\cdot & \cdot \\
\cdot &\cdot & \times \\
\times&\ttimes & \cdot \\
\cdot &\cdot & \cdot \\
\times&\cdot & \cdot\\
 \hline \\[-6pt]
 & 3 & 4
\earr
\hs\xrightarrow{\fpush}\hs
\barr{c|cc}
\times&\cdot & \cdot \\
\cdot &\ttimes & \times \\
\times&\cdot & \cdot \\
\cdot &\cdot & \cdot \\
\times&\cdot & \cdot\\
 \hline \\[-6pt]
 & 4 & 4
\earr
\hs\xrightarrow{\fpush}\hs
\barr{c|cc}
\times&\cdot & \ttimes \\
\cdot &\times & \cdot \\
\times&\cdot & \cdot \\
\cdot &\cdot & \cdot \\
\times&\cdot & \cdot\\
 \hline \\[-6pt]
 & 4 & 5
\earr}
\quand
{
\arraycolsep=0.7pt\def\arraystretch{0.7}
\barr{c|cc}
\times&\cdot & \cdot \\
\cdot &\times & \cdot \\
\times&\cdot & \cdot \\
\cdot &\cdot & \ttimes \\
\times&\cdot & \cdot\\
 \hline \\[-6pt]
 & 4 & 2
\earr
\hs\xrightarrow{\fpush}\hs
\barr{c|cc}
\times&\cdot & \cdot \\
\cdot&\times & \cdot \\
\times& \cdot & \ttimes\\
\cdot&\cdot & \cdot \\
\times & \cdot & \cdot\\
 \hline \\[-6pt]
 & 4 & 3
\earr}
\]
and the situation is analogous if we replace $24$
with any word $w_1w_2$ that has $w_2=w_1+ 2$.
Likewise, we have
$\fpush(23,2) = (24,2)$
and
$\fpush^3(21,2) = \fpush^2(22,2)=\fpush(32,1) = (42,1)$
as 
\[
{
\arraycolsep=0.7pt\def\arraystretch{0.7}
\barr{c|cc}
\cdot &\cdot & \cdot \\
\times &\cdot & \ttimes \\
\cdot &\times & \cdot \\
\times &\cdot & \cdot \\
 \hline \\[-6pt]
 & 2 & 3
\earr
\hs\xrightarrow{\fpush}\hs
\barr{c|cc}
\cdot &\cdot & \ttimes \\
\times &\cdot & \cdot \\
\cdot &\times & \cdot \\
\times &\cdot & \cdot \\
 \hline \\[-6pt]
 & 2 & 4
\earr}
\quand
{
\arraycolsep=0.7pt\def\arraystretch{0.7}
\barr{c|cc}
\cdot &\cdot & \cdot \\
\times &\cdot & \cdot \\
\cdot &\times & \cdot \\
\times &\cdot & \ttimes \\
 \hline \\[-6pt]
 & 2 & 1
\earr
\hs\xrightarrow{\fpush}\hs
\barr{c|cc}
\cdot &\cdot & \cdot \\
\times &\cdot & \cdot \\
\cdot &\times & \ttimes \\
\times &\cdot & \cdot \\
 \hline \\[-6pt]
 & 2 & 2
\earr
\hs\xrightarrow{\fpush}\hs
\barr{c|cc}
\cdot &\cdot & \cdot \\
\times &\ttimes & \cdot \\
\cdot &\cdot & \times \\
\times &\cdot & \cdot \\
 \hline \\[-6pt]
 & 3 & 2
\earr
\hs\xrightarrow{\fpush}\hs
\barr{c|cc}
\cdot &\ttimes & \cdot \\
\times &\cdot & \cdot \\
\cdot &\cdot & \times \\
\times &\cdot & \cdot \\
 \hline \\[-6pt]
 & 4 & 2
\earr}
\]
and something similar happens if we replace $23$
with any word $w_1w_2$ that has $w_2 \in \{w_1-2,w_1+1\}$.
We leave the details of the other cases to the reader.

Parts (a), (b), and (c) show that $\ffkb_\pi$ preserves descents and 
commutes with $\spck$ and $\ck_i$ for all $i>0$.
Choose an fpf-involution word $w$.
To prove part (d), it suffices to give an algorithm that computes $\QSp(w)$
using only the descent sets of words in the $\simFCK$-equivalence class of 
$w$ (plus the sequences of symplectic Coxeter-Knuth moves transforming $w$ 
to each word) as inputs.

An algorithm of this form exists with exactly the same description as the one given in the proof 
of Theorem~\ref{o-little-thm}(d); one just needs to 
replace all instances of the symbols $\ifkb_\pi$, $\ock$, $\PO$, $\QO$ 
with 
$\ffkb_\pi$, $\spck$, $\PSp$, $\QSp$, respectively,
and redirect any references to Definition~\ref{o-eg-def} to Definition~\ref{sp-eg-def}.
After making these substitutions, 
every step in the proof of Theorem~\ref{o-little-thm}(d) holds verbatim,
with the exception that 
one should also change \eqref{tildetilde-eq}
to define
\[
\tilde{\tilde x} = \begin{cases} 
x & \text{if }x <y \\
y &\text{if }x = y-1 \\
y&\text{if }x=y\end{cases}
\quand
\tilde{\tilde y} =  \begin{cases} 
y & \text{if }x <y \\
y+1 &\text{if }x = y-1 \\
y+3&\text{if }x=y\end{cases}
\] 
while keeping the same values of $\tilde x$ and $\tilde y$. The task of walking through 
the steps in the proof of Theorem~\ref{o-little-thm}(d) a second time, with these minor modifications,
is straightforward and left to reader.
The moral is that we can compute the tableau $\QSp(w)$ with an algorithm that executes in exactly the same way when
$w$ is replaced by $\ffkb_\pi(w)$, so we must have $\QSp(w) = \QSp(\ffkb_\pi(w))$.
\end{proof}

Finally, we have a symplectic version of Theorem~\ref{qi-thm2}.

\begin{theorem}\label{qi-thm3}
Let $\pi \in \Ifpf_\ZZ$.
Then $\ffkb_\pi$ is an isomorphism of $\q_n$-crystals
\[\ds\bigsqcup_{\sigma\in \Ifpf_\ZZ} \iRfpf_n(\sigma) \to \bigsqcup_{\sigma\in \Ifpf_\ZZ} \iRfpf_n(\sigma).\]
\end{theorem}

\begin{proof}
By Theorems~\ref{fsim-thm}, \ref{simCK-thm}, and \ref{eg-thm},
each full $\gl_n$-subcrystal of
$\cZ := \bigsqcup_{\sigma\in \Ifpf_\ZZ} \iRfpf_n(\sigma)$ 
is the set of all $n$-fold increasing factorizations of reduced words in a single Coxeter-Knuth equivalence class.
Theorem~\ref{qi-thm1}, Lemma~\ref{fpf-easy-lem},
and Theorem~\ref{sp-little-thm}(a) therefore imply that $\ffkb_\pi : \cZ \to \cZ$ is 
an isomorphism of abstract $\gl_n$-crystals.

Let $w=(w^1,w^2,\dots,w^n) \in \cZ$.
To show that $\ffkb_\pi$ is a $\q_n$-crystal morphism,
 it is enough to check
that 
 $\fSp(w) \neq 0$ if and only if $\fSp(\ffkb_\pi(w)) \neq 0$,
in which case $\ffkb_\pi(\fSp(w))  = \fSp (\ffkb_\pi(w))$.
Let $p=\ell(w^1)$ and $q=\ell(w^2)$.
As in the orthogonal case, we have
$\fSp(w) \neq 0$ if and only if 
$p>0$ and $p$ is not a descent of
$v:= \ck_{p-2}\cdots \ck_1\spck(w^1w^2)$,
which we interpret as $v := w^1w^2$ when $p=1$.
If this happens, then it follows by definition that 
$\fSp(w) = (v^1,v^2,w^3,\dots w^n)$
where $v^1$ and $v^2$ are the words of length $p-1$ and $q+1$ such that $v=v^1v^2$.
Since $\ffkb_\pi$ preserves descents and commutes with $\spck$ and $\ck_i$,
we have
$\fSp(w) = \fSp(\ffkb_\pi(w)) = 0$
or
$\ffkb_\pi(\fSp(w))  = \fSp (\ffkb_\pi(w)) \neq 0$, as needed.
\end{proof}

\section{Proofs of the main results}\label{proofs-sect}

Here, we leverage our results in the previous section
to give complete proofs of Theorems~\ref{o-eg-thm}, \ref{main-thm1}, \ref{sp-eg-thm}, and \ref{main-thm2}.
We then describe some other applications and open problems.

\subsection{Reduction to permutations}\label{reduction-sect}

In this section, a
\emph{permutation} of a list of distinct numbers $a_1,a_2,\dots,a_m$ is a word of length $m$ 
containing each $a_i$ as a letter exactly once.

Fix $m,n \in \NN$. Let $\cS(m)\subset \cW_m(m)$ be the set of permutations of $1,2,\dots,m$
and define $\cS_n(m)$ to be the set of all $n$-fold increasing factorizations of words in $\cS(m)$.
Let $\cSe(m)$ be the set of permutations of $2,4,6,\dots,2m$ and write $\cSe_n(m)$ for
the set of $n$-fold increasing factorizations of words in $\cSe(m)$.
Finally, let $\Sigma(m)$ be the set of  involutions
$\sigma \in I_\ZZ$ of the form
\[
\sigma=(1,m)\quad\text{or}\quad
\sigma = (1,i_1)(i_1-1,i_2)(i_2-1,i_3)\cdots(i_{k}-1,m)\]
for some $1 < i_1 -1 < i_1<  i_2 -1 < i_2<\dots < i_k -1 < i _k < m$.
Both sets are $\q_n$-crystals:

\begin{proposition}\label{cse-prop}
One has
$\cS_n(m) = \bigsqcup_{\sigma\in\Sigma(m)} \iR_n(\sigma)$
and
$
\cSe_n(m) = \iR_n(\tau) = \iRfpf_n(\pi)
$
where  $\tau = s_2s_4s_6\cdots s_{2n} \in I_\ZZ$ and
$\pi \in \Ifpf_\ZZ$ is the involution
with $\pi(i) = \minfpf(i)$ for all $i \notin [2m]$ that maps
\[ i \mapsto \begin{cases} 
i + 2 &\text{if }i \in \{1,2m-2\}  \\
i - 2 &\text{if }i \in \{3, 2m\} 
\end{cases}
\quand
 i \mapsto \begin{cases} 
i + 3 &\text{if $i$ is even and $1<i<2m-2$} \\
i - 3 & \text{if $i$ is odd and $3<i<2m$}.
\end{cases}
\]
Moreover, the abstract $\q_n$-crystal structures on $\iR_n(\tau)$ and $\iRfpf_n(\pi)$ coincide.
\end{proposition}

\begin{proof}
Note that $\pi = s_{2m}\cdots s_4s_2\cdot \minfpf \cdot s_2s_4\cdots s_{2m} \in \Ifpf_\ZZ$.
Clearly $\cS(m)$ is a union of equivalence classes under the relation $\isim$, while
$\cSe(m)$ is a single equivalence class under $\isim$ and $\fsim$. Theorems~\ref{isim-thm} and \ref{fsim-thm}
imply that $\cS_n(m)=\bigsqcup_{\sigma\in\Sigma(m)} \iR_n(\sigma)$ for a finite set  $\Sigma(m)\subset I_\ZZ$ while $\cSe_n(m) =\iR(\tau)=\iRfpf_n(\pi)$. Checking that $\Sigma(m)$ is the given set is straightforward.
\end{proof}

Fix $w=(w^1,w^2,\dots,w^n) \in \cS_n(m)$. Define $w^{-1} \in \cW_n(m)$
to be the word of length $m$ whose $i$th letter is the index $j \in [n]$ 
of the factor $w^j$ that contains $i$ as a letter.
Then form $2[w] \in \cSe_n(m)$ by doubling the letters in each component of $w$.
For example, $(245,\emptyset,1, 3)^{-1} = 31411$. 
We write $\invert$ and $\dbl$ for  the corresponding maps $\cS_n(m) \to \cW_n(m)$ and $\cS_n(m) \to \cSe_n(m)$.

The following lemma shows that the $\q_n$-crystal structures
on $\cS_n(m)$ and $\cSe_n(m)$ afforded by Proposition~\ref{cse-prop}
are isomorphic to the crystal of words $\cW_n(m) \cong (\BB_n)^{\otimes m}$.

\begin{lemma}
\label{icc-lem}
Fix $m \in \NN$, $n \in \PP$, and $w \in \cS_n(m)$. The following properties hold:
\ben
\item[(a)] The map $\invert : \cS_n(m) \to \cW_n(m)$ is a $\q_n$-crystal isomorphism.

\item[(b)] The map $\dbl : \cS_n(m) \to \cSe_n(m)$ is a $\q_n$-crystal isomorphism.

\item[(c)] We have $\PO(w) = Q_\HM(w^{-1})$ and $\QO(w) = \QO(2[w])  = \QSp(2[w]) = P_\HM(w^{-1}).$
\een
\end{lemma}

\begin{proof}
The inverse of $\invert$ is the map that sends
an $m$-letter word $v =v_1v_2\cdots v_m \in \cW_n(m)$
to the tuple $(w^1,w^2,\dots,w^n)$
in which $w^j$ is the increasing word whose letters are the positions of $j$ in $v$. 
To prove part (a), it suffices to check, for each $i \in \{\overline{1}, 1,2,\dots,n-1\}$
and $w =(w^1,w^2,\dots,w^n) \in \cS_n(m)$, that 
$f_i(w) \neq 0$ if and only if $f_i(w^{-1})\neq 0$, in which case $f_i(w)^{-1} = f_i(w^{-1})$.

If $i=\overline{1}$ so that $f_{i}$ acts on $\cS_n(m)$ as the operator $\fO$, then this is clear by definition.
Assume $i \in[n-1]$.
If $f_i(w)$ is nonzero then it is formed from $w$ by removing some letter from $w^i$ and adding the same letter to $w^{i+1}$.
On the other hand, we have  $(a,b) \in \pair(w^i,w^{i+1})$ if and only if $a$ and $b$ are the positions of matching left and right parentheses in
the word formed from $w^1w^2\cdots w^n$ by replacing each $i$ by ``$)$'' and each $i+1$ by ``$($''. After comparing these observations with the definitions of $f_i$ for each crystal, the desired claim is evident.
This completes the proof of part (a).

Part (b) and the identity $\QO(w) =\QO(2[w])= \QSp(2[w])$ 
are obvious from the definitions.
Let $w =(w^1,w^2,\dots,w^n) \in \cS_n(m)$ and define $\underline w \in \cS_m(m)$ to be the factorization
obtained by dividing $w^1w^2\cdots w^n$ into subwords of length one.
Let $\phi : \{1'<1<\dots<m'<m\} \to \{1'<1<\dots<n'<n\}$ be the map that assigns $i \mapsto j$ and $i' \mapsto j'$
if the $i$th letter of $w^1w^2\cdots w^n$ is part of $w^j$. 
Then $\PO(w) = \PO(\underline w)$ and $\QO(w) = \phi \circ \QO(\underline w)$ by definition,
and it is easy to see that 
 $Q_\HM(w^{-1}) = Q_\HM(\underline w^{-1})$ and 
 $P_\HM(w^{-1}) = \phi \circ P_\HM(\underline w^{-1})$.
Thus, it suffices to show that $\PO(\underline w) = Q_\HM(\underline w^{-1})$
and
$\QO(w) = P_\HM(\underline w^{-1})$,
but this is \cite[Theorem 6.10]{HaimanMixed} as 
orthogonal-EG insertion applied to $\underline w$
 coincides with \emph{Sagan-Worley insertion} \cite[Definition 6.1]{HaimanMixed}.
\end{proof}

\begin{example}
If $w=(\emptyset, 36, 1245 ) \in \cS_3(6)$ then  $w^{-1} = 332332$ and we compute
\[ 
{\small\barr{rl}
\ytab{
  3 
}
   \leadsto 
\ytab{
  3 & 6
}
   \leadsto 
\ytab{
  1 &3 & 6 
}
      \leadsto 
\ytab{
\none & 3 \\
  1 & 2 & 6
}
      \leadsto 
\ytab{
\none & 3  & 6\\
  1 &2   & 4 
}
        \leadsto 
\ytab{
  \none & 3 & 6 \\
  1 & 2 & 4 & 5
} &  = \PO(w)
\earr
\quand
\barr{rl}
\ytab{
\none & 3 & 3 \\
  2 &2  & 3' & 3
} & = \QO(w).
\earr}
\]
Comparing with Example~\ref{hm-ex} shows that $\PO(w) = Q_\HM(w^{-1})$ and $\QO(w) = P_\HM(w^{-1})$.
\end{example}

\subsection{Proofs in the orthogonal case}
\label{proofs-o-sect}

This section contains our proofs of Theorems~\ref{o-eg-thm} and \ref{main-thm1}.

\begin{definition}
An involution $\pi \in I_\ZZ\setminus\{1\}$ is \emph{inv-Grassmannian} if
\[ 
\pi = (m+1,m+r + \mu_r)(m+2,m+r + \phi_{r-1})\cdots(m+r,m+r+\mu_1)
\]
for some $m\in \ZZ$ and some strict partition $\mu = (\mu_1>\mu_2>\dots >\mu_r>0)$.
In this case, the strict partition $\mu$ 
is the \emph{shape} of $\pi$.
We also consider $\pi=1$ to be inv-Grassmannian with shape $\mu = \emptyset$.
\end{definition}

The map $\pi \mapsto (\mu, m)$ is a bijection from nontrivial inv-Grassmannian 
involutions to nonempty strict partitions paired with nonnegative integers.

Given a word $w=w_1w_2\cdots w_m$ and a map $\pi : \ZZ \to \ZZ$, let 
\[
w^* = (-w_1)(-w_2)\cdots (-w_m)
\quand
\pi^* : i \mapsto 1-\pi(1-i).
\]
Then $\pi \mapsto  \pi^*$
is the unique automorphism of $S_\ZZ$ 
sending $s_i \mapsto s_{-i}$,
while 
 $w\mapsto w^*$
is  a bijection $\cR(\pi) \to \cR(\pi^*)$.
Also, if $\pi \in I_\ZZ$ then $\pi^* \in I_\ZZ$ and $\iR(\pi^*) = \{w^* : w \in \iR(\pi)\}$.

An element $\pi \in I_\ZZ$ is inv-Grassmannian if and only if 
$\pi^*$ is \emph{I-Grassmannian} 
in the sense of \cite[\S4.1]{HMP4}.
It follows from \cite[\S4.1]{HMP4}
that 
if
$\pi \in I_\ZZ$  is inv-Grassmannian of shape $\mu$ then $\ellhat(\pi) = |\mu|$.
On the other hand, 
the operator $\ifkb_\pi$ for $\pi \in I_\ZZ$ is the inverse of the  
\emph{involution Little map} in \cite[\S3.3]{HMP3}.
After adjusting for these symmetries, the results in \cite{HMP4,HMP3,Marberg2019a} imply the following:

\begin{lemma}[\cite{HMP4,HMP3,Marberg2019a}]
\label{i-gr-lem}
Suppose $w$ is an involution word for some permutation in $I_\ZZ$
and $\pi \in I_\ZZ$ is an inv-Grassmannian involution of shape $\mu$.
\ben
\item[(a)] There is 
 a finite
sequence $\sigma_1,\sigma_2,\dots,\sigma_l \in I_\ZZ$
such that
$ \ifkb_{\sigma_1}\ifkb_{\sigma_2}\cdots \ifkb_{\sigma_l}(w)$ is an involution word for an inv-Grassmannian element of $I_\ZZ$.

\item[(b)] The set $\iR(\pi)$ is a single equivalence class under $\simICK$, and 
 $\QO$ is a bijection from $\iR(\pi)$
to the set of standard shifted tableaux of shape $\mu$.

\een
\end{lemma}

\begin{proof}
The inverse of $\ifkb_\sigma$ for $\sigma \in I_\ZZ$
is the operator
given by  $\ifkb_\sigma^{-1} : v \mapsto \(\ifkb_{\sigma^*}(v^*)\)^*$.
Combining \cite[Theorem 3.40]{HMP3} (where $\ifkb_\sigma^{-1}$ is denoted $\hat\cB_\sigma$) and \cite[Theorem 4.36]{HMP4}
shows that for any involution word $v$,
there 
exists 
 a finite
sequence  $y_1,y_2,\dots,y_l \in I_\ZZ$
and an inv-Grassmannian involution $\gamma \in I_\ZZ$ with
$
(\ifkb_{y_1^*}\ifkb_{y_2^*}\cdots \ifkb_{y_l^*}(v^*))^* =
 \ifkb^{-1}_{y_1}\ifkb^{-1}_{y_2}\cdots \ifkb^{-1}_{y_l}(v) \in \iR(\gamma^*).
 $
Taking $v=w^*$ and $\sigma_i = y_i^*$ for $i \in [l]$ 
 proves part (a). 
 For part (b),
 we observe that  \cite[Corollaries 5.9 and 5.10]{Marberg2019a} and
 \cite[Theorem 4.20]{HMP4} imply
 that exactly one increasing shifted tableau $P$ exists with $\row(P) \in \iR(\pi)$,
 and this tableau has shape $\mu$. The desired claims hold by
 Theorem~\ref{oeg-thm} 
  and \cite[Corollary 4.12]{Marberg2019a}, which asserts that 
  if involution words $v$ and $w$ have $\PO(v) = \PO(w)$ then $v\simICK w$.
\end{proof}

This lemma lets us prove Theorem~\ref{o-eg-thm}
from Section~\ref{main-sect}.

\begin{proof}[Proof of Theorem~\ref{o-eg-thm}]
Let $v$ and $w$ be involution words.
If $\PO(v) = \PO(w)$, then  $v\simICK w$
holds
by \cite[Corollary 4.12]{Marberg2019a}.
We must prove the converse.
Assume $v\simICK w$.
By Lemma~\ref{i-gr-lem}(a),
a finite sequence of operators 
$ \ifkb_{\sigma_1},\ifkb_{\sigma_2},\dots, \ifkb_{\sigma_l}$
transforms $v$ to 
 $\tilde v \in \iR(\pi)$ 
for some inv-Grassmannian permutation $\pi \in I_\ZZ$.
Let $\tilde w$ be the word obtained by applying the same operators to $w$.
Theorem~\ref{o-little-thm} implies that
$\QO(v) = \QO(\tilde v)$,
$\QO(w) = \QO(\tilde w)$,
and $\tilde v \simICK \tilde w$.
Therefore $\tilde w \in \iR(\pi)$,
so by Lemma~\ref{i-gr-lem}(b)
the shifted tableaux $\QO(v)$ and $\QO(w)$ must have the same shape.
Hence, by Theorem~\ref{oeg-thm},
there is a unique involution word $u$ such that $\PO(u) = \PO(v)$ 
and $\QO(u) = \QO(w)$. It suffices to show that $u=w$.
Let $\tilde u =\ifkb_{\sigma_1}\ifkb_{\sigma_2}\cdots \ifkb_{\sigma_l}(u)$.
Since $u\simICK v \simICK w$, 
Theorem~\ref{o-little-thm} implies that $\tilde u \in \iR(\pi)$ 
and $\QO(\tilde u) = \QO(u) =\QO(w)= \QO(\tilde w) $.
It therefore follows from Lemma~\ref{i-gr-lem}(b) that $\tilde u=\tilde w$,
so $u=w$ 
since each $\ifkb_\sigma$ is injective.
\end{proof}

Let $\fk t_m$ for $m\in \ZZ$ be the operator acting on words $w=w_1w_2\cdots w_n$
 and  maps $\pi : \ZZ \to \ZZ$
by
 \be\label{t-def}
\fk t_m(w) = (w_1+m)(w_2+m)\cdots (w_n+m)
\quand
 \fk t_m (\pi) : i \mapsto \pi(i - m) + m.
 \ee
  If $w=(w^1,w^2,\dots,w^n)$ is an $n$-tuple of words then set
 $\fk t_m(w) = (\fk t_m(w^1), \fk t_m(w^2),\dots, \fk t_m(w^n)).$
As an operator on words, $\fk t_m$ preserves descents and commutes with $\ock$ and $\ck_i$ for all $i> 0$.
 If $ \pi \in I_\ZZ$ is any involution,
then $\fk t_m(\pi) \in I_\ZZ$ and the map $w \mapsto \fk t_m(w)$ is 
obviously an isomorphism of abstract $\q_n$-crystals
$\iR_n(\pi) \to \iR_n(\fk t_m(\pi)) $.
  Moreover, we clearly have $\QO(w) = \QO(\fk t_m(w))$ for all $w \in \iR_n(\pi)$
  and 
 $\fk t _m \ifkb_\pi = \ifkb_{\fk t_m(\pi)} \fk t_m$ for all $\pi \in I_\ZZ$.
  
  We can now upgrade Lemma~\ref{i-gr-lem} to the following statement.
  
 \begin{lemma}\label{o-final-lem}
If $w$ is an involution word
then there is  
 a finite
sequence  $\sigma_1,\sigma_2,\dots,\sigma_l \in I_\ZZ$
and an integer $m \in \ZZ$
such that the word
$ \ifkb_{\sigma_1}\ifkb_{\sigma_2}\cdots \ifkb_{\sigma_l}\fk t_m(w)$ is a permutation of $2,4,6,\dots,2\ell(w)$.
\end{lemma}

\begin{proof}
Let $\mu = (\mu_1>\mu_2>\dots>\mu_r>0)$ be a strict partition of $n$
and consider the inv-Grassmannian permutation
$
\pi_\mu := (1,r+\mu_r)(2,r+\mu_{r-1})\cdots(r,r+\mu_1) \in I_\ZZ.
$
In view of Theorem~\ref{o-little-thm}(c) and Lemma~\ref{i-gr-lem},
it suffices to produce
 a finite
sequence  $\sigma_{\mu,1},\sigma_{\mu,2},\dots,\sigma_{\mu,l} \in I_\ZZ$
such that 
$ \ifkb_{\sigma_{\mu,1}}\ifkb_{\sigma_{\mu,2}}\cdots \ifkb_{\sigma_{\mu,l}}(v)\in \cSe(n)$
for some (and therefore every) word $v \in \iR(\pi_\mu)$.

If $r=0$ then we set $l=0$.
Suppose $r>0$,
let $q=\mu_1\geq r$,
and write $h_i$ for the number of boxes in the $i$th  column
of $\SD_\mu$.
Let $\nu$ be the strict partition with $\SD_\nu = \SD_\mu\setminus\{(h_q,q)\}$.
Assume by induction that
$\sigma_{\nu,1},\sigma_{\nu,2},\dots,\sigma_{\nu,k} \in I_\ZZ$ are given  such that 
$ \ifkb_{\sigma_{\nu,1}}\ifkb_{\sigma_{\nu,2}}\cdots \ifkb_{\sigma_{\nu,k}}(v)\in \cSe(n-1)$
for all involution words $v \in \iR(\pi_\nu)$.
We claim that the desired sequence is 
\be\label{sigma-eq}
(\sigma_{\mu,1},\sigma_{\mu,2},\dots,\sigma_{\mu,l})
=
(s_{2n}\sigma_{\nu,1},\ s_{2n}\sigma_{\nu,2},\ \dots,\ s_{2n}\sigma_{\nu,k},\  \underbrace{\pi_\nu,\  \pi_\nu,\ \dots,\ \pi_\nu}_{2n-q-r+1\text{ times}}).
\ee
To prove this, let
\[ w^i :=\begin{cases}
 (2i-1)(2i-2)\cdots i&\text{for $1\leq i \leq r$} \\
(r+i-1)(r+i-2)\cdots(r+i-h_i)&\text{for $r <i \leq q$}
\end{cases}\]
and define $w := w^1w^2\cdots w^q$.
Using Lemma~\ref{rtimes-lem}, one can check that $w \in \iR(\pi_\mu)$.
We claim more specifically 
that if $\sigma_{\mu,i}$ is defined as in \eqref{sigma-eq} then
\be\label{spec-claim-eq}
\ifkb_{\sigma_{\mu,1}}\ifkb_{\sigma_{\mu,2}}\cdots \ifkb_{\sigma_{\mu,l}}(w)
=
a^1 a^2 \cdots a^q \in \cSe(n)
\ee
where
$
a^i 
 $
 is the word formed by adding $2h_1 + 2h_2 + \dots +2 h_{i-1}$ to $(2h_i)\cdots 642$.

If $q>r$ then 
the subword of $w$ with the last letter omitted belongs to $\iR(\pi_\nu)$,
while if $q=r$ then the subword of $w$ with the largest letter $2r-1$ omitted is in $\iR(\pi_\nu)$.
For each $i \in \NN$, define 
\[
u^i := \begin{cases} (q+r+i)(q+r-1)(q+r-2)\cdots(q+r-h_q+1) & \text{if }q>r \\
(2r+i)(2r-2)(2r-3)\cdots r &\text{if }q=r
\end{cases}
\]
and let $w_{(i)} := w^1w^2\cdots w^{q-1} u^i$.
It is a straightforward exercise to show that $w_{(0)} = \ifkb_{\pi_\nu}(w)$
and
$w_{(i)} = \ifkb_{\pi_\nu}(w_{(i-1)})$ for each $i >0$.
Thus, the word
$v' := (\ifkb_{\pi_\nu})^{2n - q-r + 1}(w)$ is obtained from $w$ by replacing either its last letter or its largest letter by $2n$,
and then moving $2n$ to be in position $n-h_q + 1$.
Moreover, removing $2n$ from $v$ yields a word $v \in \iR(\pi_\nu)$.

Let $\sigma'_{\nu,i} := s_{2n} \sigma_{\nu,i}$.
Our key observation is now that
the words 
$ \ifkb_{\sigma_{\nu,i}}\ifkb_{\sigma_{\nu,i+1}}\cdots \ifkb_{\sigma_{\nu,k}}(v)$
and
$ \ifkb_{\sigma'_{\nu,i}}\ifkb_{\sigma'_{\nu,i+1}}\cdots \ifkb_{\sigma'_{\nu,k}}(v')$
have exactly the same relationship as $v$ and $v'$ for all $i \in [k]$: the second word is the same as the first but with 
 $2n$ inserted in
position $n-h_q + 1$. 
Given this fact, the desired identity \eqref{spec-claim-eq} follows by induction,
which completes our proof of the lemma.
 \end{proof}

Given $n$-tuples of words $u=(u^1,u^2,\dots,u^n)$ and $v=(v^1,v^2,\dots,v^n)$,
write $u \simICK v$ if it holds that $u^1u^2\cdots u^n \simICK v^1v^2\cdots v^n$.
Using all of our results so far, we can now prove Theorem~\ref{main-thm1}. 

\begin{proof}[Proof of Theorem~\ref{main-thm1}]
Recall that $\pi \in I_\ZZ$. Choose a factorization $w \in \iR_n(\pi)$
and let $\cC$ be the full $\q_n$-subcrystal of $\iR_n(\pi)$ containing $w$.
From Theorem~\ref{simCK-thm}, Theorem~\ref{eg-thm}(a), and the definitions of $\fO$ and $\eO$,
it is clear that a factorization $v \in \iR_n(\pi)$ belongs to $\cC$ only if $v\simICK w$.
Part (a) of Theorem~\ref{main-thm1} is equivalent to the converse statement, which is not yet evident.

By Lemma~\ref{o-final-lem},
we have 
$ \ifkb_{\sigma_1}\ifkb_{\sigma_2}\cdots \ifkb_{\sigma_l}\fk t_m(w) \in \cSe_n(\ellhat(\pi))$
for some $\sigma_1,\sigma_2,\dots,\sigma_l \in I_\ZZ$
and  $m \in \ZZ$.
Write $\fk i :=  \ifkb_{\sigma_1}\ifkb_{\sigma_2}\cdots \ifkb_{\sigma_l}\fk t_m$.
Since each $\ifkb_\sigma$ and $\fk t_m$ preserves $\simICK$,
it follows that $\fk i(v) \simICK \fk i(w)$ and
$ \fk i(v) \in \cSe_n(\ellhat(\pi))$ for all $v \in \cC$.
Thus, by Theorem~\ref{o-little-thm} and Lemma~\ref{icc-lem}, the diagram
\[
\begin{tikzcd}[row sep=large, column sep=large]
\cC 
\arrow[r, "\fk i"]
\arrow[rd, swap, "\QO"]
&
\cSe_n(\ellhat(\pi))
\arrow[r,"\dbl^{-1}"]
\arrow[d,swap, "\QO"]
&
\cS_n(\ellhat(\pi))
\arrow[r,"\invert"]
\arrow[dl,swap, "\QO"]
&
\cW_n(\ellhat(\pi))
\arrow[dll,"P_\HM"]
\\
& \STab_n(\ellhat(\pi))
\end{tikzcd}
\]
commutes.
The map $\fk i : \cC \to \cS_n(\ellhat(\pi))$ is a quasi-isomorphism of abstract $\q_n$-crystals by Theorem~\ref{qi-thm2},
both $\dbl^{-1}$ and $\invert $ are isomorphisms by Lemma~\ref{icc-lem}, and 
the map
$P_\HM : \cW_n(m) \to \STab_n(m)$
is a quasi-isomorphism by construction.
Therefore $\QO : \iR_n(\pi) \to \STab_n(\ellhat(\pi))$ is a quasi-isomorphism of abstract $\q_n$-crystals.
This proves (b). 

To prove (a),
suppose $v \in \iR_n(\pi)$ has $v\simICK w$. Let $\tilde v :=  \dbl^{-1}\circ \fk i(v)$ and
$\tilde w :=  \dbl^{-1}\circ \fk i(w)$.
Then $\tilde v \simICK \tilde w$, so $Q_\HM(\tilde v^{-1}) = \PO(\tilde v ) = \PO(\tilde w) = Q_\HM(\tilde w^{-1})$.
By Theorem-Definition~\ref{stab-thmdef}, we deduce that $\tilde v^{-1}$ and $\tilde w^{-1}$ belong to the same full $\q_n$-subcrystal of $\cW_n(\ellhat(\pi))$.
Since the maps $\fk i $, $\dbl^{-1}$, and $\invert$ send full subcrystals to full subcrystals,
we must have $v \in \cC$.
\end{proof}

\subsection{Proofs in the symplectic case}
\label{proofs-sp-sect}

This section contains our proofs of Theorems~\ref{sp-eg-thm} and \ref{main-thm2}.

\begin{definition}
Given $\pi \in \Ifpf_\ZZ$, define $\hat \pi \in I_\ZZ$ to be the involution with
\[ 
\hat \pi (i) = \begin{cases} i&\text{if no $j \in \ZZ$ with $\min\{i,\pi(i)\} < j < \max\{i,\pi(i)\}$ has $ j < \pi(j)$} \\
\pi(i)&\text{otherwise}
\end{cases}
\]
for each $i \in \ZZ$.
An element $\pi \in \Ifpf_\infty\setminus\{\minfpf\}$ is \emph{fpf-Grassmannian} if $\hat \pi \in I_\ZZ$ is inv-Grassmannian.
In this case, if $\hat \pi$ 
has shape $\mu = (\mu_1 ,\mu_2,\dots,\mu_r)$,
then the \emph{shape} of $\pi$ is  $\nu = (\mu_1-1,\mu_2-1,\dots,\mu_r-1)$.
We also consider $\minfpf$ to be fpf-Grassmannian with shape $\nu=\emptyset$.
\end{definition}

To recover $\pi \in \Ifpf_\ZZ$ from $\hat \pi \in I_\ZZ$, let $i \mapsto f_i$ be an order-preserving map from $\ZZ$
 to the fixed-points of $\hat \pi$ such that $ i\equiv f_i \modu 2)$ for all sufficiently large $i$.
 Then $\pi(f_i) = f_{i+1}$ and $\pi(f_{i+1}) = f_i$ for all odd $i \in \ZZ$, while $\pi(i) = \hat \pi(i)$ for all $\hat \pi(i)\neq i \in \ZZ$.

If $\pi \in \Ifpf_\ZZ$ then $\pi^* : i \mapsto 1 - \pi(1-i)$ is an element of $\Ifpf_\ZZ$
and $w\mapsto w^*$ is again a bijection $\iRfpf(\pi) \to \iRfpf(\pi^*)$.
An involution $\pi \in \Ifpf_\ZZ$ is fpf-Grassmannian if and only if 
$\pi^*$ is \emph{FPF-Grassmannian} 
in the sense of \cite[\S4]{HMP5}.

\begin{lemma}[\cite{HMP5,HMP3,Marberg2019a}]
\label{fpf-gr-lem}
Suppose $w$ is an fpf-involution word for some permutation in $\Ifpf_\ZZ$
and $\pi \in \Ifpf_\ZZ$ is an fpf-Grassmannian involution of shape $\nu$.
\ben
\item[(a)] There is 
 a finite
sequence $\sigma_1,\sigma_2,\dots,\sigma_l \in \Ifpf_\ZZ$
such that
$ \ffkb_{\sigma_1}\ffkb_{\sigma_2}\cdots \ffkb_{\sigma_l}(w)$ is an fpf-involution word for an fpf-Grassmannian element of $\Ifpf_\ZZ$.

\item[(b)] The set $\iRfpf(\pi)$ is a single equivalence class under $\simFCK$, and 
 $\QSp$ defines a bijection from $\iRfpf(\pi)$
to the set of standard shifted tableaux of shape $\nu$.

\een
\end{lemma}

\begin{proof}
The structure of the proof is the same as for Lemma~\ref{i-gr-lem}.
The inverse of $\ffkb_\sigma$ for $\sigma \in \Ifpf_\ZZ$
is the operator   
 $\ffkb_\sigma^{-1} : v \mapsto \(\ffkb_{\sigma^*}(v^*)\)^*$.
Combining \cite[Theorem 4.29]{HMP3} 
(where $\ffkb_\sigma^{-1}$ is denoted $\hat\cB^{{\tt FPF}}_\sigma$)
and \cite[Theorem 5.21]{HMP5}
shows that for any fpf-involution word $v$,
there 
exists 
 a finite
sequence $y_1,y_2,\dots,y_l \in \Ifpf_\ZZ$
and an fpf-Grassmannian involution $\gamma \in \Ifpf_\ZZ$ such that
$
(\ffkb_{y_1^*}\ffkb_{y_2^*}\cdots \ffkb_{y_l^*}(v^*))^* =
 \ffkb^{-1}_{y_1}\ffkb^{-1}_{y_2}\cdots \ffkb^{-1}_{y_l}(v) \in \iRfpf(\gamma^*).
 $
Taking $v=w^*$ and $\sigma_i = y_i^*$ for $i \in [l]$ 
 proves part (a). 
Next,
 \cite[Corollaries 5.9 and 5.10]{Marberg2019a}
 and
 \cite[Lemma 4.16 and Theorem 4.19]{HMP5} imply
 that exactly one increasing shifted tableau $P$ exists with $\row(P) \in \iRfpf(\pi)$,
 and this tableau has shape $\nu$. Part (b) therefore follows from
 Theorem~\ref{speg-thm} 
 and \cite[Corollary 3.22]{Marberg2019a}.
\end{proof}

We can give a second proof of Theorem~\ref{sp-eg-thm}  using the preceding lemma.

\begin{proof}[Proof of Theorem~\ref{sp-eg-thm}]
If $v$ and $w$ are fpf-involution words 
with $\PSp(v) = \PSp(w)$, then \cite[Corollary 3.22]{Marberg2019a} implies that
$v\simFCK w$.
The converse
follows by same argument as in the proof of Theorem~\ref{o-eg-thm},
using Theorem~\ref{sp-little-thm} and Lemma~\ref{fpf-gr-lem}
in place of Theorem~\ref{o-little-thm} and Lemma~\ref{i-gr-lem},
and substituting the respective symbols $\simFCK$, $\PSp$, $\QSp$, $\Ifpf_\ZZ$, $\iRfpf$, $\ffkb_\pi$
for 
$\simICK$, $\PO$, $\QO$, $I_\ZZ$, $\iR$, $\ifkb_\pi$.
\end{proof}

Recall the operator $\fk t_m$  from \eqref{t-def}.
We have $\fk t_m(\Ifpf_\ZZ) = \Ifpf_\ZZ$ if and only if $m \in \ZZ$ is even.
Assume this is the case; then it is easy to see that  $\fk t_m$ commutes with $\spck$
and 
that  $w \mapsto \fk t_m(w)$ is 
an isomorphism of abstract $\q_n$-crystals
$\iRfpf_n(\pi) \to \iRfpf_n(\fk t_m(\pi)) $ for all $\pi \in \Ifpf_\ZZ$.
  Moreover, we clearly have $\QSp(w) = \QSp(\fk t_m(w))$ for all $w \in \iRfpf_n(\pi)$
  and $\fk t _m \ffkb_\pi = \ffkb_{\fk t_m(\pi)} \fk t_m$.
  
  There is a symplectic analogue of Lemma~\ref{o-final-lem}.
  
 \begin{lemma}\label{sp-final-lem}
Suppose $w$ is an fpf-involution word. 
There is  
 a finite
sequence  $\sigma_1,\sigma_2,\dots,\sigma_l \in \Ifpf_\ZZ$
and an even integer $m \in 2\ZZ$
such that 
$ \ffkb_{\sigma_1}\ffkb_{\sigma_2}\cdots \ffkb_{\sigma_l}\fk t_m(w)$ is a permutation of $2,4,6,\dots,2\ell(w)$.
\end{lemma}

\begin{proof}
Let $\mu = (\mu_1>\mu_2>\dots>\mu_r>0)$ be a strict partition of $n$
and define $\theta_\mu $ to be the element of $\Ifpf_\ZZ$ with
$
\hat\theta_\mu = (1,r+1+\mu_r)(2,r+1+\mu_{r-1})\cdots(r,r+1+\mu_1)
$
so that $\theta_\mu$ is fpf-Grassmannian with shape $\mu$.
It suffices by Theorem~\ref{sp-little-thm}(c) and Lemma~\ref{fpf-gr-lem}
to construct a finite sequence $\sigma_{\mu,1} \sigma_{\mu,2},\dots,\sigma_{\mu,l} \in \Ifpf_\ZZ$
such that $ \ffkb_{\sigma_{\mu,1}}\ffkb_{\sigma_{\mu,2}}\cdots \ffkb_{\sigma_{\mu,l}}(v)\in \cSe(n)$
for some (and therefore every) word $v \in \iRfpf(\theta_\mu)$.

Our argument is similar to one in the proof of Lemma~\ref{o-final-lem}.
If $r=0$ then we again set $l=0$.
Suppose $r>0$,
let $q=\mu_1$,
write $h_i$ for the number of boxes in the $i$th column of $\SD_\mu$, and 
define $\nu$ to be the strict partition with $\SD_\nu = \SD_\mu\setminus\{(h_q,q)\}$.
Let $\epsilon \in \{0,1\}$ be such that $q+r+\epsilon$ is odd,
and assume  that
$\sigma_{\nu,1},\sigma_{\nu,2},\dots,\sigma_{\nu,k} \in \Ifpf_\ZZ$ are given  such that 
$ \ffkb_{\sigma_{\nu,1}}\ffkb_{\sigma_{\nu,2}}\cdots \ffkb_{\sigma_{\nu,k}}(v)\in \cSe(n-1)$
for all fpf-involution words $v \in \iRfpf(\theta_\nu)$.
We claim that the desired sequence is 
\be\label{fpf-sigma-eq}
(\sigma_{\mu,1},\sigma_{\mu,2},\dots,\sigma_{\mu,l})
=
(\sigma'_{\nu,1},\ \sigma'_{\nu,2},\ \dots,\ \sigma'_{\nu,k},\  \underbrace{\pi_\nu,\  \pi_\nu,\  \pi_\nu,\ \dots,\ \pi_\nu}_{n - \frac{1}{2}(q +r  + \epsilon - 1)\text{ times}})
\ee
where $\sigma'_{\nu,i} := s_{2n}\cdot \sigma_{\nu,u} \cdot s_{2n}$.
To prove this, let
\[w^i := 
\begin{cases}
(2i)(2i-1)\cdots (i+1) &\text{for $1\leq i \leq r$} \\
(r+i)(r+i-1)\cdots(r+i-h_i+1)&\text{for $r <i \leq q$}
\end{cases}
\]
and set $w := w^1w^2\cdots w^q$.
This is the same as the analogous word in the proof of Lemma~\ref{o-final-lem}, but with all letters incremented by one.
Using Lemma~\ref{fpf-rtimes-lem},
one can check that $w \in \iRfpf(\theta_\mu)$.
Then, specifically,
we claim that if $\sigma_{\mu,i}$ is defined as in \eqref{fpf-sigma-eq} then
\be\label{fpf-spec-claim-eq}
\ffkb_{\sigma_{\mu,1}}\ffkb_{\sigma_{\mu,2}}\cdots \ffkb_{\sigma_{\mu,l}}(w)
=
 a^1 a^2 \cdots a^q \in \cSe(n)
\ee
where
$
 a^i 
 $
 is again the word formed by adding $2h_1 + 2h_2 + \dots + 2h_{i-1}$ to $(2h_i)\cdots 642$.

If $q>r$ then 
the subword of $w$ with the last letter omitted belongs to $\iRfpf(\theta_\nu)$,
while if $q=r$ then the subword of $w$ with the largest letter $2r$ omitted is in $\iRfpf(\theta_\nu)$.
For each $i \in \NN$, define 
\[
u^i := \begin{cases} (q+r+2i + \epsilon + 1)(q+r)(q+r-1)\cdots(q+r-h_q+2) & \text{if }q>r \\
(2r+2i+2)(2r-1)(2r-2)\cdots (r+1) &\text{if }q=r
\end{cases}
\]
and let $w_{(i)} := w^1w^2\cdots w^{q-1} u^i$.
Then $w_{(0)} = \ffkb_{\theta_\nu}(w)$
and
$w_{(i)} = \ffkb_{\theta_\nu}(w_{(i-1)})$ for each $i >0$,
so we deduce that the word
$v' := (\ffkb_{\theta_\nu})^{n - \frac{1}{2}(q +r  + \epsilon - 1)}(w)$ is obtained from $w$ by replacing either its last letter or its largest letter by $2n$,
and then moving $2n$ to be in position $n-h_q + 1$.
Moreover, removing $2n$ from $v'$ yields a word $v \in \iRfpf(\theta_\nu)$.
 
As in the proof of Lemma~\ref{o-final-lem},
we now observe that the words
$ \ffkb_{\sigma_{\nu,i}}\ffkb_{\sigma_{\nu,i+1}}\cdots \ffkb_{\sigma_{\nu,k}}(v)$
and
$ \ffkb_{\sigma'_{\nu,i}}\ffkb_{\sigma'_{\nu,i+1}}\cdots \ffkb_{\sigma'_{\nu,k}}(v')$
have exactly the same relationship as $v$ and $v'$ for all $i \in [k]$: the second word is the same as the first but with 
the letter $2n$ inserted in
position $n-h_q + 1$. 
From this, the desired identity \eqref{fpf-spec-claim-eq} follows immediately by induction,
which completes the proof.
 \end{proof}

We can now prove our last main result from Section~\ref{main-sect}.

\begin{proof}[Proof of Theorem~\ref{main-thm2}]
Our argument is similar to the proof of Theorem~\ref{main-thm1}.
We extend the definition of $\simFCK$ from words to $n$-tuples of words exactly as we did with $\simICK$.
Recall that $\pi \in \Ifpf_\ZZ$. Choose a factorization $w \in \iRfpf_n(\pi)$
and let $\cC$ be the full $\q_n$-subcrystal of $\iRfpf_n(\pi)$ containing $w$.
It is again straightforward to check that every $v \in \cC$ has $v\simFCK w$.

By Lemma~\ref{sp-final-lem},
we have 
$ \ffkb_{\sigma_1}\ffkb_{\sigma_2}\cdots \ffkb_{\sigma_l}\fk t_m(w) \in \cSe_n(\ellfpf(\pi))$
for some $\sigma_1,\sigma_2,\dots,\sigma_l \in \Ifpf_\ZZ$
and  $m \in 2\ZZ$.
Write $\fk f :=  \ffkb_{\sigma_1}\ffkb_{\sigma_2}\cdots \ffkb_{\sigma_l}\fk t_m$.
Since each $\ffkb_\sigma$ and $\fk t_m$ preserves $\simFCK$,
it follows that 
$ \fk f(v) \in \cSe_n(\ellfpf(\pi))$ for all $v \in \cC$.
Therefore, by Theorem~\ref{sp-little-thm} and Lemma~\ref{icc-lem}, the diagram
\[
\begin{tikzcd}[row sep=large, column sep=large]
\cC 
\arrow[r, "\fk f"]
\arrow[rd, swap, "\QSp"]
&
\cSe_n(\ellfpf(\pi))
\arrow[r,"\dbl^{-1}"]
\arrow[d,swap, "\QSp"]
&
\cS_n(\ellfpf(\pi))
\arrow[r,"\invert"]
\arrow[dl,swap, "\QO"]
&
\cW_n(\ellfpf(\pi))
\arrow[dll,"P_\HM"]
\\
& \STab_n(\ellfpf(\pi))
\end{tikzcd}
\]
commutes.
The map $\fk f : \cC \to \cSe_n(\ellfpf(\pi))$ is a quasi-isomorphism of abstract $\q_n$-crystals by Theorem~\ref{qi-thm3},
both $\dbl^{-1}$ and $\invert$ are isomorphisms by Lemma~\ref{icc-lem}, and 
the map
$P_\HM : \cW_n(m) \to \STab_n(m)$
is a quasi-isomorphism for all $m$ by construction.
We conclude that $\QSp : \iRfpf_n(\pi) \to \STab_n(\ellfpf(\pi))$ is a quasi-isomorphism of abstract $\q_n$-crystals.
This proves (b). 

To prove (a),
suppose $v \in \iRfpf_n(\pi)$ has $v\simFCK w$. 
Then $\fk f(v) \simFCK \fk f(w)$, so $\fk f(v) \in \cSe_n(\ellfpf(w))$ and
$\dbl^{-1} \circ \fk f(v) \simICK \dbl^{-1} \circ  \fk f(w)$ since $\fk f(w)$ has only even letters.
As in the proof of Theorem~\ref{main-thm2},
we deduce 
 that $\invert\circ \dbl^{-1} \circ \fk f(v) $ and $\invert\circ \dbl^{-1} \circ \fk f(w)$ are in the same full $\q_n$-subcrystal of $\cW_n(\ellfpf(\pi))$,
so $v \in \cC$ as the maps $\fk f$, $\dbl^{-1}$, and $\invert$ send full subcrystals to full subcrystals.
\end{proof}

\subsection{Dual equivalence operators}\label{dual-sect}

As an application of Theorems~\ref{main-thm1} and \ref{main-thm2},
we can describe precisely how the Coxeter-Knuth operators interact with the 
orthogonal- and symplectic-EG-recording tableaux.

Consider a standard shifted tableau $T$ with $n$ boxes.
Given $i\in[n]$, let $\square_i$ be the unique box of $T$ containing $i$ or $i'$.
For each index $i \in [n-1]$,
let $s_i \star T$ 
be formed from $T$ as follows:
\begin{itemize}
\item If $\square_i$ and $\square_{i+1}$ are in the same row or same column
then do both of the following:
\begin{itemize}
\item Interchange $i$ and $i'$ if the box $\square_i$ is not on the main diagonal.
\item Interchange $i+1$ and $i+1'$ if the box $\square_{i+1}$ is not on the main diagonal. 
\end{itemize}

\item Otherwise, interchange $i$ and $i+1$ and then interchange $i'$ and $i+1'$.
\end{itemize}
Although $s_i \star (s_i \star T) = T$,
this operation does not extend to an action of the symmetric group.
We always have $\square_1=(1,1)$ and  $\square_2=(1,2)$,
so $s_1\star T$ is obtained from $T$ by interchanging $2$ and $2'$.

Choose an integer $q>0$ such that the domain of $T$ is a subset of $[q]\times [q]$.
For  $i \in [q]$ let $C_i$ be the sequence of primed entries in column $i$ of $T$, read in order,
and let $R_i$ be the sequence of unprimed entries in row $i$ of $T$, read in order.
The \emph{shifted reading word} of $T$ is the sequence $ \shword(T)$
formed by removing all primes from $C_qR_q\cdots C_3R_3C_2R_2C_1R_1$.
For example, if
\be\label{shword-eq} T= \ytab{ \none & \none & 7 \\ \none & 3 & 6' & 8 \\ 1 & 2' & 4' & 5 & 9}
\quad\text{then}\quad \shword(T) =467 238 159\ee
since
the nonempty sequences $C_iR_i$ are 
$C_1R_1 = 159$, $C_2R_2 = 2'38$, and $C_3R_3=4'6'7$.
%
\begin{definition} Let $\fks_i$ for $0\leq i \leq n-2$ be the operator on standard shifted tableaux
with 
\[ \fks_i (T) := \begin{cases}
s_{i} \star T &\text{if $i+2$ is between $i$ and $i+1$ in $\shword(T)$},
\\
s_{i+1} \star T &\text{if $i=0$ or if $i$ is between $i+1$ and $i+2$ in $\shword(T)$},
\\
T &\text{if $i+1$ is between $i$ and $i+2$ is $\shword(T)$}.
\end{cases}
\]
For convenience set $\fks_i(T):= T$ for all integers $i$ with $i+1 \notin[n-1]$.
\end{definition}
When we say ``$b$ is between $a$ and $c$'' in some word, we mean that 
 $abc$ or $cba$ occurs as a subword.
It is easy to show that if $T$ is standard then $\fks_i(T)$ is also standard
and that $\fks_i(\fks_i(T)) = T$.
The letter $\fks$ in this notation stands for \emph{dual equivalence operator}.

Define the \emph{descent set} of a standard shifted tableau $T$ to be
$\Des(T) := \Des(\shword(T))$, so that the tableau in \eqref{shword-eq} has 
$\Des(T) = \{1,3,5\}.$
It is a standard exercise to check that an integer $i$ belongs to $\Des(T)$ if and only if
either (a) $i$ and $i+1$ both appear in $T$ with $i+1$ in a row strictly after $i$,
(b) $i'$ and $i+1'$ both appear in $T$ with $i+1'$ in a column strictly after $i'$,
or (c) $i$ and $i+1'$ both appear in $T$.
If $w$ is an involution word then $\Des(w) = \Des(\QO(w))$
\cite[Proposition 2.24]{HKPWZZ}
and if $w$ is an fpf-involution word then $\Des(w) = \Des(\QSp(w))$
\cite[Theorem 4.4]{Marberg2019a}.

\begin{theorem}\label{dual-thm2}
Let $i$ be a positive integer.
\ben
\item[(a)] 
If $ w$ is an involution word for some element of $I_\ZZ$ then 
\[\QO(\ock(w)) = \fks_0 (\QO(w))\quand \QO(\ck_i(w)) = \fks_i (\QO(w)).\]

\item[(b)] If $ w$ is an fpf-involution word for some element of $\Ifpf_\ZZ$ then 
\[\QSp(\spck(w)) = \fks_0 (\QSp(w))
\quand
\QSp(\ck_i(w)) = \fks_i (\QSp(w)).\]

\een
\end{theorem}

\begin{proof}
The identities $\QO(\ock(w)) = \fks_0(\QO(w)) $
and
$\QSp(\spck(w)) = \fks_0 (\QSp(w))$ are easy to observe directly from the definitions of $\ock$, $\QO$,
$\spck$, and $\QSp$.

The remaining identities in part (a) are trivial unless $ 0<i \leq \ell(w)-2$ and exactly one of $i$ or $i+1$
is descent of $w$ (equivalently, $\QO(w)$),
since otherwise we have $\ck_i(w) = w$ and $\fks_i(\QO(w)) = \QO(w)$.
Assume this is the case and view $w$ as an orthogonal factorization 
by placing each letter in its own factor. Then, as explained in the proof of \cite[Proposition 10.14]{BumpSchilling}, exactly one of $f_if_{i+1} e_ie_{i+1}(w)$
or $f_{i+1}f_i e_{i+1}e_i(w)$ is nonzero and the nonzero factorization may be identified with $\ck_i(w)$.

Similarly, the remaining identities in part (b) are trivial unless $ 0<i \leq \ell(w)-2$ and exactly one of $i$ or $i+1$
is descent of $w$ (equivalently, $\QSp(w)$). When this is the case and we view $w$ 
as an symplectic factorization 
by placing each letter in its own factor, it follows by the same argument from
 \cite{BumpSchilling} that exactly one of $f_if_{i+1} e_ie_{i+1}(w)$
or $f_{i+1}f_i e_{i+1}e_i(w)$ is nonzero and the nonzero factorization may be identified with $\ck_i(w)$.

Since $\QO$ and $\QSp$ are crystal quasi-isomorphisms 
by Theorems~\ref{main-thm1} and \ref{main-thm2},
it suffices to check that if $T$ is standard shifted tableau with $n$ boxes and $|\Des(T) \cap \{i,i+1\}| = 1$,
then
exactly one of $f_if_{i+1} e_ie_{i+1}(T)$
or $f_{i+1}f_i e_{i+1}e_i(T)$ is nonzero and the nonzero shifted tableau is $\fks_i(T)$.
These assertions are immediate from Lemmas~\ref{dual-equiv-lem0} and \ref{dual-equiv-lem}
in the appendix. 
\end{proof}

These identities are shifted analogues of a similar 
formula for
$Q_\EG(\ck_i(w))$ in terms of $Q_\EG(w)$ when $w$ is any reduced word; see \cite[Definition 5.1.3 and Theorem 5.1.4]{Assaf19}, for example.

\begin{example} The word $w =2343$ is both an involution word and an fpf-involution word
(for different permutations). We have
\[\PO(w)=\PSp(w) = \ytab{
  \none & 4  \\
   2 & 3 & 4
}
\quand\QO(w)=\QSp(w) = \ytab{
  \none & 4  \\
  1 & 2 & 3 
}=:T.\] As predicted by the theorem, it holds that
\[
\ba
\QO(\ock(w))& =\QO(3243)
= \QSp(\spck(w)) = \QSp(2143)
=\
\ytab{
  \none & 4  \\
  1 & 2' & 3 
}\
=
\fks_0(T),
\\
\\[-10pt]
\QO(\ck_2(w))& =\QO(2434)
= \QSp(\ck_2(w)) = \QSp(2434)=\
\ytab{
  \none & 3  \\
  1 & 2 & 4
}\
=
\fks_2(T).
\ea
\]
\end{example}


Suppose $\sA$ is a finite set and $\Des$ is a map from $\sA$ to the set of subsets of $[n-1]$. 
A \emph{dual equivalence} for $\sA$ with respect to $\Des$ 
is a family of involutions $\{ \varphi_i : \sA \to\sA \}_{1<i<n}$,
called \emph{dual equivalence operators}, satisfying two technical conditions; see \cite[Definition 4.1]{Assaf_dual}. As explained in \cite[\S4.1]{Assaf_dual}, 
if one is given a dual equivalence on $\sA$,
then there is a natural way to turn $\sA$
into a \emph{dual equivalence graph} as axiomatized in \cite{Assaf_dual,Roberts}.

Hamaker and Young have shown that if $\pi\in S_\ZZ$ has $\ell(\pi)=n$, then taking $\varphi_i = \ck_{i-1}$ for $1<i<n$ gives a dual equivalence for  $\sA= \cR(\pi)$
with the usual descent set \cite[Theorem 3]{HamakerYoung}.
It follows that the same maps give a dual equivalence for $\iR(\pi)$ when $\pi \in I_\ZZ$ has $\ellhat(\pi) = n$ and for $\iRfpf(\pi)$ when $\pi \in \Ifpf_\ZZ$ has $\ellhat(\pi) = n$.
On the other hand, 
one can check that the operators $\fks_{i-1}$ for $1<i<n$
are the same as the involutions on standard shifted tableaux that Assaf denotes by $\psi_i$ in \cite[Definition 6.1]{Assaf14}.
The following theorem from \cite{Assaf14} is therefore a corollary of our results:

\begin{corollary}[{\cite[Theorem 6.3]{Assaf14}}]
Let $\mu$ be a strict partition of $n$.
The maps $\psi_i = \fks_{i-1}$ for $1<i<n$ give a dual equivalence for the 
set $\sA$ of standard shifted tableaux of shape $\mu$.
\end{corollary}

\begin{proof}
Let $\pi \in I_\ZZ$ be inv-Grassmannian 
of shape $\mu$.
It follows from
Lemma~\ref{i-gr-lem} and Theorem~\ref{dual-thm2}
that $\{\psi_i \}_{1<i<n}$ is the dual equivalence on $\sA$
induced by the bijection $\QO : \iR(\pi) \to \sA$. 
\end{proof}
 
\subsection{Open problems}

We mention some related questions and conjectures.
Little proved the following in \cite{Little}:

\begin{proposition}[{\cite[Lemma 5]{Little}}]
\label{once-lem}
Let $\pi \in S_\ZZ$, let $w=w_1w_2\cdots w_n$ be a reduced word, 
and suppose $\fkb_\pi(w) = \tilde w_1\tilde w_2\cdots \tilde w_n$.
Then $\tilde w_i - w_i \in \{0,1\}$ for all $i \in [n]$.
\end{proposition}

In turn, this general fact holds for Edelman-Greene insertion:

\begin{proposition}\label{once-prop}
Suppose $w=w_1w_2\cdots w_n$  and $\tilde w=\tilde w_1\tilde w_2\cdots \tilde w_n$ 
are reduced words 
with
$\tilde w_i - w_i \in \{0,1\}$ for all $i \in [n]$.
Then $Q_\EG(w) = Q_\EG(\tilde w)$.
\end{proposition}

\begin{proof}
It suffices to show that for each $i$, the tableau $P_\EG(\tilde w_1\tilde w_2\cdots \tilde w_i)$
is formed by adding one to a subset of entries in $P_\EG(w_1w_2\cdots w_i)$.
This follows by induction using the observations in Remark~\ref{insertion-remarks}. The details are left to the reader.
\end{proof}

Combining these propositions gives an immediate proof of Theorem~\ref{little-thm}(d).
Using the same sort of arguments, one can derive a similar property of orthogonal-EG insertion:

\begin{proposition}\label{qo-last-prop}
Suppose $w=w_1w_2\cdots w_n$  and $\tilde w=\tilde w_1\tilde w_2\cdots \tilde w_n$ 
are involution words 
with
$\tilde w_i - w_i \in \{0,1\}$ for all $i \in [n]$.
Then $\QO(w) = \QO(\tilde w)$.
\end{proposition}

\begin{proof}
Using Remark~\ref{insertion-remarks},
it is a straightforward exercise to show by induction
that
 $\PO(\tilde w_1\tilde w_2\cdots \tilde w_i)$
is formed by adding one to a subset of entries in $\PO(w_1w_2\cdots w_i)$ for all $i$,
and that a letter $\tilde w_i$ is row-inserted according to Definition~\ref{o-eg-def}
if and only if  $w_i$ is row-inserted.
We omit the details.
\end{proof}

Computations support the following analogue of Proposition~\ref{once-prop}. If this conjecture were true,
then we would get an immediate proof of the most difficult part of Theorem~\ref{o-little-thm}:

\begin{conjecture}
Let $\pi \in I_\ZZ$, let $w=w_1w_2\cdots w_n$ be an involution word, 
and suppose $\ifkb_\pi(w) = \tilde w_1\tilde w_2\cdots \tilde w_n$.
Then $\tilde w_i - w_i \in \{0,1\}$ for all $i \in [n]$.
\end{conjecture}

A weaker version of this conjecture seems to hold for the $\ffkb_\pi$ operators.
However, this does not lead to a simple proof of Theorem~\ref{sp-little-thm}(d)
since letters may be incremented twice.

\begin{conjecture}
Let $\pi \in \Ifpf_\ZZ$, let $w=w_1w_2\cdots w_n$ be an fpf-involution word, 
and suppose $\ffkb_\pi(w) = \tilde w_1\tilde w_2\cdots \tilde w_n$.
Then $\tilde w_i - w_i \in \{0,1,2\}$ for all $i \in [n]$.
\end{conjecture}

Little bumping operators
are naturally described in terms of \emph{wiring diagrams} for permutations; see \cite{HamakerYoung,Little}. Moreover, if $\pi\in S_\ZZ$ has a single descent,
then there is simple way of reading off the tableau $Q_\EG(w)$ for any $w \in \cR(\pi)$
from the associated wiring diagram \cite[Lemma 5]{HamakerYoung}.

\begin{problem}
Identify good definitions of wiring diagrams for (fpf-)involution words,
and describe the operators $\ifkb_\pi$ and $\ffkb_\pi$ in terms of these diagrams.
Is there is an efficient way to compute $\QO(w)$ (respectively, $\QSp(w)$)
from the wiring diagram associated to an involution word $w$ for  
an inv-Grassmannian (respectively, fpf-Grassmannian) permutation?
\end{problem}

Involution words arise naturally 
when studying the combinatorics of the 
$\O_n$- and $\Sp_n$-actions on the type A flag variety $\Fl_n$.
 These actions correspond to two of the three families of type A symmetric varieties.
 The third family comes from the action of $\GL_p \times \GL_{n-p}$ on $\Fl_n$.
 There is a natural weak order on the corresponding set of orbits \cite{CJW}.
The maximal chains in this order give another variant of reduced words,
which are studied under the name \emph{clan words} in \cite{BurksPawlowski}.

\begin{problem}
Are there interesting crystal structures on factorizations of clan words?
Is there an analogue of Edelman-Greene insertion for clan words that can be interpreted as a crystal morphism?
\end{problem}

Morse and Schilling's original aim in \cite{MorseSchilling} was to identify crystal structures on
\emph{cyclically decreasing} factorizations of reduced words for affine permutations.
They describe an abstract $\gl_2$-crystal  on the subset of such factorizations with exactly two factors \cite[Theorem 3.14]{MorseSchilling}. It remains an open problem to further extend the constructions in Section~\ref{reduced-sect} to the affine case.

There are versions of involution words for affine permutations, with many of the same combinatorial
properties as involution words for elements of $I_\ZZ$ \cite{Marberg2019,MZ2018,Zhang2019}. This suggests the following:
 
 \begin{problem}
Are there interesting crystal structures on factorizations of reduced words or involution words for affine permutations?
Is there an affine analogue of Edelman-Greene insertion that can be interpreted as a crystal morphism?
\end{problem}

The papers \cite{Assaf14,BHRY} develop a notion of \emph{shifted dual equivalence graphs} based on standard shifted tableaux with no primed entries. The set of reduced words for a signed permutation,
connected by type B Coxeter-Knuth moves, is an example of such a graph \cite[Theorem 1.3]{BHRY}.
The results in Section~\ref{dual-sect} suggest the existence of another interesting kind of shifted dual equivalence.

\begin{problem}
Is there a version of shifted dual equivalence graphs based on standard shifted tableaux with primed entries that includes $\iR_n(\pi)$ and $\iRfpf_n(\pi)$ as examples?
\end{problem}

\appendix
\def\unpaired{\mathsf{unpaired}}

\section{Crystal operators on shifted tableaux}\label{app-sect}

Fix a positive integer $n$ and a strict partition $\lambda$ with at most $n$ parts.
By Theorem-Definition~\ref{stab-thmdef}, 
the set $\STab_n(\lambda)$ 
 of semistandard shifted tableaux of shape $\lambda$ with all entries at most $n$
 has a $\q_n$-crystal structure.
In this appendix
we review the explicit formulas from \cite{AssafOguz, HPS, Hiroshima2018} for the raising and lowering operators 
in
this crystal.
Recall that the weight map for $\STab_n(\lambda)$ is given by \eqref{stab-weight-eq}.
We include this material both for completeness and to assist the proof of Theorem~\ref{dual-thm2}.

The formulas below
have already appeared in 
\cite{AssafOguz, HPS, Hiroshima2018}.
The $\gl_n$-crystal operators $e_i$ and $f_i$ for $i\in[n-1]$ acting on $\STab_n(\lambda)$ were described first, in \cite{HPS}.
About a year later \cite{AssafOguz} and \cite{Hiroshima2018}  independently supplied the queer operators $f_{\overline 1}$ and $e_{\overline 1}$.
Our exposition mostly follows the conventions of \cite{AssafOguz}, which will let us 
correct some minor errors in the published version of that paper.
We continue to draw all tableaux in French notation.

\subsection{Shifted tableau pairing}

If $\mu$ and $\nu$ are strict partitions, then we write $\mu \subset \nu$ to indicate that $\SD_\mu \subset \SD_\nu$.
In this case we set $\SD_{\nu/\mu} := \SD_\nu \setminus \SD_\mu$
and define 
a \emph{skew shifted tableau}
of shape $\nu/\mu$ to be a map \[ \SD_{\nu/\mu} \to \tfrac{1}{2}\ZZ = \{ \dots < 1' < 1 < 2' < 2 < \dots\}.\]
If $T$ is a semistandard shifted tableau and $i\leq j$ are positive integers, 
then \[T^{-1}(\{ i' < i < \dots <j' <j\})\] is equal to $\SD_{\nu/\mu}$
for some strict partitions $\mu \subset \nu$,
and we write $T|_{[i,j]}$ for the skew shifted tableau obtained by restricting $T$ to this subdomain.

We say that a skew shifted tableau is a \emph{rim} if its domain  has no positions $(i_1,j_1), (i_2,j_2)$ 
with $i_1 < i_2 $ and $j_1 < j_2$. 
If $T$ is a semistandard shifted tableau then $T|_{[i,i]}$ is always a rim.
A rim whose domain is connected is a \emph{ribbon}.
In French notation, the domain of a ribbon must appear as
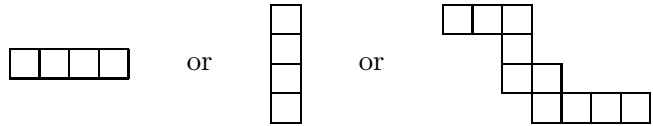
\[
\ytab{\ &\ &\ &\ }
\qquord
\ytab{\ \\ \ \\ \ \\ \ }
\qquord
\ytab{ \ & \ &  \
 \\ \none &  \none & \  
   \\ \none &  \none  & \ & \ 
      \\ \none &  \none  & \none & \ & \ & \ & \  }
\]
or some analogous sequence of contiguous boxes.

 Suppose $T$ is a semistandard shifted tableau. Then $T|_{[i,i]}$
is a disjoint union of ribbons, which we call the \emph{$i$-ribbons} of $T$.
Each entry in an $i$-ribbon  is  $i$ or $i'$, and all of these 
 are uniquely determined except for
the top left entry, for which there are two possibilities as in these examples:
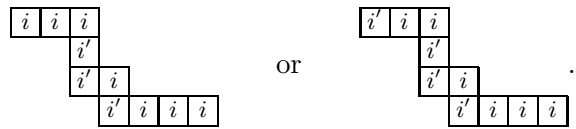
\[
\ytab{ i & i &  i
 \\ \none &  \none & i'
   \\ \none &  \none  & i' &i
      \\ \none &  \none  & \none & i' & i & i & i  }
\qquord
\ytab{ i' & i &  i
 \\ \none &  \none & i'
   \\ \none &  \none  & i' &i
      \\ \none &  \none  & \none & i' & i & i & i  }.
\]

Recall the definition of the \emph{shifted reading word} $ \shword(T)$ of 
 $T$  from \eqref{shword-eq}.
This definition extends to skew shifted tableaux with no changes.
 Fix a positive integer $i$ and
assume the domain of $T|_{[i,i+1]}$ has size $N$. Let $\alpha_1,\alpha_2,\dots,\alpha_N$
be the positions in this domain, ordered such that $\alpha_j$ is the position contributing the $j$th letter of $\shword(T|_{[i,i+1]})$.

\begin{definition}
Consider the word formed by replacing each $i$ in $\shword(T|_{[i,i+1]})$ by a right parenthesis ``)''
and each $i+1$ in $\shword(T|_{[i,i+1]})$ by a left parenthesis ``(''.
If $j$ and $k$ are the indices of a matching set of parentheses in this word then we say that $\alpha_j$ and $\alpha_k$ 
are paired. Remove all paired positions from  $(\alpha_1,\alpha_2,\dots,\alpha_N)$
and let 
$ \unpaired_i(T) $ 
denote the resulting subsequence.
\end{definition}

\begin{example} Suppose $i=4$ and $T|_{[4,5]}$ is the skew shifted tableau
\[ 
\ytab{ \none & \none  & 4 & 5 & 5 \\
\none & \none[\cdot] & \none[\cdot] & 4' & 4 & 5' & 5  \\ 
\none[\cdot] & \none[\cdot] & \none[\cdot] & \none[\cdot] & 4' & 4  & 4 & 5' & 5}.
\]
Then $\shword(T|_{[4,5]}) = 5 5 4 4 455 45 445$ and the corresponding ordering of the boxes in $T|_{[4,5]}$ is
\[ 
\ytab{ \none & \none  & 5 & 6 & 7 \\
\none & \none[\cdot] & \none[\cdot] & 4 & 8 & 2 & 9  \\ 
\none[\cdot] & \none[\cdot] & \none[\cdot] & \none[\cdot] & 3 & 10  & 11 & 1 & 12}.
\]
The  paired positions are $
(\alpha_2,\alpha_3), $ $ (\alpha_1,\alpha_4),$ $ (\alpha_6,\alpha_{11}),$
$(\alpha_7,\alpha_{8}),$ and
$  (\alpha_9,\alpha_{10})$,
  so \[\unpaired_4(T) = (\alpha_5,\alpha_{12}) = ((3,3), (1,9)).\]
  \end{example}

\subsection{Lowering operators}

The queer lowering operator $f_{\bar 1}$ for  $\STab_n(\lambda)$
from Theorem-Definition~\ref{hm-bijection-thm} 
has the following description. This appears as both \cite[Definition 4.4]{AssafOguz} and \cite[Lemma 3.2]{Hiroshima2018}.

\begin{proposition}[\cite{AssafOguz,Hiroshima2018}]
Let $T\in \STab_n(\lambda)$.
If no box of $T$ contains $1$
or some box of $T$ contains $2'$, then $f_{\overline 1}(T) = 0$.
Otherwise $f_{\overline 1}(T)$ is formed from $T$
by changing the rightmost $1$ in the first row of $T$ to 
be $2$ if the rightmost $1$ is on the diagonal,
or else $2'$.
\end{proposition}

\begin{example}
Thus $f_{\bar 1}\(\ytab{\none& 3 \\ 1 & 2' & 2}\) = f_{\bar 1}\(\ytab{\none &3 \\ 2 & 2}\)=0
$
while
\[
f_{\bar 1}\(\ytab{\none& 3 \\ 1 & 1 & 2}\) = \ytab{\none& 3 \\ 1 & 2' & 2}
\quand
f_{\bar 1}\(\ytab{\none& 3 \\ 1 & 2 & 2}\) = \ytab{\none& 3 \\ 2 & 2 & 2}.
\]
\end{example}

The lowering operators $f_i$ on $\STab_n(\lambda)$ for $i \in [n-1]$
are more complicated, and were first described in \cite[\S4]{HPS}.
The theorem below reproduces \cite[Definition 3.5]{AssafOguz}, which is equivalent to the
formulation in \cite{HPS} by \cite[Proposition 3.19]{AssafOguz}. 

Checking that the formula for $f_i$ in the next theorem gives a 
map  $\STab_n(\lambda) \to \STab_n(\lambda)\sqcup \{0\}$
is already nontrivial (see \cite[Theorem 3.8]{AssafOguz}).
Showing that $f_i$ commutes with $P_\HM$ in the sense required for part (b) of Theorem-Definition~\ref{hm-bijection-thm}
is even harder (see \cite[Theorem 4.3]{HPS}).

\begin{theorem}[\cite{AssafOguz,HPS}]
\label{shtab-f}
Let $i\in [n-1]$ and $T\in \STab_n(\lambda)$.
Consider the positions  $(x,y)$ in $\unpaired_i(T)$ with $ T_{xy}\in \{i',i\}$.
If there are no such positions then  $f_i(T) = 0$.
Otherwise, let $(x,y)$ be the last such position.
Then $f_i(T) \neq 0$ is formed from $T$ by the following procedure.
\ben
\item[(L1)] First assume $T_{xy} = i$. 

 \item[a.] If $T_{x,y+1} = i+1'$ then $f_i(T)$ is formed by changing $T_{xy}$ to $i+1'$ and $T_{x,y+1} $ to $i+1$:
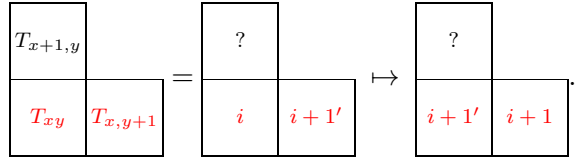
\[ {\scriptsize
\ytableausetup{boxsize = 1.0cm,aligntableaux=center}
\begin{ytableau} 
\color{black}T_{x+1,y} & \none \\
\color{red}T_{xy} & \color{red} T_{x,y+1}
\end{ytableau} = \begin{ytableau} 
\color{black} ? & \none\\
\color{red} i & \color{red} i+1'
\end{ytableau}
\ \mapsto\ 
 \begin{ytableau} 
\color{black} ? & \none\\
\color{red} i+1' & \color{red} i+1
\end{ytableau}}.
\]

\item[b.] If $T_{x,y+1} \neq i+1'$ and $ T_{x+1,y} \notin \{i+1',i+1\}$ then $f_i(T)$ is formed by changing $T_{xy}$ to $i+1$:
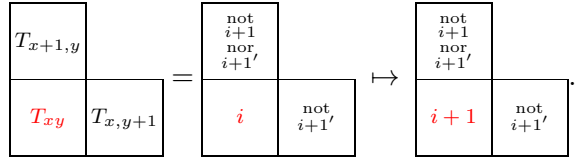
\[ {\scriptsize
\begin{ytableau} 
\color{black} T_{x+1,y} & \none \\
\color{red}T_{xy} & \color{black} T_{x,y+1}
\end{ytableau}  =
 \begin{ytableau} 
\color{black}\substack{\text{not}\\i+1 \\ \text{nor} \\ i+1'} & \none \\
\color{red}i & \color{black}\substack{\text{not}\\i+1'}
\end{ytableau}
\ \mapsto\ 
 \begin{ytableau} 
\color{black} \substack{\text{not}\\i+1 \\ \text{nor} \\ i+1'} & \none \\
\color{red}i+1 & \color{black}\substack{\text{not}\\i+1'}
\end{ytableau}}.\]
 
\item[c.] Suppose $T_{x,y+1} \neq i+1'$ and  $ T_{x+1,y} \in \{i+1',i+1\}$.
  Let $(\tilde x, \tilde y)$ 
be the position  farthest northwest in the
 $(i+1)$-ribbon containing $(x+1,y)$. 
If  $\tilde x \neq \tilde y$, then it holds that
$T_{\tilde x\tilde y} = i+1'$ and 
$f_i(T)$ is formed by
changing  $T_{xy}$ to $i+1'$ and $T_{\tilde x\tilde y}$ to $i+1$:
      \[ {\scriptsize
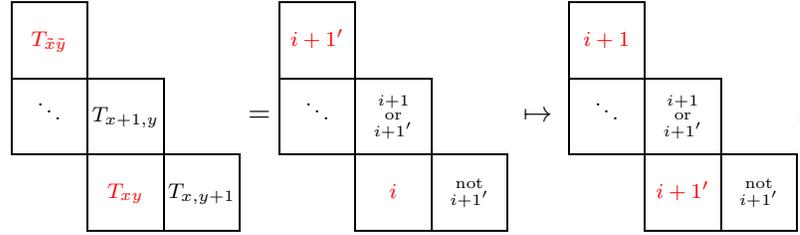
\begin{ytableau} 
\color{red}T_{\tilde x\tilde y}  &  \none & \none \\
\color{black}\ddots & \color{black} T_{x+1,y} & \none \\
 \none &\color{red} T_{xy} & \color{black} T_{x,y+1} 
\end{ytableau}  =
 \begin{ytableau} 
\color{red}i+1'  &  \none & \none \\
\color{black}\ddots &  \color{black}\substack{i+1 \\ \text{or} \\ i+1'} & \none \\
 \none & \color{red}i & \color{black}\substack{\text{not}\\ i+1'}
\end{ytableau} 
\ \mapsto\ 
 \begin{ytableau} 
\color{red}i+1  &  \none & \none \\
\color{black}\ddots & \color{black}\substack{i+1 \\ \text{or} \\ i+1'} & \none \\
 \none &\color{red} i+1' & \color{black}\substack{\text{not}\\ i+1'}
\end{ytableau} }.
\]
If  $\tilde x = \tilde y$,
then $f_i(T)$ is formed by
just changing $T_{xy}$ to $i+1'$.

\item[(L2)] Next assume $T_{xy} = i'$. 

 \item[a.] If $T_{x+1,y} = i$ then  $f_i(T)$ is formed  by changing $T_{xy}$ to $i$ and $T_{x+1,y} $ to $i+1'$:
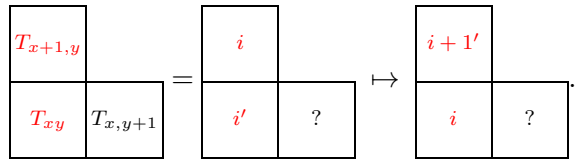
\[ {\scriptsize
\ytableausetup{boxsize = 1.0cm,aligntableaux=center}
\begin{ytableau} 
\color{red}T_{x+1,y} & \none \\
\color{red}T_{xy}  &\color{black} T_{x,y+1} 
\end{ytableau} = \begin{ytableau} 
\color{red}i & \none \\
\color{red}i' & \color{black}?
\end{ytableau}
\ \mapsto\ 
 \begin{ytableau} 
\color{red}i+1' & \none \\
\color{red}i & \color{black}?
\end{ytableau}}.
\]

 \item[b.] If $T_{x+1,y} \neq i$ and $T_{x,y+1} \notin \{i, i+1'\}$ then  $f_i(T)$ is formed by changing $T_{xy}$ to $i+1'$:
\[ 
{\scriptsize
\begin{ytableau} 
\color{black}T_{x+1,y} & \none \\
\color{red}T_{xy} &\color{black} T_{x,y+1}
\end{ytableau} = \begin{ytableau} 
\color{black}\substack{\text{not}\ i} & \none\\
\color{red}i' &\color{black} \substack{\text{not}\ i \\ \text{nor} \\ i+1'}
\end{ytableau}
\ \mapsto\ 
 \begin{ytableau} 
\color{black}\substack{\text{not}\ i}  & \none\\
\color{red}i+1' & \color{black}\substack{\text{not}\ i \\ \text{nor} \\ i+1'}
\end{ytableau}}.
\]

 \item[c.] Suppose $T_{x+1,y} \neq i$ and $T_{x,y+1} \in \{i, i+1'\}$.
 Let $(\tilde x, \tilde y)$ 
be the first position in the $i$-ribbon containing $(x,y)$ 
that is southeast of $(x,y)$ and has $T_{\tilde x \tilde y} = i$ and
 $T_{\tilde x, \tilde y+1} \notin \{i, i+1'\}$.
 Such a positions always exists,
 and $f_i(T)$ is formed by changing $T_{xy}$ to $i$ and $T_{\tilde x \tilde y}$ to $i+1'$:
 \[ {\scriptsize
\begin{ytableau} 
\color{black}T_{x+1,y} & \none \\
\color{red}T_{xy} & \color{black}T_{x,y+1}  \\
\none & \color{black}\ddots & \color{red}T_{\tilde x \tilde y}
\end{ytableau} =\begin{ytableau} 
\color{black}\substack{\text{not}\ i}  & \none \\
\color{red}i' & \color{black}\substack{i \\ \text{or} \\ i+1'}  \\
\none & \color{black}\ddots & \color{red}i
\end{ytableau}
\ \mapsto\ 
\begin{ytableau} 
\color{black}\substack{\text{not}\ i}  & \none \\
\color{red}i & \color{black}\substack{i \\ \text{or} \\ i+1'}  \\
\none & \color{black}\ddots & \color{red}i+1'
\end{ytableau}}.
\]

\een
\end{theorem}

\begin{remark}\label{correct-rmk}
The adjectives ``northeastern'' and ``southwestern'' 
in
cases L1(c) and L2(c) of \cite[Definition 3.5]{AssafOguz}
 should be ``northwestern'' and ``southeastern,'' respectively.
\end{remark}

\subsection{Raising operators}

The queer raising operator $e_{\bar 1}$ for  $\STab_n(\lambda)$
from Theorem-Definition~\ref{hm-bijection-thm} 
also has a relatively simple description. This appears as both \cite[Definition 4.5]{AssafOguz} and \cite[Lemma 3.1]{Hiroshima2018}.

\begin{proposition}[\cite{AssafOguz,Hiroshima2018}]
Let $T\in \STab_n(\lambda)$.
If the first entry in the first row of $T$ is equal to $2$ then $e_{\overline 1}(T)$ is formed by changing this entry to $1$.
If the first row of $T$ has a (necessarily unique) entry equal to $2'$, then $e_{\overline 1}(T)$ is formed by changing this entry to $1$.
Otherwise $e_{\overline 1}(T) = 0$.
\end{proposition}
\begin{example}
Thus $
e_{\bar 1}\(\ytab{\none &3 \\ 1 & 2 & 2}\) = e_{\bar 1}\(\ytab{\none &4 \\ 1 & 4' }\) = 0
$ while
\[
e_{\bar 1}\(\ytab{\none &3 \\ 2 & 2}\) = \ytab{\none &3 \\ 1 & 2}
\quand e_{\bar 1}\(\ytab{\none& 3 \\ 1 & 2' & 2}\) = \ytab{\none &3 \\ 1 & 1 & 2}.
\]
\end{example}

Remarks similar to above apply to the remaining raising operators for $\STab_n(\lambda)$.
These were first described in \cite[\S4]{HPS}.
The theorem below reproduces \cite[Definition 3.9]{AssafOguz}, which is equivalent
to the formulas in \cite{HPS} by \cite[Proposition 3.19]{AssafOguz}.

\begin{theorem}[\cite{AssafOguz,HPS}]\label{shtab-e}
Let $i\in [n-1]$ and $T\in \STab_n(\lambda)$.
Consider the positions  $(x,y)$ in $\unpaired_i(T)$ with $ T_{xy}  \in \{i+1',i+1\}$.
If there are no such positions then $e_i(T) = 0$.
Otherwise, let $(x,y)$ be the first such position.
Then $e_i(T) \neq 0$ is formed from $T$ by the following procedure.
\ben
\item[(R1)] First assume $T_{xy} = i+1$. 
\item[a.] If $T_{x,y-1} = i+1'$ then $e_i(T)$ is formed by changing $T_{xy}$ to $i+1'$ and $T_{x,y-1} $ to $i$:
\[ {\scriptsize
\ytableausetup{boxsize = 1.0cm,aligntableaux=center}
\begin{ytableau} 
\color{red}T_{x,y-1} &\color{red} T_{xy} \\
\none & T_{x-1,y}
\end{ytableau} = \begin{ytableau} 
\color{red}i+1' &\color{red} i+1\\
\none & ?
\end{ytableau}
\ \mapsto\ 
 \begin{ytableau} 
\color{red}i&\color{red} i+1'\\
\none & ?
\end{ytableau}}.
\] 

\item[b.] If $T_{x,y-1} \neq i+1'$ and $ T_{x-1,y} \notin\{i,i+1'\}$ then $e_i(T)$ is formed by changing $T_{xy}$ to $i$:
\[ {\scriptsize
\begin{ytableau} 
T_{x,y-1} & \color{red} T_{xy} \\
\none & T_{x-1,y}
\end{ytableau} = \begin{ytableau} 
\substack{\text{not} \\ i+1'} & \color{red} i+1\\
\none &  \substack{\text{not }i \\\text{nor} \\ i+1'} 
\end{ytableau}
\ \mapsto\ 
 \begin{ytableau} 
\substack{\text{not} \\ i+1'} & \color{red} i\\
\none &  \substack{\text{not }i \\\text{nor} \\ i+1'} 
\end{ytableau}}.
\]

\item[c.] Suppose $T_{x,y-1} \neq i+1'$ and $ T_{x-1,y} \in \{i,i+1'\}$.
Let $(\tilde x, \tilde y)$ 
be the first position  in the
 $(i+1)$-ribbon containing $(x,y)$
 that is southeast of $(x,y)$ with $T_{\tilde x \tilde y} = i+1'$ and
$T_{\tilde x-1, \tilde y} \notin \{i, i+1'\}$.
Such a position exists,
and $e_i(T)$ is formed from $T$ by 
changing $T_{xy}$ to $i+1'$ and $T_{\tilde x\tilde y}$ to $i$:
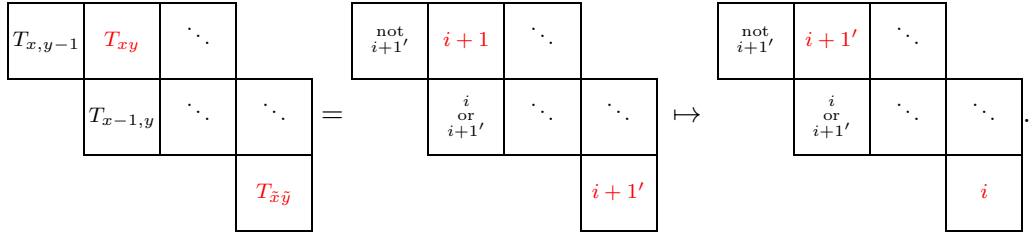
 \[ {\scriptsize
\begin{ytableau} 
T_{x,y-1} & \color{red} T_{xy} & \ddots\\
\none & T_{x-1,y}  & \ddots &\ddots \\
\none &\none  & \none & \color{red} T_{\tilde x \tilde y}
\end{ytableau} =\begin{ytableau} 
\substack{\text{not} \\ i+1'} & \color{red}i+1& \ddots \\
\none& \substack{i \\\text{or} \\ i+1'}  & \ddots &\ddots \\
\none&\none  & \none &\color{red} i+1'
\end{ytableau}
\ \mapsto\ 
\begin{ytableau} 
\substack{\text{not} \\ i+1'} &\color{red} i+1' & \ddots \\
\none & \substack{i \\\text{or} \\ i+1'}  & \ddots  & \ddots \\
\none & \none  & \none  &\color{red} i
\end{ytableau}}.
\]

\item[(R2)] Next assume $T_{xy} = i+1'$. 

 \item[a.] If $T_{x-1,y} = i$ then $e_i(T)$ is formed  by changing $T_{xy}$ to $i$ and $T_{x-1,y} $ to $i'$:
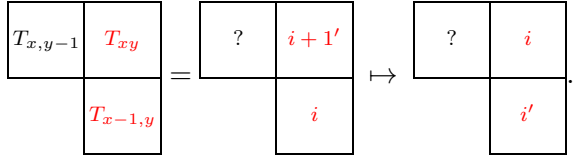
\[ {\scriptsize
\ytableausetup{boxsize = 1.0cm,aligntableaux=center}
\begin{ytableau} 
T_{x,y-1} &\color{red} T_{xy} \\
\none & \color{red}T_{x-1,y}
\end{ytableau} = \begin{ytableau} 
? &\color{red} i+1'\\
\none & \color{red}i
\end{ytableau}
\ \mapsto\ 
 \begin{ytableau} 
?&\color{red} i\\
\none & \color{red}i'
\end{ytableau}}.
\] 

 \item[b.] If $T_{x-1,y} \neq i$ and $T_{x,y-1} \notin \{i', i\}$ then $e_i(T)$ is formed by changing $T_{xy}$ to $i+1'$:
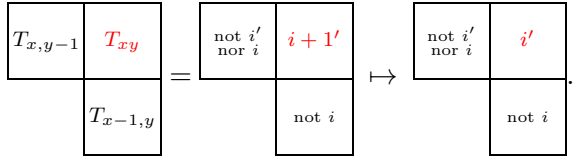
\[ {\scriptsize
\ytableausetup{boxsize = 1.0cm,aligntableaux=center}
\begin{ytableau} 
T_{x,y-1} &\color{red} T_{xy} \\
\none &T_{x-1,y}
\end{ytableau} = \begin{ytableau} 
\substack{\text{not }i' \\ \text{nor }i} &\color{red} i+1'\\
\none &  \substack{\text{not }i}
\end{ytableau}
\ \mapsto\ 
 \begin{ytableau} 
\substack{\text{not }i' \\ \text{nor }i} &\color{red} i'\\
\none & \substack{\text{not }i} 
\end{ytableau}}.
\] 

 \item[c.] Suppose $T_{x-1,y} \neq i$ and $T_{x,y-1} \in \{i', i\}$.
Let $(\tilde x, \tilde y)$ 
be the position that is farthest northwest in the
 $i$-ribbon containing $(x,y-1)$. 
If  $\tilde x \neq \tilde y$, then it holds that
$T_{\tilde x\tilde y}=i$
and $e_i(T)$ is formed by changing $T_{xy}$ to $i$
and $T_{\tilde x\tilde y}$ to $i'$:
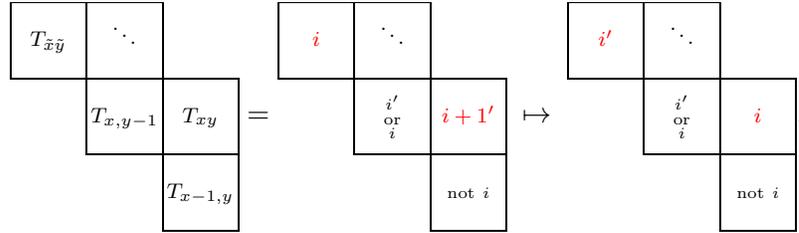
      \[ {\scriptsize
\begin{ytableau} 
T_{\tilde x\tilde y}  &  \ddots & \none \\
\none &  T_{x,y-1} & T_{xy} \\
 \none & \none & T_{x-1,y} 
\end{ytableau}  =
 \begin{ytableau} 
\color{red}i &  \ddots & \none \\
\none &   \substack{i' \\ \text{or} \\ i} & \color{red}i+1' \\
 \none & \none & \substack{\text{not} \ i}
\end{ytableau} 
\ \mapsto\ 
 \begin{ytableau} 
\color{red}i'  &  \ddots & \none \\
\none &  \substack{i' \\ \text{or} \\ i} & \color{red}i \\
 \none &\none &  \substack{\text{not} \ i}
\end{ytableau} }.
\]
If  $\tilde x = \tilde y$, then
 $e_i(T)$ is formed by just changing $T_{xy}$ to $i$.

\een
\end{theorem}

\begin{remark}
As in Remark~\ref{correct-rmk}, in cases R1(c) and R2(c) of \cite[Definition 3.9]{AssafOguz}
the directions ``southwest'' and ``northeast'' 
should be ``southeast'' and ``northwest,'' respectively.
In addition, the phrase ``changes $x$ to $\bar i$'' in case R2(c) of \cite[Definition 3.9]{AssafOguz}
should be ``changes $x$ to $i$.'' Finally, the picture illustrating R2(c)
in \cite[Figure 18]{AssafOguz}
has an extra box on the left; this picture should be 
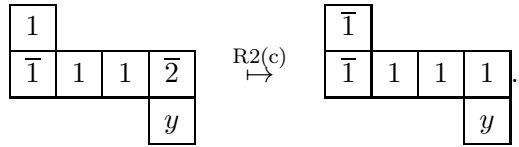
\[
\ytableausetup{boxsize = 0.6cm,aligntableaux=center}
 \begin{ytableau} 
1 \\ 
\overline 1 & 1 & 1 & \overline 2 \\
\none & \none & \none & y
\end{ytableau} \quad\overset{\text{R2(c)}}{\mapsto}\quad 
 \begin{ytableau} 
\overline 1 \\ 
\overline 1 & 1 & 1 & 1 \\
\none & \none & \none & y
\end{ytableau}.\]
The marked numbers $\overline 1$ and $\overline 2$ in these tableaux are what we would write as $1'$ and $2'$.
\end{remark}

\subsection{Finishing the proof of Theorem~\ref{dual-thm2}}

As an application of the preceding discussion, we  can now give a self-contained derivation of
the two lemmas
cited at the end of the proof of Theorem~\ref{dual-thm2}.
Recall that if $i \in \ZZ$ then $i'+ 1 = (i+ 1)'=i+1'$
and $i'- 1 = (i- 1)'\neq i-1'$
 since $i' := i - \frac{1}{2}$ and $1':=\frac{1}{2}$.

\begin{lemma}\label{dual-equiv-lem0}
Let $T$ be a standard shifted tableau with $n$ boxes.
For each $j \in [n]$, let $\square_j$ be the unique box of $T$ containing $j$ or $j'$.
Suppose $i \in [n-1]$.
\ben
\item[(a)] If $i \in \Des(T)$ then $e_i(T) = f_i(T) = 0$.
\item[(b)] If $i \notin \Des(T)$ then $f_i(T)$ is formed from $T$ by adding $1$ to the entry in box $\square_i$.
\item[(c)] If $i \notin \Des(T)$ then $e_i(T)$ is formed from $T$ by subtracting $1$ from the entry in box $\square_{i+1}$.
\een
\end{lemma}

\begin{proof}
If $i \in \Des(T)$ then $\unpaired_i(T)$ is empty, which implies part (a).
If $i \notin\Des(T)$ then $\unpaired_i(T) = (\square_i,\square_{i+1})$
and 
parts (b) and (c) follow from Theorems~\ref{shtab-f} and \ref{shtab-e}.
\end{proof}

\begin{lemma}\label{dual-equiv-lem}
Let $T$ be a standard shifted tableau with $n$ boxes.
Suppose $i \in [n-2]$.
\ben
\item[(a)] If $\Des(T) \cap \{i,i+1\} = \{i\}$ then $f_if_{i+1}e_ie_{i+1}(T) = \fks_i(T)$.
\item[(b)] If $\Des(T) \cap \{i,i+1\} = \{i+1\}$ then $f_{i+1}f_{i}e_{i+1}e_{i}(T) =  \fks_i(T)$.
\een
\end{lemma}

\begin{proof}
For each $j \in [n]$, let $\square_j$ be the unique box of $T$ containing $j$ or $j'$.
If $f_if_{i+1}e_ie_{i+1}(T)$ is nonzero 
then $f_{i+1}f_{i}e_{i+1}e_{i}f_if_{i+1}e_ie_{i+1}(T) = T$.
Thus parts (a) and (b) are equivalent  since 
 $\fks_i$ acts as an involution
and one can check that  
$\Des(T) \cap \{i,i+1\} = \{i\}$ if and only if $\Des(\fks_i(T))  \cap \{i,i+1\} = \{i+1\}$.
 We can therefore assume $\Des(T) \cap \{i,i+1\} = \{i\}$ and just prove part (a).
Then $\shword(T)$ contains either $i+1,i,i+2$ or $i+1,i+2,i$ as a subword.
We consider these cases in turn:
\ben
\item[(1)] Suppose $i+1,i,i+2$ is a subword of $\shword(T)$.
Then $\fks_i(T) = s_{i+1} \star T$
and $\square_{i+2}$ must come after both $\square_{i+1}$ and $\square_i$ 
in the order corresponding to the shifted reading word of $T$.
By the previous lemma $e_{i+1}(T)$ is formed from $T$ by subtracting one from the entry in position $\square_{i+2}$,
which is subsequently the only position in $\unpaired_i(e_{i+1}(T))$.

If $\square_{i+1}$ and $\square_{i+2}$ are in different rows and different columns, 
then it is easy to see from Theorems~\ref{shtab-f} and \ref{shtab-e}
that applying $e_i$ to $e_{i+1}(T)$ subtracts one from the entry in position $\square_{i+2}$,
while applying $f_{i+1}$ to $e_ie_{i+1}(T)$ adds one to $\square_{i+1}$,
while applying $f_{i}$ to $f_{i+1}e_ie_{i+1}(T)$ adds one to $\square_{i+2}$, as illustrated by the 
following picture where $i=3$:
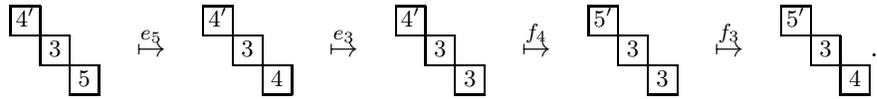
\[
\ytab{ 
4' \\ 
\none & 3 \\ 
\none & \none & 5}
\quad\overset{e_{5}}{\mapsto}\quad
\ytab{ 
4' \\ 
\none & 3 \\ 
\none & \none & 4}
\quad\overset{e_{3}}{\mapsto}\quad
\ytab{ 
4' \\ 
\none & 3 \\ 
\none & \none & 3}
\quad\overset{f_{4}}{\mapsto}\quad
\ytab{ 
5' \\ 
\none & 3 \\ 
\none & \none & 3}
\quad\overset{f_{3}}{\mapsto}\quad
\ytab{ 
5' \\ 
\none & 3 \\ 
\none & \none & 4}.
\]
 In this case the aggregate effect of applying $f_{i}f_{i+1}e_ie_{i+1}$ to $T$ is to subtract
 one from $\square_{i+2}$ and add one to $\square_{i+1}$, which gives $s_{i+1} \star T  = \fks_i(T)$
 as desired.
 
Assume instead that $\square_{i+1}$ and $\square_{i+2}$ are in the same row. 
Then it is only possible for $i+1,i,i+2$ to be a subword of $\shword(T)$
if the entry in $\square_{i+1}$ is primed, the entry in $\square_{i+2}$ is unprimed,
and $\square_i$ belongs to the region
strictly left of and weakly above $\square_{i+1}$.
It follows from
 Theorems~\ref{shtab-f} and \ref{shtab-e}
that the effect of applying $e_{i+1}$, $e_i$, $f_{i+1}$, and $f_i$ successively to $T$ 
is represented by the following picture, and ultimately gives $s_{i+1} \star T = \fks_i(T)$ as needed:
\[
\ytableausetup{boxsize = 0.8cm,aligntableaux=center}
\ba T|_{[i,i+2]} \ =\ \begin{ytableau} 
 \substack{i \\ \text{or} \\ i'}  & \substack{i+1'} & \substack{i+2}  
\end{ytableau} 
 \quad\overset{e_{i+1}}{\mapsto}&\quad
 \begin{ytableau} 
 \substack{i \\ \text{or} \\ i'}    &  \substack{i+1'}  & \substack{i+1} 
\end{ytableau} 
 \quad\overset{e_{i}}{\mapsto}\quad
 \begin{ytableau} 
 \substack{i \\ \text{or} \\ i'}   &  \substack{i}  & \substack{i+1'}  
\end{ytableau} 
\\  \quad\overset{f_{i+1}}{\mapsto}&\quad
 \begin{ytableau} 
 \substack{i \\ \text{or} \\ i'}  &  \substack{i}  & \substack{i+2'} 
\end{ytableau} 
 \quad\overset{f_{i}}{\mapsto}\quad
 \begin{ytableau} 
 \substack{i \\ \text{or} \\ i'}  & \substack{i+1} & \substack{i+2'}  
\end{ytableau}  \ =\ \fks_i(T)|_{[i,i+2]}.
\ea
 \]
 Here, box $\square_i$ need not be in the same row as $\square_{i+1}$ and $\square_{i+2}$;
 this position may be anywhere weakly northwest of the one shown.
Also,
 the application of $e_i$ invokes case R1(a) of Theorem~\ref{shtab-e},
 which interchanges which of $\square_{i+1}$ or $\square_{i+2}$ contains a primed entry.
 
Finally assume that $\square_{i+1}$ and $\square_{i+2}$ are in the same column.
Then it is only possible for $i+1,i,i+2$ to be a subword of $\shword(T)$
if the entry in $\square_{i+1}$ is primed, the entry in $\square_{i+2}$ is unprimed,
and $\square_i$ belongs to the region
strictly left of and strictly above $\square_{i+1}$.
It follows from
 Theorems~\ref{shtab-f} and \ref{shtab-e}
that the effect of applying $e_{i+1}$, $e_i$, $f_{i+1}$, and $f_i$ successively to $T$ 
is represented by the following picture, and again gives $s_{i+1} \star T = \fks_i(T)$:
\[
\ytableausetup{boxsize = 0.8cm,aligntableaux=center}
\ba T|_{[i,i+2]} \ =\ \begin{ytableau} 
 \substack{i \\ \text{or} \\ i'}  & \substack{i+2} \\
\none & \substack{i+1'}  
\end{ytableau} 
 \quad\overset{e_{i+1}}{\mapsto}&\quad
 \begin{ytableau} 
 \substack{i \\ \text{or} \\ i'}   & \substack{i+1} \\
\none & \substack{i+1'}  
\end{ytableau} 
 \quad\overset{e_{i}}{\mapsto}\quad
 \begin{ytableau} 
 \substack{i \\ \text{or} \\ i'}   & \substack{i+1'} \\
\none & \substack{i}  
\end{ytableau} 
\\  \quad\overset{f_{i+1}}{\mapsto}&\quad
 \begin{ytableau} 
 \substack{i \\ \text{or} \\ i'}   & \substack{i+2'} \\
\none & \substack{i}  
\end{ytableau} 
 \quad\overset{f_{i}}{\mapsto}\quad
 \begin{ytableau} 
 \substack{i \\ \text{or} \\ i'}   & \substack{i+2'} \\
\none & \substack{i+1}  
\end{ytableau}  \ =\ \fks_i(T)|_{[i,i+2]}.
\ea
 \]
The box $\square_i$ again might occur anywhere weakly northwest
  of the position shown. The application of $e_i$ now invokes case R1(c) of Theorem~\ref{shtab-e}.
 
 \item[(2)] Suppose $i+1,i+2,i$ is a subword of $\shword(T)$. Our arguments are similar.
 Now we have $\fks_i(T) = s_i\star T$
and $\square_{i}$ must come after both $\square_{i+1}$ and $\square_{i+2}$ 
in the order corresponding to the shifted reading word.
By the previous lemma $e_{i+1}(T)$ is formed from $T$ by subtracting one from the entry in position $\square_{i+2}$,
which leaves $\square_{i+1}$ as the only position in $\unpaired_i(e_{i+1}(T))$.

Similar to the previous case, if $\square_{i}$ and $\square_{i+1}$ are in different rows and different columns
then it follows from Theorems~\ref{shtab-f} and \ref{shtab-e}
that applying $e_i$ to $e_{i+1}(T)$ subtracts one from the entry in position $\square_{i+1}$,
while applying $f_{i+1}$ to $e_ie_{i+1}(T)$ adds one to $\square_{i+2}$,
while applying $f_{i}$ to $f_{i+1}e_ie_{i+1}(T)$ adds one to $\square_{i}$, as illustrated by the 
following picture where $i=3$:
\[
\ytab{ 
3 \\ 
\none & 5' \\ 
\none & \none & 4'}
\quad\overset{e_{5}}{\mapsto}\quad
\ytab{ 
3 \\ 
\none & 4' \\ 
\none & \none & 4'}
\quad\overset{e_{3}}{\mapsto}\quad
\ytab{ 
3 \\ 
\none & 4' \\ 
\none & \none & 3'}
\quad\overset{f_{4}}{\mapsto}\quad
\ytab{ 
3 \\ 
\none & 5' \\ 
\none & \none & 3'}
\quad\overset{f_{3}}{\mapsto}\quad
\ytab{ 
4 \\ 
\none & 5' \\ 
\none & \none & 3'}.
\]
Thus we have $f_if_{i+1}e_ie_{i+1}(T) = s_i \star T = \fks_i(T)$ as desired.

Assume instead that $\square_{i}$ and $\square_{i+1}$ are in the same row.
Then it possible for $\square_i$ and $\square_{i+2}$ to be in consecutive diagonal positions
if the entry in $\square_{i+1}$ is primed, in which case 
the 
effect of applying $e_{i+1}$, $e_i$, $f_{i+1}$, and $f_i$ successively to $T$ 
is represented by the following picture, and gives $s_{i} \star T = \fks_i(T)$ as needed:
\[
\ytableausetup{boxsize = 0.8cm,aligntableaux=center}
\ba T|_{[i,i+2]} \ =\ \begin{ytableau} 
\none & \substack{i+2}   \\
\substack{i} & \substack{i+1'}
\end{ytableau} 
 \quad\overset{e_{i+1}}{\mapsto}&\quad
 \begin{ytableau} 
\none & \substack{i+1}   \\
\substack{i} & \substack{i+1'}
\end{ytableau} 
 \quad\overset{e_{i}}{\mapsto}\quad
 \begin{ytableau} 
\none & \substack{i+1}   \\
\substack{i} & \substack{i}
\end{ytableau} 
\\  \quad\overset{f_{i+1}}{\mapsto}&\quad
 \begin{ytableau} 
\none & \substack{i+2}   \\
\substack{i} & \substack{i}
\end{ytableau} 
 \quad\overset{f_{i}}{\mapsto}\quad
 \begin{ytableau} 
\none & \substack{i+2}   \\
\substack{i} & \substack{i+1}
\end{ytableau}  \ =\ \fks_i(T)|_{[i,i+2]}.
\ea
 \]
 Here the application of $e_i$ invokes case R2(c) of Theorem~\ref{shtab-e}.
If $\square_i$ and $\square_{i+1}$ are in the row but the preceding situation does not occur, then
it is only possible for $i+1,i+2,i$ to be a subword of $\shword(T)$
if the entry in $\square_{i}$ is unprimed, the entry in $\square_{i+1}$ is primed,
and $\square_{i+2}$ belongs to the region
strictly above and weakly left of $\square_{i}$.
In this event
 Theorems~\ref{shtab-f} and \ref{shtab-e}
imply that the effect of applying $e_{i+1}$, $e_i$, $f_{i+1}$, and $f_i$ successively to $T$ 
is represented by the following picture, and  gives $s_{i} \star T = \fks_i(T)$:
\[
\ytableausetup{boxsize = 0.8cm,aligntableaux=center}
\ba T|_{[i,i+2]} \ =\ \begin{ytableau} 
 \substack{i+2 \\ \text{or} \\ i+2'}  \\ \substack{i} & \substack{i+1'}  
\end{ytableau} 
 \quad\overset{e_{i+1}}{\mapsto}&\quad
 \begin{ytableau} 
 \substack{i+1 \\ \text{or} \\ i+1'}  \\ \substack{i} & \substack{i+1'}  
\end{ytableau} 
 \quad\overset{e_{i}}{\mapsto}\quad
 \begin{ytableau} 
 \substack{i+1 \\ \text{or} \\ i+1'}  \\ \substack{i'} & \substack{i}  
\end{ytableau} 
\\  \quad\overset{f_{i+1}}{\mapsto}&\quad
 \begin{ytableau} 
 \substack{i+2 \\ \text{or} \\ i+2'}  \\ \substack{i'} & \substack{i}  
\end{ytableau} 
 \quad\overset{f_{i}}{\mapsto}\quad
 \begin{ytableau} 
 \substack{i+2 \\ \text{or} \\ i+2'}  \\ \substack{i'} & \substack{i+1}  
\end{ytableau}  \ =\ \fks_i(T)|_{[i,i+2]}.
\ea
 \]
The application of $e_i$ here again invokes case R2(c) of Theorem~\ref{shtab-e},
and the actual location of $\square_{i+2}$ may be anywhere weakly northwest of the box shown.

Finally assume that $\square_{i}$ and $\square_{i+1}$ are in the same column.
Then it is only possible for $i+1,i+2,i$ to be a subword of $\shword(T)$
if the entry in $\square_{i}$ is unprimed, the entry in $\square_{i+1}$ is primed,
and $\square_{i+2}$ belongs to the region
strictly above and weakly left of $\square_{i+1}$.
It follows from
 Theorems~\ref{shtab-f} and \ref{shtab-e}
that the effect of applying $e_{i+1}$, $e_i$, $f_{i+1}$, and $f_i$ successively to $T$ 
is represented by the following picture, and again gives $s_{i} \star T = \fks_i(T)$:
\[
\ytableausetup{boxsize = 0.8cm,aligntableaux=center}
\ba T|_{[i,i+2]} \ =\ \begin{ytableau} 
 \substack{i+2 \\ \text{or} \\ i+2'}  \\ \substack{i+1'} \\  \substack{i}  
\end{ytableau} 
 \quad\overset{e_{i+1}}{\mapsto}&\quad
 \begin{ytableau} 
 \substack{i+1 \\ \text{or} \\ i+1'}  \\ \substack{i+1'} \\  \substack{i}  
\end{ytableau} 
 \quad\overset{e_{i}}{\mapsto}\quad
 \begin{ytableau} 
 \substack{i+1 \\ \text{or} \\ i+1'}  \\ \substack{i} \\  \substack{i'}  
 \end{ytableau} 
\\  \quad\overset{f_{i+1}}{\mapsto}&\quad
 \begin{ytableau} 
 \substack{i+2 \\ \text{or} \\ i+2'}  \\ \substack{i} \\  \substack{i'}  
\end{ytableau} 
 \quad\overset{f_{i}}{\mapsto}\quad
 \begin{ytableau} 
 \substack{i+2 \\ \text{or} \\ i+2'}  \\ \substack{i+1} \\  \substack{i'}   
\end{ytableau}  \ =\ \fks_i(T)|_{[i,i+2]}.
\ea
 \]
  Here, the application of $e_i$ invokes case R2(a) of Theorem~\ref{shtab-e},
  and once again the actual location of $\square_{i+2}$ may be anywhere weakly northwest of the box shown.
 
\een
This completes the proof of part (a), which suffices to prove the lemma.
\end{proof}


\end{document}